\documentclass{amsart}
\usepackage{amsmath,amscd,amssymb}
\usepackage{amsfonts}
\usepackage[all]{xy}
\usepackage{enumerate}
\usepackage{color}
\numberwithin{equation}{subsection}
\theoremstyle{plain}
\newtheorem{thm}[subsection]{Theorem}
\newtheorem{prop}[subsection]{Proposition}
\newtheorem{lemma}[subsection]{Lemma}
\newtheorem{cor}[subsection]{Corollary}

\theoremstyle{definition}
\newtheorem{defn}[subsection]{Definition}
\newtheorem{notn}[subsection]{Notation}
\newtheorem{cont}[subsection]{Contents}
\newtheorem{ackn}[subsection]{Acknowledgement}
\theoremstyle{remark}
\newtheorem{rem}[subsection]{Remark}

\newtheorem{example}[subsection]{Example}
\newtheorem{examples}[subsection]{Examples}
\newcommand{\dal}{{\mathchoice{\mbox{$\langle\hspace{-0.15em}\langle$}}
                              {\mbox{$\langle\hspace{-0.15em}\langle$}}
                              {\mbox{\scriptsize$\langle\hspace{-0.15em}\langle$}}
                              {\mbox{\tiny$\langle\hspace{-0.15em}\langle$}}}}
\newcommand{\dar}{{\mathchoice{\mbox{$\rangle\hspace{-0.15em}\rangle$}}
                              {\mbox{$\rangle\hspace{-0.15em}\rangle$}}
                              {\mbox{\scriptsize$\rangle\hspace{-0.15em}\rangle$}}
                              {\mbox{\tiny$\rangle\hspace{-0.15em}\rangle$}}}}
\def\noqed{\renewcommand{\qedsymbol}{}}
\begin{document}
\title{The $p$-adic monodromy group of abelian varieties over global function fields of characteristic $p$}
\author{Ambrus P\'al}
\date{December 11, 2015.}
\address{Department of Mathematics, 180 Queen's Gate, Imperial College, London, SW7 2AZ, United Kingdom}
\email{a.pal@imperial.ac.uk}
\begin{abstract} We prove an analogue of the Tate isogeny conjecture and the semi-simplicity conjecture for overconvergent crystalline Dieudonn\'e modules of abelian varieties defined over global function fields of characteristic $p$. As a corollary we deduce that monodromy groups of such overconvergent crystalline Dieudonn\'e modules are reductive, and after a finite base change of coefficients their connected components  are the same as the connected components of monodromy groups of Galois representations on the corresponding $l$-adic Tate modules, for $l$ different from $p$. We also show such a result for general compatible systems incorporating overconvergent $F$-isocrystals, conditional on a result of Abe.
\end{abstract}
\maketitle
\pagestyle{myheadings}
\markboth{Ambrus P\'al}{The $p$-adic monodromy group of abelian varieties}

\section{Introduction}

Let $U$ be a geometrically connected smooth quasi-projective curve defined over the finite field $\mathbb F_q$ of characteristic $p$. For every perfect field $k$ of characteristic $p$ let $\mathbb W(k)$ denote the ring of Witt vectors of $k$ of infinite length. Let $\mathbb Z_q$ and $\mathbb Q_q$ denote $\mathbb W(\mathbb F_q)$ and its the fraction field, respectively. For every abelian scheme $A$ over $U$ let $D^{\dagger}(A)$ denote the overconvergent crystalline Dieudonn\'e module of $A$ over $U$ (for a construction see \cite{KT}, sections 4.3--4.8). It is a $\mathbb Q_q$-linear $F$-isocrystal equipped with the $p$-power Frobenius. The first result of this paper is an easy consequence of de Jong's theorem (\cite{J1}):
\begin{thm}\label{isogeny} Let $A$ and $B$ be two abelian schemes over $U$. Then the map:
$$\textrm{\rm Hom}(A,B)\otimes\mathbb Q_p
\longrightarrow^{\!\!\!\!\!\!\!\!\alpha}\ \ 
\textrm{\rm Hom}(D^{\dagger}(A),D^{\dagger}(B))$$
induced by the functoriality of overconvergent Dieudonn\'e modules is an isomorphism.
\end{thm}
We will also give another proof of this claim by deriving it directly from Zarhin's celebrated theorem on the endomorphism of abelian varieties over function fields (see \cite{Za1} and \cite{Za2})) using the formidable machinery of rigid cohomology (see Chapters \ref{cycle-class} and \ref{tate} for the details). This line of attack necessitates the development of a suitable formalism of $p$-adic cycle class maps and Tate's conjecture over function fields as an additional bonus, and it also gives a new, although not entirely independent proof of de Jong's theorem mentioned above. Our next result is the natural pair of the result above, analogous to the semi-simplicity of the action of the absolute Galois group on the Tate module of abelian varieties, also due to Zarhin in the function field case.
\begin{thm}\label{semisimple} Let $A$ be an abelian scheme over $U$. Then the overconvergent $F$-isocrystal $D^{\dagger}(A)$ is semi-simple.
\end{thm}
This claim is an easy consequence of an analogue of Deligne's semi-simplicity result for pure lisse sheaves (see Theorem 4.3.1 of \cite{AC}) and the corresponding result for abelian varieties over finite fields. However we will also give a proof using $p$-divisible groups closely following the methods of de Jong, Faltings and Zarhin, mainly because this argument also leads to two interesting auxiliary results; one is a useful condition for sub $p$-divisible groups to have semi-stable reduction (see Theorem \ref{overconvergent-condition}), while the second is that the Faltings height of the quotients of a semi-stable abelian variety over a global function field by the $p$-primary torsion subgroup-schemes of a semi-stable sub $p$-divisible group is constant (see Theorem \ref{height}).

The second theorem can be reformulated in terms of the monodromy group of $D^{\dagger}(A)$ whose construction we recall next. Let $X$ be any connected $\mathbb F_q$-scheme and let $F$ be the $p$-power Frobenius on $X$. Let $\textrm{$F$-Isoc}^{\dagger}(X/\mathbb Q_q)$ denote the $\mathbb Q_p$-linear rigid tensor category of $\mathbb Q_q$-linear overconvergent $F$-isocrystals on $X$. For every object $\mathcal F$ of $\textrm{$F$-Isoc}^{\dagger}(X/\mathbb Q_q)$ let $\dal\mathcal F\dar$ denote the full rigid abelian tensor subcategory of $\textrm{$F$-Isoc}^{\dagger}(X/\mathbb Q_{q})$ generated by $\mathcal F$. For every $x:\textrm{Spec}(\mathbb F_{q^n})\to X$ the pull-back of an $F$-isocrystal to $x$ supplies a functor from the category $\textrm{$F$-Isoc}_{\mathbb Q_{q}}^{\dagger}(X)$ into the category $\textrm{$F$-Isoc}^{\dagger}(\textrm{Spec}(\mathbb F_{q^n})/\mathbb Q_{q})$. By forgetting the Frobenius structure we get a fibre functor $\omega_x$ into the category of $\mathbb Q_{q^n}$-linear vector spaces, which makes $\textrm{$F$-Isoc}^{\dagger}(X/\mathbb Q_{q})$ into a Tannakian category (see 2.2 of \cite{Cr} on page 440 for details). For every object $\mathcal F$ of
$\textrm{$F$-Isoc}^{\dagger}(X/\mathbb Q_q)$ let $\textrm{Gr}(\mathcal F,x)=\textrm{Aut}^\otimes(\omega_x|\dal\mathcal F\dar)$ denote the monodromy group of $\mathcal F$ with respect to the fibre functor $\omega_x$ (see \cite{DM}, especially Proposition 3.11). It is a linear algebraic group over $\mathbb Q_{q^n}$. By the Tannakian formalism Theorem \ref{semisimple} has the following immediate 
\begin{cor}\label{reductive} Let $A$ be an abelian scheme over $U$. For every $x\in U(\mathbb F_{q^n})$ the group $\textrm{\rm Gr}(D^{\dagger}(A),x)$ is reductive.\qed
\end{cor}
For every field $L$ let $\overline L$ denote the separable closure of $L$. Let $L$ denote the function field of $U$, and for every $A$ as above let $A_L$ denote the base change of $A$ to $L$. For every prime $l$ different from $p$ let $T_l(A)$ denote the $l$-adic Tate module of $A_L$ and let $V_l(A)=T_l(A)\otimes_{\mathbb Z_l}\mathbb Q_l$. The absolute Galois group Gal$(\overline L/L)$ acts continuously on $V_l(A)$; let $\rho_l$ denote the corresponding homomorphism Gal$(\overline L/L)\rightarrow\textrm{Aut}(V_l(A))$. Fix a point $x\in U(\mathbb F_{q^n})$ now and let $V_p(A)$ denote the vector space $\omega_x(D^{\dagger}(A))$. We have a representation
$\rho_p:\textrm{\rm Gr}(D^{\dagger}(A),x)\rightarrow\textrm{Aut}(V_p(A))$ by definition. For every prime number $l$ let $G_l$ denote the Zariski closure of the image of $\rho_l$ in $\textrm{Aut}(V_l(A))$. For every linear algebraic group $G$ defined over a field let $G^o$ denote its identity component. For every $l$ as above let $\rho_l^{\textrm{alg}}$ denote the representation of $G_l^o$ on $V_l(A)$. For every global field $K$ let $|K|$ denote the set of places of $K$ and for every $\lambda\in|K|$ let $K_{\lambda}$ denote the completion of $K$ with respect to $\lambda$. For every prime number $l$ let $r(l)$ be $l$, if $l$ is different from $p$, and $q^n$, otherwise. Our last result is the following extension of Chin's main theorem in \cite{Ch} (stated here for compatible systems coming from abelian varieties only):
\begin{thm}\label{strong-independence} There exists a number field $K$ such that for every place $\lambda\in|K|$ above $p$ the field $K_{\lambda}$ contains $\mathbb Q_{q^n}$, moreover there exist a connected split semi-simple algebraic group $\mathcal G$ over $K$ and a $K$-linear vector space $V$ equipped with a $K$-linear representation $\rho$ of $\mathcal G$ such that for every prime number $l$ and for every $\lambda\in|K|$ lying over $l$ the triples:
$$(\mathcal G\otimes_K K_{\lambda},V\otimes_KK_{\lambda},
\rho\otimes_KK_{\lambda})$$
and
$$(G_l^o\otimes_{\mathbb Q_{r(l)}}K_{\lambda},V_l(A)\otimes_{\mathbb Q_{r(l)}}K_{\lambda},\rho_l^{\textrm{\rm alg}}\otimes_{\mathbb Q_{r(l)}}K_{\lambda})$$
are isomorphic.
\end{thm}
The novelty of the theorem is that it includes the $p$-adic monodromy group and its representation, too, although it is not the first result to do so; in \cite{Cr1} Crew studied the $p$-adic monodromy of generic families of abelian varieties, and proved a form of the result above in this special case. I like to think of this result as a positive characteristic version of the Manin--Mumford conjecture, since it connects the image of the absolute Galois group with the Tannakian fundamental group of the cohomology of the abelian variety with respect to a cohomology theory defined with the help of differential forms. However its proof only requires the extension of the arguments of Larsen--Pink in the proof of one of the main results of \cite{LP2}, a precursor to this type of results, to rigid cohomology, because of Chin's theorem in \cite{Ch} quoted above. In this argument Theorems \ref{isogeny} and \ref{semisimple} play an important role. In the course of this proof we also show other independence results; in particular we show that the group of connected components are independent of $l$ (Proposition \ref{connected_components}), extending a classical result of Serre, and we also show an extension of Chin's main theorem for general $E$-compatible systems incorporating overconvergent $F$-isocrystals (Theorem \ref{chin_independence}), conditional on a result of Abe (see \cite{Abe}).
\begin{cont} In the next section we show Theorem \ref{overconvergent-condition} mentioned above, which roughly says that a sub $p$-divisible group of a $p$-divisible group defined over a Laurent series field, which is semi-stable in the sense of de Jong, is semi-stable if and only if its crystalline Dieudonn\'e module is overconvergent. Not surprisingly the proof mostly concerns $(\sigma,\nabla)$-modules over various rings of $p$-adic analytic functions, and heavily relies on Kedlaya's work on this subject. Whenever I could find the claim I needed in his papers, I gave a direct reference, but otherwise the arguments are strongly influenced by his proofs. We fix some notation for oveconvergent isocrystals and $F$-isocrystals, then we prove some useful semi-simplicity criteria for oveconvergent $F$-isocrystals in the third section. In the fourth section we define the analogue of the geometric monodromy group in the $p$-adic setting and prove that it is reductive and its identity component is the derived group of the identity component of the full monodromy group for pure sheaves. We also show short exact sequences relating various monodromy groups. In the fifth section we give a purely $p$-adic proof of Theorems \ref{isogeny} and \ref{semisimple}, combining methods of de Jong and Faltings. The crucial result is Theorem \ref{height} mentioned above, from which Theorem \ref{semisimple} follows in a rather standard way. We introduce a variant of the cycle class map into rigid cohomology for varieties defined over function fields in the sixth section, and use it to give another proof of Theorems \ref{isogeny} and \ref{semisimple}, deducing them from Zarhin's results, in the seventh section. In the last section, following Larsen and Pink, prove our independence results, including Theorem \ref{strong-independence}.
\end{cont}
\begin{ackn} This paper started as a complement to my joint project \cite{HP} with Urs Hartl. I wish to thank him, Chris Lazda and Richard Pink for some useful discussions related to the contents of this article. The author was partially supported by the EPSRC grant P36794.
\end{ackn}

\section{$p$-divisible groups with semi-stable reduction}

\begin{defn} Let $k$ a perfect field of characteristic $p>0$ and let $\mathcal O=\mathbb W(k)$ denote the ring of Witt vectors over $k$.  Let $v_p$ denote the valuation on $\mathcal O$ normalised so that $v_p(p) = 1$. For $x \in \mathcal O$, let $\overline{x}$ denote its reduction in $k$. Let $\Gamma$ denote the ring of bidirectional power series:
$$\Gamma=\big\{\sum_{i \in \mathbb Z} x_i u^i| x_i \in \mathcal O,\ \ 
\lim_{i \to -\infty}v_p(x_i)=\infty\big\}.$$
Then $\Gamma$ is a complete discrete valuation ring whose residue field
we could identify with $k((t))$ by identifying the reduction of $\sum x_i u^i$
with $\sum \overline{x}_it^i$. Let $\Gamma_+$ and $\Gamma^{\dagger}$ denote the subrings:
$$\Gamma_+=\big\{\sum_{i \in \mathbb N} x_i u^i|
x_i \in \mathcal O\big\}
\subset\Gamma,$$
$$\Gamma^{\dagger}=\big\{\sum_{i \in \mathbb Z} x_i u^i|
x_i \in \mathcal O, \ \ \liminf_{i \to -\infty}\frac{v_p(x_i)}{-i} > 0\big\}
\subset\Gamma.$$
The latter is also a discrete valuation ring with residue field $k((t))$, although it is not complete.
\end{defn}
\begin{defn} Let $\mathcal E_+=\Gamma_+[\frac{1}{p}]$, $\mathcal E=\Gamma[\frac{1}{p}]$ and $\mathcal E^{\dagger}=\Gamma^{\dagger}[\frac{1}{p}]$. Then $\mathcal E$ and
$\mathcal E^{\dagger}$ are the fraction fields of the rings $\Gamma$ and
$\Gamma^{\dagger}$, respectively. Let $\mathcal R$ denote the ring of bidirectional power series:
$$\mathcal R=\big\{\sum_{i \in \mathbb Z} x_i u^i| x_i \in 
\mathcal O[\frac{1}{p}], \ \ \liminf_{i \to -\infty}\frac{v_p(x_i)}{-i}>0,\ \ 
\liminf_{i \to+\infty}\frac{v_p(x_i)}{i}\geq 0\}.$$
Let $\mathcal R_+$ denote its subring:
$$\mathcal R_+=\mathcal R\cap\big\{\sum_{i \in \mathbb N} x_i u^i| x_i \in 
\mathcal O[\frac{1}{p}]\big\}.$$
Clearly $\mathcal E_+\subset\mathcal R_+$ and $\mathcal E^{\dagger}
\subset\mathcal R$.
\end{defn}
\begin{prop} \label{matfact} Let $X$ be an invertible $n\times n$ matrix over
$\mathcal R$. Then there exist invertible $n\times n$ matrices $Y$ over $\mathcal E^{\dagger}$ and $Z$ over $\mathcal R_+$ such that $X = YZ$.
\end{prop}
\begin{proof} This is Proposition 6.5 of \cite{Ke0} on page 172.
\end{proof}
\begin{defn} Let $\sigma_0$ denote the canonical lift of the absolute Frobenius $x \mapsto x^p$ on $k$ to $\mathcal O$. Let $q = p^f$ be a power of $p$ and put $\sigma=\sigma_0^f$. By slight abuse of notation let $\sigma_0$ also denote the ring endomorphism of $\Gamma$ given by the rule:
\begin{equation}\label{2.4.1}
\sigma_0(\sum_{i \in \mathbb Z} x_i u^i)=\sum_{i \in \mathbb Z}
\sigma_0(x_i)u^{ip}.
\end{equation}
This map extends the endomorphism $\sigma_0$ of $\mathcal O$ introduced above. Since it is injective it induces an endomorphism of $\mathcal E$. The latter maps the subrings $\Gamma_+,\Gamma^{\dagger},\mathcal E_+$ and $\mathcal E^{\dagger}$ into themselves. These ring endomorphisms will be denoted by the same symbol by slight abuse of notation. Note that the rule (\ref{2.4.1}) also defines an endomorphism of $\mathcal R$, which maps the subring $\mathcal R_+$ onto itself, and extends the endomorphisms of $\mathcal E^{\dagger}$ defined above. All of these endomorphisms will be denoted by the same symbol, too, by the usual abuse of notation. Moreover let $\sigma$ denote the $f$-th power of any of these endomorphisms. Note that there are other endomorphisms of $\mathcal E$ with these properties and the categories which will consider below do not depend on the choice of this extension, however for some formulas the choice which we have made is more convenient. 
\end{defn}
\begin{defn} Let $R$ be one of the rings $\Gamma_+,\Gamma,\Gamma^{\dagger},\mathcal E_+,\mathcal E,\mathcal E^{\dagger},\mathcal R_+$ or $\mathcal R$. For every $R$-module $M$ let $M\otimes_{R,\sigma}R$
 denote the $R$-module which is the $\mathbb Z$-linear tensor product of $M$ and $R$, subject to the additional conditions:
$$sm\otimes_{R,\sigma}r=m\otimes_{R,\sigma}\sigma(s)r,\quad
m\otimes_{R,\sigma}sr=s(m\otimes_{R,\sigma}r)\quad(\forall r,s\in R,\forall  m\in M).$$
A $\sigma$-module $(M,F)$ over $R$ is a finitely generated free $R$-module $M$ equipped with an $R$-linear map $F: M \otimes_{R,\sigma}R \to M$ that becomes an isomorphism over $R[\frac{1}{p}]$. To specify $F$, it is equivalent to give an additive, $\sigma$-linear map from $M$ to $M$ that acts on any basis of $M$ by a matrix invertible over $R[\frac{1}{p}]$. By slight abuse of notation let $F$ denote this map, too. 
\end{defn}
\begin{defn} Let $R$ be the same as above. Let $\Omega^1_R$ be the free module over $R$ generated by a symbol $du$, and define the derivation $d: R \to \Omega^1_R$ by the formula
$$d\big( \sum_j x_j u^j \big) = \big( \sum_j j x_j u^{j-1} \big)\,du.$$
Recall that a connection on an $R$-module $M$ is an additive map $\nabla: M \to M \otimes_R \Omega^1_{R}$ satisfying the Leibniz rule
$$\nabla(c\mathbf v)=c\nabla(\mathbf v)+\mathbf v\otimes dc\ \ 
(\forall c\in R,\mathbf v\in M).$$ 
There is a natural identification $(M\otimes_{R,\sigma}R)\otimes_R\Omega^1_R
\cong (M\otimes_R\Omega^1_R)\otimes_{R,\sigma}R$ given by the rule:
$$(m\otimes_{R,\sigma}r)\otimes_R\omega\mapsto
(m\otimes_R\omega)\otimes_{R,\sigma}r\ \ 
(\forall m\in M,r\in R,\omega\in\Omega^1_R).$$ 
Using this identification we may define a unique connection $\nabla_{\sigma}$ on $M\otimes_{R,\sigma}R$ with the property:
$$\nabla_{\sigma}(m\otimes_{R,\sigma}1)=(\textrm{id}_M\otimes_Rd\sigma)(\nabla m)\otimes_{R,\sigma}1\ \ 
(\forall m\in M),$$ 
where $d\sigma:\Omega^1_R\to\Omega^1_R$ is the differential of $\sigma$ given by the formula:
$$d\sigma\big( \sum_j x_j u^jdu\big) = \big( \sum_j\sigma(x_j)qu^{jq+q-1}du \big).$$
\end{defn}
\begin{defn} A $(\sigma, \nabla)$-module $(M,F,\nabla)$ over $R$ is a $\sigma$-module $(M,F)$ and a connection $\nabla$ on $M$ such that the following 
diagram commutes:
$$\xymatrix{
M\otimes_{R,\sigma}R \ar^-{\nabla_{\sigma}}[r] \ar^{F}[d] & M\otimes_{R,\sigma}R \otimes_R \Omega^1_{R} \ar^{F \otimes_R\textrm{\rm id}_{\Omega^1_R}}[d] \\
M  \ar^-{\nabla}[r] & M \otimes_R\Omega^1_{R}.}$$
Whenever it is convenient we will let $M$ denote the whole triple
$(M,F,\nabla)$. As usual, a morphism of $\sigma$-modules or $(\sigma,\nabla)$-modules is a homomorphism of the underlying $R$-modules compatible with the additional structure in the obvious fashion. We will write Hom${}_{(\sigma,\nabla)}(\cdot,\cdot)$ for the group of these homomorphisms. An isomorphism of $\sigma$-modules or $(\sigma,\nabla)$-modules is a morphism which has an an inverse. Moreover we can also talk about sub and quotient $\sigma$-modules of $(\sigma,\nabla)$-modules, too.
\end{defn}
\begin{defn}\label{2.8} Now let $R\subset R'$ be two rings from the list above and let $(M,F,\nabla)$ be a $(\sigma,\nabla)$-module over $R$. Let $F'$ denote
$$F\otimes_R\textrm{id}_{R'}:(M\otimes_{R,\sigma}R)\otimes_RR'\cong (M\otimes_RR')\otimes_{R',\sigma}R'\longrightarrow M\otimes_RR'$$
and let $\nabla'$ be the unique connection:
$$\nabla':M\otimes_RR'\longrightarrow
 (M\otimes_RR')\otimes_{R'}\Omega^1_{R'}\cong (M\otimes_R\Omega^1_{R})\otimes_{R'}R'$$
such that
$$\nabla'(m\otimes_Rs)=\nabla m\otimes_Rs+m\otimes_Rds,\ \ 
(\forall m\in M,\forall s\in R).$$ 
Then the triple $(M\otimes_RR',F',\nabla')$ is a $(\sigma,\nabla)$-module over $R'$ which we will denote by $M\otimes_RR'$ for simplicity. Moreover for every homomorphism $h:M\rightarrow M'$ of $(\sigma,\nabla)$-modules over $R$ the $R'$-linear extension $h\otimes_R\textrm{id}_{R'}:M\otimes_RR'
\rightarrow M'\otimes_RR'$ is a morphism of $(\sigma,\nabla)$-modules over $R'$.
\end{defn}
We will need the local versions of de Jong's and Kedlaya's full faithfulness theorems:
\begin{thm}\label{local-faithful} Let $R$ be either $\mathcal E_+$ or $\mathcal E^{\dagger}$, and let $M,M'$ be $(\sigma,\nabla)$-modules over $R$. Then the forgetful map:
$$\textrm{\rm Hom}_{(\sigma,\nabla)}(M,M')\longrightarrow
\textrm{\rm Hom}_{(\sigma,\nabla)}(M\otimes_R\mathcal E,
M'\otimes_R\mathcal E)$$
is an isomorphism.
\end{thm}
\begin{proof} When $R=\mathcal E_+$, see Theorem 1.1 of \cite{J2}, when $R=\mathcal E^{\dagger}$, see Theorem 1.1 of \cite{Ke1}.
\end{proof}
\begin{thm}\label{thm-extension} Let $M$ be a $(\sigma,\nabla)$-module over $\mathcal E^{\dagger}$ such that $M\otimes_{\mathcal E^{\dagger}}
\mathcal R$ is isomorphic to $M'\otimes_{\mathcal R_+}\mathcal R$, where 
$M'$ is a $(\sigma,\nabla)$-module over $\mathcal R_+$. Then there is a
$(\sigma,\nabla)$-module $M''$ over $\mathcal E_+$ such that $M$ is isomorphic to $M''\otimes_{\mathcal E_+}\mathcal E^{\dagger}$.
\end{thm}
\begin{proof} Let $\mathbf e_1, \dots, \mathbf e_n$ be an $\mathcal E^{\dagger}$-basis of $M$ so that
$$F\mathbf e_j = \sum_i \Phi_{ij} \mathbf e_i\textrm{ and }
\nabla \mathbf e_j = \sum_i N_{ij} \mathbf e_i\otimes du,$$
where $\Phi=(\Phi_{ij})$ and $N=(N_{ij})$ are $n\times n$ matrices with coefficients in $\mathcal E^{\dagger}$. By assumption there exists an invertible $n\times n$ matrix $X$ over $\mathcal R$ such that
$$X^{-1} \Phi X^\sigma\textrm{ and }X^{-1} N X + X^{-1}d(X)$$
have entries in $\mathcal R_+$, where the superscript $(\cdot)^{\sigma}$ denotes the action of $\sigma$ on matrices. By Proposition \ref{matfact} we can factor $X$ as $YZ$, where $Y$ is an invertible $n\times n$ matrix over $\mathcal E^{\dagger}$ and $Z$ is an invertible $n\times n$ matrix over $\mathcal R_+$. Now put $\mathbf v_j = \sum_{j} Y_{ij}\mathbf e_i$; then
$$F\mathbf v_j = \sum_i \overline{\Phi}_{ij} \mathbf v_i\textrm{ and }
\nabla \mathbf v_j = \sum_i \overline{N}_{ij} \mathbf v_i
\otimes du,$$
where $\overline{\Phi}=(\overline{\Phi}_{ij})$ and $\overline N=(\overline N_{ij})$ are $n\times n$ matrices over $\mathcal R$ such that
\begin{align*}
\overline{\Phi} &= Y^{-1} \Phi Y^\sigma = Z (X^{-1} \Phi X^\sigma)(Z^{-1})^{\sigma}, \\
\overline{N} &= Y^{-1} N Y + Y^{-1}d(Y) = 
Z(X^{-1}NX+X^{-1}d(X))Z^{-1}+Zd(Z^{-1}).
\end{align*}
Therefore $\overline{\Phi}$ and $\overline N$ have entries in
$\mathcal E_+=\mathcal E^{\dagger}\cap\mathcal R_+$. So the free $\mathcal E_+$-module $M''$ spanned by $\mathbf v_1, \dots, \mathbf v_n$ is left invariant by $F$, while $\nabla$ maps $M''$ into $M''\otimes_{\mathcal E_+}\Omega^1_{\mathcal E_+}$, and hence it is the $(\sigma,\nabla)$-module whose existence the theorem claims.
\end{proof}
Let $R$ be the same as above and let $M$ be an $R$-module equipped with a connection $\nabla: M \to M \otimes_R \Omega^1_{R}$. Recall that a $\mathbf v\in M$ is horizontal if $\nabla\mathbf v=0$. We will need the following form of Dwork's trick:
\begin{prop}\label{horizontal1} Let $M$ be a $(\sigma,\nabla)$-module over $\mathcal R_+$. Then $M$ has an $\mathcal R_+$-basis consisting of horizontal vectors.
\end{prop}
\begin{proof} This is a special case of Corollary 17.2.2 of \cite{Ke3} on page 295.
\end{proof}
\begin{lemma}\label{submodule} Let $M$ be a $(\sigma,\nabla)$-module over $\mathcal R$ which has an $\mathcal R$-basis consisting of horizontal vectors. Let $M'$ be a sub $(\sigma,\nabla)$-module of $M$. Then $M'$ also has an $\mathcal R$-basis consisting of horizontal vectors.
\end{lemma}
\begin{proof} Let $\mathbf e_1, \dots, \mathbf e_n$ be an $\mathcal R$-basis of $M$ consisting of horizontal vectors, and write:
$$F\mathbf e_j = \sum_i\Phi_{ij} \mathbf e_i,$$
where $\Phi=(\Phi_{ij})$ is an $n\times n$ matrix with coefficients in $\mathcal R$. By the compatibility condition between $F$ and $\nabla$ in the definition of $(\sigma,\nabla)$-modules, we get that $d(\Phi)=0$, which implies that $\Phi$ actually has coefficients in
$\mathcal O[\frac{1}{p}]$. We get that $F$ maps the $\mathcal O[\frac{1}{p}]$-span $M_0$ of $\mathbf e_1, \dots, \mathbf e_n$ into itself. Note that every element of $M_0$ is horizontal. So by applying the Dieudonn\'e--Manin classification of $F$-crystals over the perfect field $k$ to $M_0$ we may assume that the matrix of $F^m$ in the basis $\mathbf e_1, \dots, \mathbf e_n$ is diagonal for some positive integer $m$, after an $\mathcal O[\frac{1}{p}]$-linear change of this basis. By substituting $\sigma^m$ for $\sigma$ and $F^m$ for $F$ we may even assume that $m=1$. For every $i=1,2,\ldots,n$ let $\lambda_i\in\mathcal O[\frac{1}{p}]$ be such that $F(\mathbf e_i)=\lambda_i\mathbf e_i$.

Now let $\mathbf v_1, \dots, \mathbf v_l$ be an $\mathcal R$-basis of $M'$, where $l\leq n$. Because the ring $\mathcal R$ is B\'ezout (see Theorem 3.20 of \cite{Ke0} on page 129), we may use Gauss elimination to reduce to the case when for every $j=1,2,\ldots,l$ we have:
$$\mathbf v_j=g_j\mathbf e_j+\sum_{i>j}f_{ij}\mathbf e_i$$
for some $g_j,f_{ij}\in\mathcal R$, after possibly reordering $\mathbf e_1, \dots, \mathbf e_n$. Note that the determinant of $F$ in the basis $\mathbf v_1, \dots, \mathbf v_l$ is
$\lambda_1\cdots\lambda_l\cdot g^{\sigma}_1\cdots g^{\sigma}_l$. Because the matrix of $F$ is invertible, we get that $g^{\sigma}_j\in\mathcal R^*$ for every $j\leq l$. Since for every $g\in\mathcal R$ such that $g^{\sigma}$ is invertible we have $g\in\mathcal R^*$, we may multiply each $\mathbf v_j$ with $g_j^{-1}$ to get another $\mathcal R$-basis of $M'$. In other words we may assume that $g_j=1$ for every $j\leq l$. Then we can perform an additional elimination on the basis $\mathbf v_1, \dots, \mathbf v_l$ and hence we may assume that $f_{ij}=0$ for every $i\leq l$. For every $j\leq l$ we have:
$$\nabla\mathbf v_j=\sum_{i>l}\mathbf e_i\otimes df_{ij},$$
which lies in the sub $\mathcal R$-module of $M\otimes_{\mathcal R}\Omega^1_{\mathcal R}$ spanned by $\mathbf e_{l+1}, \dots, \mathbf e_n$. It also lies in $M'\otimes_{\mathcal R}\Omega^1_{\mathcal R}$, whose intersection with the previous sub $\mathcal R$-module is the zero module. Therefore the vectors $\mathbf v_1, \dots, \mathbf v_l$ are horizontal.
\end{proof}
\begin{rem} It is possible to give a short proof of the lemma above using the equivalence between $(\sigma,\nabla)$-modules over $\mathcal R$ and Weil--Deligne representations (see \cite{Mar}), but at the price of using some heavy machinery. 
\end{rem}
\begin{thm}\label{extension-useful} Let $M$ be a $(\sigma,\nabla)$-module over $\mathcal E^{\dagger}$ which is a sub $(\sigma,\nabla)$-module of $M'\otimes_{\mathcal E_+}\mathcal E^{\dagger}$, where $M'$ is a $(\sigma,\nabla)$-module over $\mathcal E_+$. Then there is a sub $(\sigma,\nabla)$-module $M''$ of $M'$ over $\mathcal E_+$ such that $M$ is equal to $M''\otimes_{\mathcal E_+}\mathcal E^{\dagger}$.
\end{thm}
\begin{proof} By Proposition \ref{horizontal1} the $(\sigma,\nabla)$-module $M'\otimes_{\mathcal E_+}\mathcal R_+$ has an $\mathcal R_+$-basis consisting of horizontal vectors. Therefore $M'\otimes_{\mathcal E_+}\mathcal R\cong(M'\otimes_{\mathcal E_+}\mathcal R_+)\otimes_{\mathcal R_+}\mathcal R$ also has an $\mathcal R$-basis consisting of horizontal vectors. By Lemma \ref{submodule} we get that its sub $(\sigma,\nabla)$-module $M\otimes_{\mathcal E^{\dagger}}\mathcal R$ has an $\mathcal R$-basis consisting of horizontal vectors, too. Let $\mathbf e_1, \dots, \mathbf e_n$ be such an $\mathcal R$-basis, and write:
$$F\mathbf e_j = \sum_i\Phi_{ij} \mathbf e_i,$$
where $\Phi=(\Phi_{ij})$ is an $n\times n$ matrix with coefficients in $\mathcal R$. By the compatibility condition between $F$ and $\nabla$ in the definition of 
$(\sigma,\nabla)$-modules, we get that $d(\Phi)=0$, which implies that
$\Phi$ actually has coefficients in $\mathcal O[\frac{1}{p}]$. Therefore $M\otimes_{\mathcal E^{\dagger}}\mathcal R$ is isomorphic to $M_0\otimes_{\mathcal R_+}\mathcal R$, where $M_0$ is a $(\sigma,\nabla)$-module over $\mathcal R_+$, so by Theorem \ref{thm-extension} there is a $(\sigma,\nabla)$-module $M_+$ over $\mathcal E_+$ such that $M$ is isomorphic to $M_+\otimes_{\mathcal E_+}\mathcal E^{\dagger}$. By Theorem \ref{local-faithful} the $(\sigma,\nabla)$-module $M_+$ is isomorphic to a sub
$(\sigma,\nabla)$-module $M''$ of $M'$, and hence the claim follows. 
\end{proof}
\begin{defn}\label{p-adic} Let $R$ be one of the rings $\Gamma_+,\Gamma,\Gamma^{\dagger},\mathcal E_+,\mathcal E$ or $\mathcal E^{\dagger}$ (but not $\mathcal R_+$ nor $\mathcal R$). Because $\mathcal E$ is a $p$-adically complete local field, for every finitely generated $R$-module $M$ the
$\mathcal E$-vector space $M\otimes_R\mathcal E$ is equipped with a canonical topology compatible with the $p$-adic topology on $\mathcal E$. We will call the restriction of this topology onto $M$ the $p$-adic topology on $M$. Let $\nabla$ be a connection on $M$. Then the associated differential operator $D:M\rightarrow M$ is the composition of $\nabla$ and the map $\textrm{id}_M\otimes_R\upsilon$, where $\upsilon:\Omega^1_R\rightarrow R$ is the unique $R$-linear isomorphism with $\upsilon(du)=1$. Recall that $\nabla$ is said to be topologically quasi-nilpotent if $D^n(x)$ converges to zero for every $x\in M$ with respect to the $p$-adic topology.
\end{defn}
\begin{defn}\label{dieu} For the rest of this section we assume that $f=1$, and hence $q=p$ and $\sigma=\sigma_0$. A Dieudonn\'e module $(M,\nabla,F,V)$ over $R$ (of the type considered in Definition \ref{p-adic}) is 
\begin{enumerate}
\item[$(i)$]  a finitely generated free $R$-module $M$, 
\item[$(ii)$] a topologically quasi-nilpotent connection $\nabla:M\rightarrow M\otimes_R\Omega^1_R$,
\item[$(iii)$] two $R$-linear maps $F:M\otimes_{R,\sigma}R\rightarrow M$ and $V:M\rightarrow M\otimes_{R,\sigma}R$ such that
$$F \circ V = p\cdot\textrm{\rm id}_M\textrm{ and }V \circ F = p \cdot
\textrm{\rm id}_{M\otimes_{R,\sigma}R},$$
and the diagrams:
$$\xymatrix{
M\otimes_{R,\sigma}R \ar^-{\nabla_{\sigma}}[r] \ar^{F}[d] & M\otimes_{R,\sigma}R \otimes_R \Omega^1_{R} \ar^{F \otimes\textrm{\rm id}_{\Omega^1_R}}[d] \\
M  \ar^-{\nabla}[r] & M \otimes \Omega^1_{R}.}\quad\quad\quad
\xymatrix{
M \ar^-{\nabla}[r] \ar^{V}[d] & M \otimes \Omega^1_{R} \ar^{V
\otimes \textrm{\rm id}_{\Omega^1_R}}[d] \\
M\otimes_{R,\sigma}R \ar^-{\nabla_{\sigma}}[r]  & M\otimes_{R,\sigma}R \otimes_R \Omega^1_{R} }$$
are commutative.
\end{enumerate}
These objects form a category. Clearly for every Dieudonn\'e module $(M,\nabla,F,V)$ over $R$ the triple $(M,F,\nabla)$ is a $(\sigma,\nabla)$-module over $R$, and hence we have a forgetful functor from the category of Dieudonn\'e modules over $R$ to the category of $(\sigma,\nabla)$-modules over $R$. 
\end{defn}
\begin{lemma}\label{forgetful} The following hold:
\begin{enumerate}
\item[$(a)$] the forgetful functor from the category of Dieudonn\'e modules over $R$ to the category of $(\sigma,\nabla)$-modules over $R$ is fully faithful,
\item[$(b)$] if $1/p\in R$ then the forgetful functor from the category of Dieudonn\'e modules over $R$ to the category of $(\sigma,\nabla)$-modules over $R$ is an equivalence.
\end{enumerate}
\end{lemma}
\begin{proof} Let $(M,\nabla,F,V)$ be a Dieudonn\'e module over $R$. Because $F$ is invertible over $R[\frac{1}{p}]$, the map $V$ must be unique, so it can be recovered just from the data $(M,F,\nabla)$. Let $(M',\nabla',F',V')$ be another Dieudonn\'e module over $R$ and let $\phi:M\rightarrow M'$ be a homomorphism of the underlying $(\sigma,\nabla)$-modules. Then for every $m\in M$ we have:
$$F'(\phi(Vm))=\phi(FVm)=\phi(pm)=p\phi(m)=F'V'\phi(m).$$
Because $F'$ is invertible over $R[\frac{1}{p}]$ we get that $\phi(Vm)=V'\phi(m)$, so $\phi$ commutes with the Verschiebung operators, too. Therefore claim
$(a)$ holds. Claim $(b)$ is trivial, since $F$ is invertible in this case. 
\end{proof}
\begin{prop} \label{matfact2} Let $X$ be an invertible $n\times n$ matrix over
$\mathcal E$. Then there exist invertible $n\times n$ matrices $Y$ over $\Gamma$ and
$Z$ over $\mathcal O[\frac{1}{p}]$ such that $X = YZ$.
\end{prop}
\begin{proof} The proof is completely standard, but we include it for the reader's convenience. We will keep on multiplying $X$ with invertible $n\times n$ matrices over $\Gamma$ on the left and with invertible $n\times n$ matrices over $\mathcal O[\frac{1}{p}]$ on the right until we get the identity matrix. By multiplying $X$ with a scalar matrix with diagonal term $p^m$ on the right for some suitable integer $m$, the new matrix, also called $X$, will have terms in $\Gamma$ and will have one term which is a unit in $\Gamma$. By permuting the rows and columns of $X$ we may assume that this term is in the upper left corner. These operations correspond to multiplying by an invertible matrix over $\mathbb Z$ on the left and on the right, respectively. By applying row operations to $X$ we may assume that there are no non-zero terms in the first column in $X$ other than in the upper left hand corner. Since $X$ has terms in $\Gamma$, the latter correspond to multiplying by invertible matrices over $\Gamma$ on the left. Applying the same argument repeatedly to the lower right $(n-1)\times(n-1)$ block of $X$ we get an upper triangular matrix with terms in $\Gamma$ whose diagonal terms are invertible in $\Gamma$. This matrix is invertible over $\Gamma$, and hence the claim follows. 
\end{proof}
Now let $R\subset R'$ be two rings of the type considered in Definition \ref{p-adic} and let $(M,\nabla,F,V)$ be a Dieudonn\'e module over $R$. Let $V'$ denote
$$V\otimes_R\textrm{id}_{R'}:M\otimes_RR'\longrightarrow
(M\otimes_{R,\sigma}R)\otimes_RR'\cong (M\otimes_RR')\otimes_{R',\sigma}R'.$$ 
Then the quadruple $(M\otimes_RR',\nabla',F',V')$, where $F'$ and $\nabla'$ are the same as in Definition \ref{2.8}, is a Dieudonn\'e module over $R'$ which we will denote by $M\otimes_RR'$ for simplicity. 
\begin{prop}\label{intersection} Let $M_1$ and $M_2$ be Dieudonn\'e modules over $\Gamma$ and $\mathcal E_+$, respectively, such that 
$M_1\otimes_{\Gamma}\mathcal E$ and $M_2\otimes_{\mathcal E_+}\mathcal E$ are isomorphic Dieudonn\'e modules over $\mathcal E$. Then 
there is a Dieudonn\'e module $M_+$ over $\Gamma_+$ such that $M_1$ and $M_+\otimes_{\Gamma_+}\Gamma$ are isomorphic Dieudonn\'e modules over $\Gamma$, and $M_2$ and $M_+\otimes_{\Gamma_+}\mathcal E_+$ are isomorphic Dieudonn\'e modules over $\mathcal E_+$.
\end{prop}
\begin{proof} The argument will follow the same line of reasoning as the proof of Theorem \ref{thm-extension}. Let $\mathbf e_1, \dots, \mathbf e_n$ be a $\Gamma$-basis of
$M_1$. Then $\mathbf e_1\otimes_{\Gamma,\sigma}1, \dots, \mathbf e_n\otimes_{\Gamma,\sigma}1$ is a $\Gamma$-basis of $M_1\otimes_{\Gamma,\sigma}\Gamma$, so
$$F\mathbf e_j\otimes_{\Gamma,\sigma}1 = \sum_i \Phi_{ij} \mathbf e_i,\ 
V\mathbf e_j = \sum_i B_{ij} \mathbf e_i\otimes_{\Gamma,\sigma}1,
\textrm{ and }
\nabla \mathbf e_j = \sum_i N_{ij} \mathbf e_i\otimes du,$$
where $\Phi=(\Phi_{ij})$, $B=(B_{ij})$ and $N=(N_{ij})$ are $n\times n$ matrices with coefficients in $\Gamma$. By assumption there exists a $n\times n$ matrix $X$ over
$\mathcal E$ such that
$$X^{-1} \Phi X^\sigma,\ (X^{-1})^{\sigma}B X\textrm{ and }X^{-1} N X + X^{-1}d(X)$$
have entries in $\mathcal E_+$, where the superscript $(\cdot)^{\sigma}$ denotes the action of $\sigma$ on matrices. By Proposition \ref{matfact2} we can factor $X$ as $YZ$, where $Y$ is an invertible $n\times n$ matrix over $\Gamma$ and $Z$ is an invertible $n\times n$ matrix over $\mathcal E_+\supset\mathcal O[\frac{1}{p}]$. Now put $\mathbf v_j = \sum_{j} Y_{ij}\mathbf e_i$; then
$$F\mathbf v_j\otimes_{\Gamma,\sigma}1 = \sum_i \overline{\Phi}_{ij} \mathbf v_i,\ 
V\mathbf v_j = \sum_i\overline B_{ij} \mathbf v_i\otimes_{\Gamma,\sigma}1
\textrm{ and }
\nabla \mathbf v_j = \sum_i \overline{N}_{ij} \mathbf v_i
\otimes du,$$
where $\overline{\Phi}=(\overline{\Phi}_{ij})$, $\overline B=(\overline B_{ij})$ and $\overline N=(\overline N_{ij})$ are $n\times n$ matrices over $\mathcal E$ such that
\begin{align*}
\overline{\Phi} &= Y^{-1} \Phi Y^\sigma = Z (X^{-1} \Phi X^\sigma)(Z^{-1})^{\sigma}, \\
\overline B &=(Y^{-1})^{\sigma} B Y= Z^{\sigma}((X^{-1})^{\sigma} B X)Z^{-1}, \\
\overline{N} &= Y^{-1} N Y + Y^{-1}d(Y) = 
Z(X^{-1}N X+X^{-1}d(X))Z^{-1}+Zd(Z^{-1}).
\end{align*}
Therefore $\overline{\Phi}$, $\overline B$ and $\overline N$ have entries in
$\Gamma_+=\Gamma\cap\mathcal E_+$. Let $M_+$ be the free $\Gamma_+$-module spanned by $\mathbf v_1, \dots, \mathbf v_n$; clearly $F$ maps $M_+\otimes_{\Gamma_+,\sigma}\Gamma_+\subset M_+\otimes_{\Gamma,\sigma}\Gamma$ into $M_+$, similarly $V$ maps $M_+$ into $M_+\otimes_{\Gamma_+,\sigma}\Gamma_+$, while $\nabla$ maps $M_+$ into $M_+\otimes_{\Gamma_+}\Omega^1_{\Gamma_+}$. Therefore $(M_+,\nabla|_{M_+},F|_{M_+\otimes_{\Gamma_+,\sigma}}\Gamma_+,V|_{M_+})$ is the Dieudonn\'e-module whose existence the theorem claims.\end{proof}
\begin{rem}\label{d-modules} For every $p$-divisible group $G$ over a $\mathbb F_p$-scheme $S$ let $\mathbf D(G)$ denote the (convergent) Dieudonn\'e module of $G$ over $S$. Now set $S=\textrm{Spec}(k[[t]])$ and $\eta=\textrm{Spec}(k((t)))$. The Dieudonn\'e module of a $p$-divisible group $G$ over $S$ (or over $\eta$) is a Dieudonn\'e module over $\Gamma_+$ (or over $\Gamma$, respectively) in the sense defined above (see 2.2.2, 2.2.4 $h)$ and 2.3.4 of \cite{J1}). Moreover the Dieudonn\'e module of the base change $G_{\eta}$ of a $p$-divisible group $G$ over $S$ to $\eta$ is naturally isomorphic to $\mathbf D(G)\otimes_{\Gamma_+}\Gamma$.\end{rem}
\begin{defn}\label{filtrationdef} Let $G$ be a $p$-divisible group over $\eta$. We say that $G$ has good reduction over $S$ if it extends to a $p$-divisible group $G_+$ over $S$. This extension is unique by Corollary 1.2 of \cite{J2} on page 301 (because $k[[t]]$ and $k((t))$ have finite $p$-basis). We say that $G$ has semi-stable reduction if there exists a filtration
$$0\subseteq G^{\mu}\subseteq G^f\subseteq G$$
by $p$-divisible groups such that the following conditions hold:
\begin{enumerate}
\item[$(a)$] Both $G^f$ and $G/G^{\mu}$ extend to $p$-divisible groups $G_1$ and $G_2$ over $S$.
\item[$(b)$] By Corollary 1.2 of \cite{J2} cited above there is a unique morphism $h:G_1\rightarrow G_2$ extending $G^f\rightarrow G/G^{\mu}$. Then the sheaf
$$G^{m}=\textrm{Ker}(h)\textrm{ and } G^{et}=\textrm{Coker}(h)$$
is a multiplicative, respectively an \'etale $p$-divisible group over $S$. 
\end{enumerate}
\end{defn}
We will say that a $(\sigma,\nabla)$-module $M$ over $\mathcal E$ is overconvergent if there is a $(\sigma,\nabla)$-module $M^{\dagger}$ over
$\mathcal E^{\dagger}$ such that $M$ is isomorphic to $M^{\dagger}\otimes_{\mathcal E^{\dagger}}\mathcal E$.
\begin{thm}\label{overconvergent-condition} Let $H$ be a $p$-divisible group over $\eta$ with semi-stable reduction. Let $G\subseteq H$ be a sub $p$-divisible group. Then $G$ has semi-stable reduction if and only if $\mathbf D(G)\otimes_{\Gamma}\mathcal E$ is overconvergent.
\end{thm}
\begin{rem} It is clear that some condition on $G$ is required. The standard example of the $p$-divisible group of a generically ordinary elliptic curve over $S$ with supersingular special fibre shows that not every sub $p$-divisible group of a $G$ as above will have semi-stable reduction, even when $G$ has good reduction.
\end{rem}
\begin{proof}[Proof of Theorem \ref{overconvergent-condition}] If $G$ has semi-stable reduction then $\mathbf D(G)\otimes_{\Gamma}\mathcal E$ is overconvergent by Corollary 3.16 of \cite{Tr} on page 421. Therefore we only need to show the converse. Let 
$$
0\subseteq H^{\mu}\subseteq H^f\subseteq H
$$
be a filtration by $p$-divisible groups postulated by Definition \ref{filtrationdef}. Let
\begin{equation}\label{2.20.2}
0\subseteq\mathbf D(H^{\mu})\subseteq\mathbf D(H^f)\subseteq\mathbf D(H)
\end{equation}
be the corresponding filtration of Dieudonn\'e modules over $\Gamma$ furnished by functoriality. Consider the filtration:
$$
0\subseteq\mathbf D(H^{\mu})\cap\mathbf D(G)\subseteq\mathbf D(H^f)\cap\mathbf D(G)\subseteq\mathbf D(H)\cap\mathbf D(G)=\mathbf D(G).$$
Because $\Gamma$ is a principal ideal domain this is a sequence of finitely generated free $\Gamma$-modules which are Dieudonn\'e modules if we equip them with the restriction of the operators $F$ and $\nabla$ of $\mathbf D(H)$. Under the correspondence of the main theorem of \cite{J1} on page 6 the filtration corresponds to a filtration of $p$-divisible groups:
\begin{equation}\label{2.20.3}
0\subseteq G^{\mu}\subseteq G^f\subseteq G.
\end{equation}
It will be enough to show that this filtration satisfies the conditions of Definition \ref{filtrationdef}. Let
$$0\subseteq M^{\mu}\subseteq M^f\subseteq M$$
be the filtration of $(\sigma,\nabla)$-modules over $\mathcal E$ which we get from (\ref{2.20.2}) by base change, that is $M^{\mu}=\mathbf D(H^{\mu})\otimes_{\Gamma}\mathcal E,M^f=\mathbf D(H^f)\otimes_{\Gamma}\mathcal E$, and $M=\mathbf D(H)\otimes_{\Gamma}\mathcal E$. Let $H_1$ and $H_2$ be the unique $p$-divisible groups over $S$ which extend $H^f$ and $H/H^{\mu}$, respectively, and let $M_1$ and $M_2$ denote the $(\sigma,\nabla)$-modules $\mathbf D(H_1)\otimes_{\Gamma_+}\mathcal E_+$ and $\mathbf D(H_2)
\otimes_{\Gamma_+}\mathcal E_+$ over $\mathcal E_+$, respectively. Then $M^f\cong M_1\otimes_{\mathcal E_+}\mathcal E$ and $M/M^{\mu}\cong M_2\otimes_{\mathcal E_+}\mathcal E$.

Now let $L$ be the $(\sigma,\nabla)$-module $\mathbf D(G)\otimes_{\Gamma}\mathcal E$ over $\mathcal E$ and let $L^{\dagger}$ be a $(\sigma,\nabla)$-module over $\mathcal E^{\dagger}$ such that $L$ is isomorphic to $L^{\dagger}\otimes_{\mathcal E^{\dagger}}\mathcal E$. Let $\iota:L\to M$ be the injection induced by the inclusion $G\subseteq H$ and let $\pi:L\to M/M^{\mu}$ be the composition of $\iota$ and the quotient map $M\to M/M^{\mu}$. By Theorem
\ref{local-faithful} there is a homomorphism $\pi^{\dagger}:L^{\dagger}\to
M_2\otimes_{\mathcal E_+}\mathcal E^{\dagger}$ such that $\pi=\pi^{\dagger}
\otimes_{\mathcal E_+}\textrm{id}_{\mathcal E^{\dagger}}$. Let $L^{\dagger}_2$ be the image of $h^{\dagger}$. Then $L^{\dagger}_2$ is a sub $(\sigma,\nabla)$-module of $M_2\otimes_{\mathcal E_+}\mathcal E^{\dagger}$ because $\mathcal E^{\dagger}$ is a field. Therefore there is a sub $(\sigma,\nabla)$-module $L_2$ of $M_2$ over $\mathcal E_+$ such that $L^{\dagger}_2$ is equal to $L_2\otimes_{\mathcal E_+}\mathcal E^{\dagger}$ by Theorem \ref{extension-useful}. 

Let $h:H_1\rightarrow H_2$ be the unique morphism extending
$H^f\rightarrow H/H^{\mu}$ and let $\chi:M_1\to M_2$ be the homomorphism of $(\sigma,\nabla)$-modules induced by $h$ via functoriality. Let $M_{1,\mu}$ denote the image of $M_1$ with respect to $\chi$, and let $L^{\dagger}_1$ be the pre-image of $M_{1,\mu}\otimes_{\mathcal E_+}\mathcal E^{\dagger}\subseteq M_2\otimes_{\mathcal E_+}\mathcal E^{\dagger}$ with respect to $\pi^{\dagger}$. It is a sub $(\sigma,\nabla)$-module of $L^{\dagger}$, again because $\mathcal E^{\dagger}$ is a field. The image of $L^{\dagger}_1\otimes_{\mathcal E^{\dagger}}\mathcal E$ with respect to the injection $\iota:L\to M$ lies in:
$$M_f\cong M_1\otimes_{\mathcal E_+}\mathcal E\cong 
(M_1\otimes_{\mathcal E_+}\mathcal E^{\dagger})
\otimes_{\mathcal E^{\dagger}}\mathcal E.$$
Therefore by Theorem \ref{local-faithful} there is a homomorphism
$\iota^{\dagger}:L^{\dagger}_1\to M_1\otimes_{\mathcal E_+}\mathcal E^{\dagger}$ such that $\iota|_{L^{\dagger}_1\otimes_{\mathcal E^{\dagger}}
\mathcal E}=\iota^{\dagger}
\otimes_{\mathcal E_+}\textrm{id}_{\mathcal E^{\dagger}}$. Because $\iota$ is injective and $\mathcal E$ is a field, and hence it is flat over $\mathcal E^{\dagger}$, the map $\iota^{\dagger}$ is injective, too. Therefore there is a sub $(\sigma,\nabla)$-module $L_1$ of $M_1$ over $\mathcal E_+$ such that $L^{\dagger}_1$ is equal to $L_1\otimes_{\mathcal E_+}\mathcal E^{\dagger}$ by Theorem \ref{extension-useful}. Now let
$$0\subseteq L^{\mu}\subseteq L^f\subseteq L$$
be the filtration of $(\sigma,\nabla)$-modules over $\mathcal E$ such that $L^{\mu}=\mathbf D(G^{\mu})\otimes_{\Gamma}\mathcal E$ and $L^f=\mathbf D(G^f)\otimes_{\Gamma}\mathcal E$. By construction we have:
\begin{equation}\label{2.20.5}
L^f\cong L_1\otimes_{\mathcal E_+}\mathcal E\textrm{ and }L/L^{\mu}\cong L_2\otimes_{\mathcal E_+}\mathcal E.
\end{equation}
The $(\sigma,\nabla)$-modules $L_1$ and $L_2$ are Dieudonn\'e modules by part $(b)$ of Lemma \ref{forgetful}. Moreover the isomorphisms in (\ref{2.20.5}) above are isomorphisms in the category of Dieudonn\'e modules by part $(a)$ of Lemma \ref{forgetful}. So we may use Proposition \ref{intersection} to conclude that there are Dieudonn\'e modules $L_{1+}$ and $L_{2+}$ over $\Gamma_+$ such that 
$$L_1\cong L_{1+}\otimes_{\Gamma_+}\mathcal E_+,\ 
L_2\cong L_{2+}\otimes_{\Gamma_+}\mathcal E_+,
\ \mathbf D(G^f)\cong L_1\otimes_{\Gamma_+}\Gamma,\ 
\mathbf D(G/G^{\mu})\cong L_2\otimes_{\mathcal E_+}\Gamma$$
as Dieudonn\'e modules.

By the main theorem of \cite{J1} on page 6 there are $p$-divisible groups $G_1$ and $G_2$ over $S$ such that $L_{1+}\cong\mathbf D(G_1)$ and 
$L_{2+}\cong\mathbf D(G_2)$ as Dieudonn\'e modules. Therefore the base change of $G_1$ and $G_2$ to $\eta$ is isomorphic to $G^f$ and $G/G^{\mu}$, respectively, again by the main theorem of \cite{J1} on page 6. Moreover this result also implies that $G_1, G_2$ is a closed subgroup scheme of $H_1,H_2$, respectively, and the restriction of $h$ onto $G_1$ maps $G_1$ into $G_2$. Let $g$ denote this restriction. Then $\textrm{Ker}(g),\textrm{Coker}(g)$ is represented by a closed subgroup scheme of $\textrm{Ker}(h),\textrm{Coker}(h)$, respectively. We get that the filtration (\ref{2.20.3}) satisfies the conditions of Definition \ref{filtrationdef}, and therefore the theorem follows.
\end{proof}

\section{Semi-simplicity of isocrystals}

\begin{defn} Let again $k$ be a perfect field of characteristic $p$ and let $\mathbb K$ denote the field of fractions of the ring of Witt vectors $\mathcal O$ of $k$. For every finite extension $\mathbb L$ of $\mathbb K$ let $\mathcal O_{\mathbb L}$ denote the valuation ring of $\mathbb L$. For every separated, quasi-projective scheme $X$ over $k$ and finite extension $\mathbb L$ of $\mathbb K$ let
$\textrm{Isoc}(X/\mathbb L)$ and $\textrm{Isoc}^{\dagger}(X/\mathbb L)$ denote the category of $\mathbb L$-linear convergent and overconvergent isocrystals on $X$, respectively. Let $F:X\to X$ be a power of the $p$-power Frobenius, for example $F(x)=x^q$ with $q=p^f$, and let $\sigma:\mathbb L\to\mathbb L$ be a lift of the $q$-power automorphism of the residue field of $\mathbb L$; then we let $\textrm{$F_{\sigma}$-Isoc}(X/\mathbb L)$ and $\textrm{$F_{\sigma}$-Isoc}^{\dagger}(X/\mathbb L)$ denote the category of $\mathbb L$-linear convergent and overconvergent $F_{\sigma}$-isocrystals on $X$, respectively. When $F_{\sigma}$ is a lift of $F$ on a smooth formal lift $\mathfrak X$ of $X$ to Spf$(\mathcal O_{\mathbb L})$, compatible with $\sigma$, we may describe
$\textrm{$F_{\sigma}$-Isoc}(X/\mathbb L)$ as the category of certain vector bundles $\mathcal F$ with an integral connection $\nabla$ on the rigid analytification $\mathcal X$ of $\mathfrak X$, equipped with an isomorphism $F_{\sigma}^*(\mathcal F,\nabla)\to(\mathcal F,\nabla)$ (where we let $F_{\sigma}$ denote the rigid analytification of $F_{\sigma}$, too). When $X$ does not have such a lift $\mathfrak X$, it is more involved to describe this category (see section 2.3 of \cite{Be0} or \cite{LeSt}, for example). These categories are functorial in $X$, that is, given a morphism of $\pi:Y\to X$ of separated, quasi-projective schemes over $k$, there are corresponding pull-back functors
$$\pi^*:\textrm{Isoc}(X/\mathbb L)
\to\textrm{Isoc}(Y/\mathbb L)\textrm{ , }
\pi^*:\textrm{$F_{\sigma}$-Isoc}(X/\mathbb L)\longrightarrow
\textrm{$F_{\sigma}$-Isoc}(Y/\mathbb L),$$
$$\pi^*:\textrm{Isoc}^{\dagger}(X/\mathbb L)
\to\textrm{Isoc}^{\dagger}(Y/\mathbb L)\textrm{ and }
\pi^*:\textrm{$F_{\sigma}$-Isoc}^{\dagger}(X/\mathbb L)\longrightarrow
\textrm{$F_{\sigma}$-Isoc}^{\dagger}(Y/\mathbb L)$$
(see part (i) of 2.3.3 in \cite{Be0}). When the choice of $\sigma$ is clear from the setting (for example when $\mathbb L$ is an unramified extension of $\mathbb K$, and hence $\sigma$ is unique), we will drop it from the notation.
\end{defn}
We will need another result due to Kedlaya:
\begin{thm}\label{extension_thm} Assume that $X$ is smooth and let $\mathcal F$ be an object of $\textrm{\rm $F_{\sigma}$-Isoc}^{\dagger}(X/\mathbb L)$. Let $\pi:Y\to X$ be an open immersion with Zariski-dense image and let $\mathcal G\subseteq\pi^*(\mathcal F)$ be a sub-object. Then there is a unique sub-object
$\mathcal G'\subseteq\mathcal F$ such that $\mathcal G=\pi^*(\mathcal G')$.
\end{thm}
\begin{proof} This is a special case of Proposition 5.3.1 of \cite{Ke2b} on page 1201.
\end{proof}
\begin{cor}\label{very_useful} Assume that $X$ is smooth and let $\mathcal F$ be an object of $\textrm{\rm $F_{\sigma}$-Isoc}^{\dagger}(X/\mathbb L)$. Let $\pi:Y\to X$ be an open immersion with  Zariski-dense image. Then $\mathcal F$ is semi-simple if and only if $\pi^*(\mathcal F)$ is semi-simple.
\end{cor}
\begin{proof} First assume that $\pi^*(\mathcal F)$ is semi-simple  and let
$\mathcal G\subseteq\mathcal F$ be a sub-object. By assumption there is a projection operator $v:\pi^*(\mathcal F)\to\pi^*(\mathcal F)$ with image $\pi^*(\mathcal G)$. By Kedlaya's another full faithfulness theorem (see Theorem 5.2.1 of \cite{Ke2b} on page 1199) there is a unique map $v':\mathcal F\to\mathcal F$ whose pull-back $\pi^*(v')$ is $v$. Since $v^2=v$, we get that $(v')^2=v'$, using again full faithfulness. Moreover the image Im$(v')$ of $v'$ restricted to $Y$ is $\pi^*(\mathcal G)$, so Im$(v')=\mathcal G$, as a consequence of full faithfulness, too. Therefore $v'$ is a projection operator with image $\mathcal G$, and hence $\mathcal F$ is semi-simple.

Now assume that $\mathcal F$ is semi-simple and let $\mathcal G\subseteq\pi^*(\mathcal F)$ be a sub-object. By Theorem \ref{extension_thm} there is a unique sub-object $\mathcal G'\subseteq\mathcal F$ such that $\mathcal G=\pi^*(\mathcal G')$. Since $\mathcal F$ is semi-simple there is a projection operator $v:\mathcal F\to\mathcal F$ with image $\mathcal G'$. The pull-back of $v$ is a projection operator $\pi^*(v):\pi^*(\mathcal F)\to\pi^*(\mathcal F)$ with image $\mathcal G$. Therefore $\pi^*(\mathcal F)$ is semi-simple.
\end{proof}
\begin{notn} Let $(\cdot)^{\wedge}:\textrm{$F_{\sigma}$-Isoc}^{\dagger}(X/\mathbb L)\to\textrm{Isoc}^{\dagger}(X/\mathbb L)$ be the functor furnished by forgetting the Frobenius structure. There is also a forgetful functor
$$(\cdot)^{\sim}:\textrm{$F_{\sigma}$-Isoc}^{\dagger}(X/\mathbb L)\to \textrm{$F_{\sigma}$-Isoc}(X/\mathbb L)$$
(see part (i) of 2.3.9 in \cite{Be0}). By the global version of Kedlaya's full faithfulness theorem, the main result of \cite{Ke1}, this functor is fully faithful when $X$ is smooth. In addition for every finite extension $\mathbb L'/\mathbb L$ there is a functor $\cdot\otimes_{\mathbb L}\mathbb L':\textrm{Isoc}^{\dagger}(X/\mathbb L)\to\textrm{Isoc}^{\dagger}(X/\mathbb L')$ (see 2.3.6 in \cite{Be0}). When $\sigma':\mathbb L'\to\mathbb L'$ be a lift of the $q$-power automorphism of $k$ such that the restriction of $\sigma'$ onto $\mathbb L$ is $\sigma$, there is a similar functor
$\cdot\otimes_{\mathbb L}\mathbb L':\textrm{$F_{\sigma}$-Isoc}^{\dagger}(X/\mathbb L)\to\textrm{$F_{\sigma'}$-Isoc}^{\dagger}(X/\mathbb L')$. Note that the $n$-th iterate $F^n_{\sigma}$ is a lift of $F^n$ and $\sigma^n$, so the category
$\textrm{$F^n_{\sigma^n}$-Isoc}^{\dagger}(X/\mathbb L)$ is well-defined. Let
$$(\,\cdot\,)^{(n)}:\textrm{$F_{\sigma}$-Isoc}^{\dagger}(X/\mathbb L)\longrightarrow \textrm{$F^n_{\sigma^n}$-Isoc}^{\dagger}(X/\mathbb L),\quad\mathcal F\mapsto\mathcal F^{(n)}$$
denote the functor which replaces the Frobenius $\mathbf F_{\mathcal F}$ of an object $\mathcal F$ of $\textrm{$F_{\sigma}$-Isoc}^{\dagger}(X/\mathbb L)$
by
$$\mathbf F_{\mathcal F}^{[n]}=\mathbf F_{\mathcal F}\circ F_{\sigma}^*
(\mathbf F_{\mathcal F})\circ\cdots\circ (F_{\sigma}^{n-1})^*(\mathbf F_{\mathcal F}).$$ 
Finally for the sake of simple notation for every negative integer $n$ let $\mathbf F_{\mathcal F}^{[n]}$ denote the inverse of $\mathbf F_{\mathcal F}^{[-n]}$, and we decree $\mathbf F_{\mathcal F}^{[0]}$ to be the identity of $\mathcal F$.
\end{notn}
\begin{prop}\label{reductions_for_ss} Assume that $X$ is a smooth and let
$\mathcal F$ be an object of the category $\textrm{\rm $F_{\sigma}$-Isoc}^{\dagger}(X/\mathbb L)$. 
\begin{enumerate}
\item[$(i)$] Let $n$ be a positive integer. Then $\mathcal F^{(n)}$ is semi-simple if and only if $\mathcal F$ is semi-simple.
\item[$(ii)$] Let $\mathbb L'/\mathbb L$ be a finite Galois extension and let $\sigma':\mathbb L'\to\mathbb L'$ be a lift of the $q$-power automorphism of $k$ such that the restriction of $\sigma'$ onto $\mathbb L$ is $\sigma$. Then $\mathcal F\otimes_{\mathbb L}\mathbb L'$ is semi-simple if and only if $\mathcal F$ is semi-simple. 
\end{enumerate}
\end{prop}
\begin{proof} Since the functors $(\cdot)^{(n)}$ and $(\cdot)\otimes_{\mathbb L}\mathbb L'$ commute with pull-back, we may assume that $X$ has a smooth formal lift to $\textrm{\rm Spf}(\mathcal O)$ which can be equipped with a lift of the $q$-power Frobenius compatible with $\sigma$, by shrinking $X$ and using Corollary \ref{very_useful}. We first prove $(i)$. First assume that $\mathcal F^{(n)}$ is semi-simple. Let $\mathcal G\subset\mathcal F$ be a sub $F_{\sigma}$-isocrystal. By assumption there is a projection $\pi:\mathcal F^{(n)}\to\mathcal F^{(n)}$ with image $\mathcal G^{(n)}$. Let $\mathbf F_{\mathcal F}:F_{\sigma}^*(\mathcal F)\to\mathcal F$ be the Frobenius of $\mathcal F$ and consider:
$$\pi'=\frac{1}{n}\left(\pi+\mathbf F_{\mathcal F}\circ F_{\sigma}^*(\pi)\circ\mathbf F^{-1}_{\mathcal F}+\cdots+
\mathbf F^{[n-1]}_{\mathcal F}\circ(F_{\sigma}^{n-1})^*(\pi)\circ\mathbf F^{[1-n]}_{\mathcal F}\right).$$
Since $\pi'$ is the linear combination of the composition of horizontal maps, it is horizontal, too. Moreover
\begin{eqnarray}\nonumber
\mathbf F_{\mathcal F}\circ F_{\sigma}^*(\pi')\circ\mathbf F_{\mathcal F}^{-1}\!\!\!&=&\!\!\!
\frac{1}{n}\big(\mathbf F_{\mathcal F}\circ F_{\sigma}^*(\pi)\circ\mathbf F_{\mathcal F}^{-1}+\cdots+\mathbf F_{\mathcal F}^{[n-1]}\circ(F_{\sigma}^{n-1})^*(\pi)\circ\mathbf F_{\mathcal F}^{[1-n]}\\
\nonumber & &\!\!\!+\mathbf F_{\mathcal F}^{[n]}\circ(F_{\sigma}^{n})^*(\pi)\circ
\mathbf F_{\mathcal F}^{[-n]}\big)\\
\nonumber \!\!\!&=&\!\!\!
\frac{1}{n}\big(\mathbf F_{\mathcal F}\circ F_{\sigma}^*(\pi)\circ\mathbf F_{\mathcal F}^{-1}+\cdots+\mathbf F_{\mathcal F}^{[n-1]}\circ(F_{\sigma}^{n-1})^*(\pi)\circ\mathbf F_{\mathcal F}^{[1-n]}\\
\nonumber \!\!\!& &\!\!\!+\pi\big)=\pi',
\end{eqnarray}
where we used that 
$$\mathbf F\circ F_{\sigma}^*(\mathbf F^{[i]})=\mathbf F^{[i+1]},\quad
F_{\sigma}^*(\mathbf F^{[-i]})\circ\mathbf F^{-1}=\mathbf F^{[-i-1]}\quad(\forall
i=0,1,\ldots,n-1),$$
and that $\pi$ is an endomorphism of $\mathcal F^{(n)}$, so it satisfies the identity:
$$\mathbf F^{[n]}\circ(F_{\sigma}^{n})^*(\pi)\circ\mathbf F^{[-n]}=\pi.$$
So $\pi'$ is an endomorphism of the $F_{\sigma}$-isocrystal $\mathcal F$. Since $\mathcal G\subset\mathcal F$ is a sub $F_{\sigma}$-isocrystal, we have
$\mathbf F^{[i]}((F_{\sigma}^i)^*(\mathcal G))\subseteq\mathcal G$ and
$\mathbf F^{[-i]}(\mathcal G)\subseteq (F_{\sigma}^i)^*(\mathcal G)$ for every $i=0,1,\ldots,n-1$. Since for every such $i$ the map $(F_{\sigma}^*)^{i}(\pi)$ is a projection onto $(F_{\sigma}^i)^*(\mathcal G)$, we get that the restriction of $\mathbf F_{\mathcal F}^{[i]}\circ(F_{\sigma}^{i})^*(\pi)\circ\mathbf F_{\mathcal F}^{[-i]}$ onto $\mathcal G$ is the identity for every such $i$, and hence the same holds for $\pi'$. Moreover for every such $i$ the image of $(F_{\sigma}^{i})^*(\pi)$ lies in $(F_{\sigma}^i)^*(\mathcal G)$, so the image of $\mathbf F_{\mathcal F}^{[i]}\circ(F_{\sigma}^{i})^*(\pi)\circ\mathbf F_{\mathcal F}^{[-i]}$ lies in $\mathcal G$. Therefore the same holds for $\pi'$. We get that the latter is a projection onto $\mathcal G$ and hence one implication of claim $(i)$ follows.

Assume now that $\mathcal F$ is semi-simple. Since the direct sum of semi-simple $F$-isocrystals is semi-simple, we may assume without the loss of generality that $\mathcal F$ is actually simple and non-trivial. Then there is a  non-trivial simple sub $F^n_{\sigma}$-isocrystal $\mathcal G\subset\mathcal F^{(n)}$. Let $S$ denote the set of all sub $F^n_{\sigma}$-isocrystals of
$\mathcal F^{(n)}$ isomorphic to $\mathcal G$. For every $\mathcal H\in S$ let $e_{\mathcal H}:\mathcal H\to
\mathcal F^{(n)}$ be the inclusion map. There is a non-empty finite subset $B\subseteq S$ such that the sum:
$$\sum_{\mathcal H\in B}e_{\mathcal H}:\bigoplus_{\mathcal H\in B}
\mathcal H\longrightarrow\mathcal F^{(n)}$$
is an embedding, and $B$ is maximal with respect to this property. Let
$\mathcal J$ denote the image of this map; it is a sub $F^n_{\sigma}$-isocrystal of $\mathcal F^{(n)}$. 

We claim that $\mathcal J$ contains every $\mathcal H\in S$. Assume that this is not the case and let $\mathcal I\in S$ be such that $\mathcal I\nsubseteq\mathcal J$. Then the composition $f$ of $e_{\mathcal I}$ and the quotient map:
$$\mathcal F^{(n)}\longrightarrow\mathcal F^{(n)}/
\bigoplus_{\mathcal H\in B}\mathcal H$$
is non-trivial. Since $\mathcal I$ is simple the kernel of $f$ is the zero crystal, and hence this composition is actually an embedding. Therefore the sum:
$$e_{\mathcal I}+\sum_{\mathcal H\in B}e_{\mathcal H}:\mathcal I\oplus\bigoplus_{\mathcal H\in B}\mathcal H
\longrightarrow\mathcal F^{(n)}$$
is an embedding, too, a contradiction. Note that for $\mathcal I\in S$ the image of $F_{\sigma}^*(\mathcal I)$ with respect to $\mathbf F_{\mathcal F}$ is also an element of $S$, so $\mathbf F_{\mathcal F}$ maps $F_{\sigma}^*(\mathcal J)$ into $\mathcal J$. Since they have the same rank, the map $\mathbf F_{\mathcal F}:F_{\sigma}^*(\mathcal J)\to\mathcal J$ is an isomorphism. We get that
$\mathcal J$ underlies an object of $\textrm{\rm $F_{\sigma}$-Isoc}^{\dagger}(X/\mathbb L)$, and hence it is equal to $\mathcal F^{(n)}$. We get that the latter is semi-simple, so $(i)$ is true.

Next we will show claim $(ii)$. First assume that $\mathcal F\otimes_{\mathbb L}\mathbb L'$ is semi-simple. Let $G$ be the Galois group of $\mathbb L'/\mathbb L$ and let $\mathcal G\subset\mathcal F$ be a sub $F_{\sigma}$-isocrystal. By assumption there is a projection $\pi:\mathcal F\otimes_{\mathbb L}\mathbb L'\to\mathcal F\otimes_{\mathbb L}\mathbb L'$ with image $\mathcal G\otimes_{\mathbb L}\mathbb L'$. Let $\mathfrak X$ be a smooth formal lift of $X$ to Spf$(\mathcal O_{\mathbb L})$, and let $\mathfrak X'$ be the base change of $\mathfrak X$ to Spf$(\mathcal O_{\mathbb L'})$. Let $\mathcal X$ and $\mathcal X'$ denote the rigid analytic space attached to $\mathfrak X$ over $\mathbb L$ and attached to $\mathfrak X'$ over $\mathbb L'$, respectively. Let $F_{\sigma}:\mathfrak X\to\mathfrak X$ be a lift of $F$ to
$\mathfrak X$ compatible with $\sigma$. Then the fibre product $F_{\sigma'}$ of $F_{\sigma}$ and $\sigma'$ is a lift $\mathfrak X'\to\mathfrak X'$ of
$F$ to $\mathfrak X'$ compatible with $\sigma'$. By slight abuse of notation let $F_{\sigma}$ and $F_{\sigma'}$ also denote the morphism of $\mathcal X$ and $\mathcal X'$ induced by $F_{\sigma}$ and $F_{\sigma'}$, respectively.

Note that $G$ acts on Spf$(\mathcal O_{\mathbb L'})$ which in turn induces an action of $G$ on $\mathfrak X'$, and hence on the rigid analytic space $\mathcal X'$ attached to $\mathfrak X'$. By definition $\mathcal F^{\sim}$ is a vector bundle with a flat connection on $\mathcal X$ equipped with a horizontal isomorphism $\mathbf F_{\mathcal F}:F_{\sigma}^*(\mathcal F^{\sim})\to\mathcal F^{\sim}$. In order to avoid overloading the notation, we will drop the superscript
$\sim$ in the rest of the proof. By the global version of Kedlaya's full faithfulness theorem this will not matter at all. Therefore the base change $\mathcal F'=\mathcal F\otimes_{\mathbb L}\mathbb L'$ is a vector bundle with a flat connection on $\mathcal X'$ equipped with an isomorphism $\mathbf F'_{\mathcal F}:F_{\sigma'}^*(\mathcal F')\to\mathcal F'$ which is also equipped with a compatible descent data with respect to the $G$-action, that is, there is an isomorphism $\iota_g:g^*(\mathcal F')\to\mathcal F'$ for every $g\in G$ such that
$\iota_h\circ h^*(\iota_g)=\iota_{hg}$ for every $g,h\in G$ and $\iota_{1}=\textrm{id}_{\mathcal F'}$. Set:
$$\pi'=\frac{1}{|G|}\sum_{g\in G}\iota_g\circ g^*(\pi)\circ\iota_g^{-1}.$$
Since $\pi'$ is the linear combination of the composition of morphism of $F_{\sigma'}$-isocrystals, it is a morphism of $F_{\sigma'}$-isocrystals, too. Set
$\mathcal G'=\mathcal G\otimes_{\mathbb L}\mathbb L'\subset\mathcal F'$; then
$\iota_g(g^*(\mathcal G'))\subseteq\mathcal G'$ and $\iota_g^{-1}
(\mathcal G')\subseteq g^*(\mathcal G')$ for every $g\in G$. Since for every such $g$ the map $g^*(\pi)$ is a projection onto $g^*(\mathcal G')$, we get that the restriction of $\iota_g\circ g^*(\pi)\circ\iota_g^{-1}$ onto $\mathcal G'$ is the identity for every such $g$, and hence the same holds for $\pi'$. Moreover for every such $g$ the image of $g^*(\pi)$ lies in $g^*(\mathcal G)$, so the image of $\iota_g\circ g^*(\pi)\circ\iota_g^{-1}$ lies in $\mathcal G'$. Therefore the same holds for $\pi'$. We get that the latter is a projection onto $\mathcal G'$. Then
\begin{eqnarray}\nonumber
\iota_h\circ h^*(\pi')\circ\iota_h^{-1} &=&
\frac{1}{|G|}\sum_{g\in G}\iota_h\circ h^*(\iota_g)\circ(hg)^*(\pi)\circ
h^*(\iota_g^{-1})\circ\iota_h^{-1}\\
\nonumber & = &
\frac{1}{|G|}\sum_{g\in G}\iota_{gh}\circ(hg)^*(\pi)\circ\iota_{gh}^{-1}=\pi'
\end{eqnarray}
for every $h\in H$, since $\iota_g^{-1}=g^*(\iota_{g^{-1}})$ for every $g\in G$, and hence
\begin{eqnarray}\nonumber
h^*(\iota_g^{-1})\circ\iota_h^{-1} &=&
h^*(g^*(\iota_{g^{-1}}))\circ h^*(\iota_{h^{-1}})
=(hg)^*(\iota_{g^{-1}}\circ(g^{-1})^*(\iota_{h^{-1}}))\\
\nonumber &=&
(hg)^*(\iota_{g^{-1}h^{-1}})=(hg)^*(\iota_{(hg)^{-1}})=\iota_{hg}^{-1}
\end{eqnarray}
for every $g,h\in G$. So we get via Grothendieck's descent that $\pi'$ is the base change of a projection $\mathcal F\to\mathcal F$ with image $\mathcal G$, and hence one implication of claim $(ii)$ holds.

Assume now that $\mathcal F$ is semi-simple. Again we may assume without the loss of generality that $\mathcal F$ is actually simple and non-trivial. Let
$\mathfrak X$ be again a smooth formal lift of $X$ to Spf$(\mathcal O_{\mathbb L})$, and let $\mathfrak X',\mathcal X,\mathcal X',F_{\sigma}$ and
$F_{\sigma'}$ be as above. Moreover let $G$ be again the Galois group of $\mathbb L'/\mathbb L$, and let $\{\iota_g:g^*(\mathcal F')\to\mathcal F'|
g\in G\}$ be the descent data with respect to the $G$-action. There is a non-trivial simple sub $F_{\sigma'}$-isocrystal $\mathcal G\subset\mathcal F\otimes_{\mathbb L}
\mathbb L'$. Let $S$ denote the set of all sub $F_{\sigma'}$-isocrystals of
$\mathcal F\otimes_{\mathbb L}\mathbb L'$ isomorphic to $\mathcal G$, and for every $\mathcal H\in S$ let $e_{\mathcal H}:\mathcal H\to\mathcal F\otimes_{\mathbb L}\mathbb L'$ be the inclusion map. There is a non-empty finite subset $B\subseteq S$ such that the sum:
$$\sum_{\mathcal H\in B}e_{\mathcal H}:\bigoplus_{\mathcal H\in B}
\mathcal H\longrightarrow\mathcal F\otimes_{\mathbb L}\mathbb L'$$
is an embedding, and $B$ is maximal with respect to this property. Let
$\mathcal J$ denote the image of this map; it is a sub $F_{\sigma'}$-isocrystal of
$\mathcal F\otimes_{\mathbb L}\mathbb L'$.

Arguing as above we get that $\mathcal J$ contains every $\mathcal H\in S$. Note that for $\mathcal I\in S$ and $g\in G$ the image of $g^*(\mathcal I)$ with respect to $\iota_g$ is also an element of $S$, so $\iota_g$ maps $g^*(\mathcal J)$ into $\mathcal J$. Since they have the same rank, the map $\iota_g|_{g^*(\mathcal J)}:
g^*(\mathcal J)\to\mathcal J$ is an isomorphism. We get that $\{\iota_g|_{g^*(\mathcal J)}:g^*(\mathcal J)\to\mathcal J|g\in G\}$ is a descent data with respect to the $G$-action. By Grothendieck's decent $\mathcal J$ is the pull-back of a sub-object of $\mathcal F$, and hence it is equal to $\mathcal F\otimes_{\mathbb L}\mathbb L'$. We get that the latter is semi-simple, so $(i)$ is true.
\end{proof}
\begin{defn}\label{3.8} Assume now that $k$ is the finite field $\mathbb F_q$. Let $F:X\to X$ be the $q$-power Frobenius, and let $\mathbb L$ be a totally ramified finite extension of $\mathbb K$. Let $\sigma$ be the identity of $\mathbb L$; it is a lift of the $q$-power automorphism of its residue field. Let $\mathcal F$ be an object of $\textrm{$F_{\sigma}$-Isoc}^{\dagger}(X/\mathbb L)$. Then for every point $x:\textrm{Spec}(\mathbb F_{q^n})\to X$ of degree $n$ the automorphism $x^*(\mathcal F)\to x^*(\mathcal F)$ induced by the $n$-th power of the Frobenius of $x^*(\mathcal F)$ is $\mathbb L_n$-linear, where $\mathbb L_n$ is the unique unramified extension of $\mathbb L$ of degree $n$. Let $\textrm{Frob}_x(\mathcal F)$ denote this map. Let $|X|$ denote the set of closed points of $X$ and let $\|X\|$ denote the set of isomorphism classes of pairs consisting of a finite extension $\mathbb F_{q^n}$ of $\mathbb F_q$ and a morphism $\alpha:\textrm{Spec}(\mathbb F_{q^n})\to X$. Associating to such $\alpha$ the point $\alpha(\textrm{Spec}(\mathbb F_{q^n}))\in|X|$ one gets a canonical map $\phi:\|X\|\to|X|$. Let $[X]\subseteq\|X\|$ denote the subset of the isomorphism classes of those morphisms $\alpha:\textrm{Spec}(\mathbb F_{q^n})\to X$ which induce an isomorphism between $\textrm{Spec}(\mathbb F_{q^n})$ and the closed sub-scheme $\alpha(\textrm{Spec}(\mathbb F_{q^n}))$. If $x,y\in[X]$ are such that $\phi(x)=\phi(y)$, then the linear maps $\textrm{Frob}_x(\mathcal F)$ and $\textrm{Frob}_y(\mathcal F)$ are isomorphic, and we let $\textrm{Frob}_{\phi(x)}(\mathcal F)$ denote this common isomorphism class, by slight abuse of notation.
\end{defn}
\begin{defn}  Fix now an isomorphism $\iota:\overline{\mathbb K}\rightarrow\mathbb C$ and let $|\cdot|:\mathbb C\rightarrow\mathbb R_{\geq0}$ be the usual archimedean absolute value on $\mathbb C$. Recall that $\mathcal F$ is (point-wise) $\iota$-pure of weight $w$, where $w\in\mathbb Z$, if for every $x\in|X|$ and for every eigenvalue $\alpha\in\overline{\mathbb K}$ of $\textrm{Frob}_x(\mathcal F)$ we have $|\iota(\alpha)|=q^{w\deg(x)/2}$. We say that an overconvergent $F$-isocrystal $\mathcal F$ on $U$ is $\iota$-mixed of weight $\leq w$, if it has a filtration by objects of $\textrm{$F_{\sigma}$-Isoc}^{\dagger}(X/\mathbb L)$ which are all $\iota$-pure of weight at most $w$. We will drop $\iota$ from the terminology if $\mathcal F$ satisfies the conditions for every possible choice of $\iota$. Similarly we may talk about the purity and mixedness of rigid cohomology groups of overconvergent $F$-isocrystals. 
\end{defn}
\begin{prop}\label{abe-caro} Let $\mathcal F$ be a $\iota$-pure object of $\textrm{\rm $F_{\sigma}$-Isoc}^{\dagger}(X/\mathbb L)$ and assume that $X$ is realisable. Then $\mathcal F^{\wedge}$ is a semi-simple overconvergent isocrystal.
\end{prop}
\begin{proof} This is a special case of Theorem 4.3.1 in \cite{AC}.
\end{proof}
\begin{prop}\label{3.9} Let $\mathcal F$ be a semi-simple object of $\textrm{\rm $F_{\sigma}$-Isoc}^{\dagger}(X/\mathbb L)$. Then $\mathcal F^{\wedge}$ is a semi-simple overconvergent isocrystal.
\end{prop}
\begin{proof} We are going to argue very similarly to the proof of one half of Proposition \ref{reductions_for_ss}. Since the direct sum of semi-simple overconvergent isocrystals is semi-simple, we may assume without the loss of generality that $\mathcal F$ is actually simple and non-trivial. Then there is a  non-trivial simple overconvergent sub-isocrystal $\mathcal G\subset\mathcal F^{\wedge}$. Let $S$ denote the set of all overconvergent sub-isocrystals of $\mathcal F^{\wedge}$ isomorphic to $\mathcal G$. For every $\mathcal H\in S$ let $e_{\mathcal H}:\mathcal H\to
\mathcal F^{\wedge}$ be the inclusion map. There is a non-empty finite subset $B\subseteq S$ such that the sum:
$$\sum_{\mathcal H\in B}e_{\mathcal H}:\bigoplus_{\mathcal H\in B}
\mathcal H\longrightarrow\mathcal F^{\wedge}$$
is an embedding, and $B$ is maximal with respect to this property. Let $\mathcal J$ denote the image of this map; it is an overconvergent sub-isocrystal of $\mathcal F^{\wedge}$.  Arguing exactly the same we as we did in the proof of Proposition \ref{reductions_for_ss} we get that $\mathcal J$ contains every $\mathcal H\in S$, and hence the map $\mathbf F_{\mathcal F}:F_{\sigma}^*(\mathcal J)\to\mathcal J$ is an isomorphism. We get that $\mathcal J$ underlies an object of $\textrm{\rm $F_{\sigma}$-Isoc}^{\dagger}(X/\mathbb L)$, and hence it is equal to $\mathcal F^{\wedge}$. We get that the latter is semi-simple.
\end{proof}

\section{The $p$-adic arithmetic and geometric monodromy groups}

\begin{notn} Assume for a moment that $K$ is any field of characterstic zero. For any $K$-linear Tannakian category
$\mathbf T$, finite extension $L/K$ of $K$, and $L$-valued fibre functor $\omega$ on $\mathbf T$ let $\pi(\mathbf T,\omega)$ denote the Tannakian fundamental group of $\mathbf T$ with respect to $\omega$. Let $\mathbf B$ be a $K$-linear Tannakian category, let $\mathbf C$ be a strictly full rigid abelian tensor subcategory of $\mathbf B$ and let $b:\mathbf C\to\mathbf B$ be the inclusion functor. Let $L/K$ be a finite extension of $K$ and let $\mathbf A$ be another $K$-linear Tannakian category equipped with an $L$-valued fibre functor $\omega$ and assume that there is a faithful tensor functor $a:\mathbf B\to\mathbf A$ of Tannakian categories. Let $a_*:\pi(\mathbf A,\omega)\to\pi(\mathbf B,\omega\circ a)$ and $b_*:\pi(\mathbf B,\omega\circ a)\to \pi(\mathbf C,\omega\circ a\circ b)$ be the homomorphisms induced by $a$ and $b$, respectively.
\end{notn}
The next proposition will supply a handy condition for certain sequences of Tannakian fundamental groups to be exact. 
\begin{prop}\label{my_variant} Assume that the following hods:
\begin{enumerate}
\item[$(i)$] For an object $\mathcal G$ of $\mathbf B$ the object $a(\mathcal G)$ of $\mathbf A$ is trivial if and only if $\mathcal G$ is an object of $\mathbf C$.
\item[$(ii)$] Let $\mathcal G$ be an object of $\mathbf B$, and let $\mathcal H_0\subseteq a(\mathcal G)$ denote the largest trivial sub-object. Then there exists an
$\mathcal H\subseteq\mathcal G$ with $\mathcal H_0=a(\mathcal H)$.
\item[$(iii)$] Every object $\mathcal G$ of $\mathbf A$ is a sub-object of an object of the form $a(\mathcal H)$ with some object $\mathcal H$ of $\mathbf B$. 
\end{enumerate}
Then the sequence:
$$\xymatrix { 0 \ar[r] & \pi(\mathbf A,\omega) \ar[r]^-{a_*} &
\pi(\mathbf B,\omega\circ a) \ar[r]^-{b_*} & 
\pi(\mathbf C,\omega\circ a\circ b) \ar[r] & 0.}$$
is exact.
\end{prop}
\begin{proof} Assume first that $L=K$. Then the Tannakian categories $\mathbf A,\mathbf B$ and $\mathbf C$ are actually neutral. Since $\mathbf C$ is a subcategory of $\mathbf B$, the map $b_*$ is surjective, and in particular it is fully faithful. Therefore the claim is true in this case by Theorem A.1 of \cite{EHS} on page 396. We prove the proposition in general by reducing to this case. 

Let $\mathbf T$ be an arbitrary $K$-linear Tannakian category. Recall that an $L$-module in $\mathbf T$ is a pair $(\mathcal X,\alpha_{\mathcal X})$ with $\mathcal X$ an object of $\mathbf T$ and $\alpha_{\mathcal X}$ a homomorphism
$L\to\textrm{End}_{\mathbf T}(\mathcal X)$. Let $\mathbf T_{(L)}$ denote the category of $L$-modules in $\mathbf T$. It is an $L$-linear Tannakian category. If there is an $L$-linear fibre functor $\lambda$ on $\mathbf T$ then it induces a $L$-linear fibre functor $\lambda_{(L)}$ on $\mathbf T_{(L)}$ which makes $\mathbf T_{(L)}$ into a neutral Tannakian category tensor-equivalent to the representation category of the Tannakian fundamental group $\pi(\mathbf T,\lambda)$ (see Proposition 3.11 of \cite{DM}). So by the above we need to show the following:
\begin{enumerate}
\item[$(a)$] For an object $(\mathcal G,\alpha_{\mathcal G})$ of $\mathbf B_{(L)}$ the object $a_{(L)}(\mathcal G,\alpha_{\mathcal G})$ of $\mathbf A_{(L)}$ is trivial if and only if $(\mathcal G,\alpha_{\mathcal G})$ is an object of $\mathbf C_{(L)}$.
\item[$(b)$] Let $(\mathcal G,\alpha_{\mathcal G})$ be an object of $\mathbf B_{(L)}$, and let $(\mathcal H_0,\alpha_{\mathcal H_0})\subseteq a_{(L)}(\mathcal G,\alpha_{\mathcal G})$ denote the largest trivial sub-object. Then there exists an
$(\mathcal H,\alpha_{\mathcal H})\subseteq(\mathcal G,\alpha_{\mathcal G})$ with $(\mathcal H_0,\alpha_{\mathcal H_0})=a_{(L)}(\mathcal H,\alpha_{\mathcal H})$.
\item[$(c)$] Every object $(\mathcal G,\alpha_{\mathcal G})$ of $\mathbf A_{(L)}$ is a sub-object of an object $a_{(L)}(\mathcal H,\alpha_{\mathcal H})$ with some object $(\mathcal H,\alpha_{\mathcal H})$ of $\mathbf B_{(L)}$. 
\end{enumerate}
In order to do so we will need to recall the following construction. Fix a $K$-basis $e_1,e_2,\ldots,e_n$ of $L$ and for every $l\in L$ let $M(l)\in M_n(K)$ denote the matrix of multiplication by $l$ in this basis. For every object $\mathcal X$ of a $K$-linear Tannakian category $\mathbf T$ as above let $\pi^{\mathcal X}_{ij}:\mathcal X^{\oplus n}\to\mathcal X^{\oplus n}$ be the morphism which maps the $i$-th component identically to the $j$-th component and maps all other components to zero. Then the $K$-span $M(\mathcal X)$ of the $\pi^{\mathcal X}_{ij}$ in
$\textrm{End}_{\mathbf T}(\mathcal X)$ is isomorphic to the matrix algebra $M_n(K)$ such that $\pi^{\mathcal X}_{ij}$ corresponds to the elementary matrix which has zeros everywhere except in the $j$-th term of the $i$-th row, where it has entry $1$. Now let
$\mathcal X_{(L)}$ denote $\mathcal X^{\oplus n}$ equipped with the homomorphism $\alpha_{\mathcal X_{(L)}}$ which maps every $l\in L$ to the element of $M(\mathcal X)$ corresponding to $M(l)$ under the isomorphism $M(\mathcal X)\cong M_n(K)$ constructed above. Then $(\mathcal X^{\oplus n},\alpha_{\mathcal X_{(L)}})$ is an object of $\mathbf T_{(L)}$. By definition an object $(\mathcal G,\alpha_{\mathcal G})$ of $\mathbf T_{(L)}$ is trivial if it is isomorphic to $\mathcal X_{(L)}$ for some trivial object $\mathcal X$ of $\mathbf T$.
\begin{lemma}\label{triviality} An object $(\mathcal G,\alpha_{\mathcal G})$ of $\mathbf T_{(L)}$ is trivial if and only if $\mathcal G$ is trivial.
\end{lemma}
\begin{proof} The first condition obviously implies the second. Now let $(\mathcal G,\alpha_{\mathcal G})$ be an object of $\mathbf T_{(L)}$ such that $\mathcal G$ is trivial. Then $\textrm{End}_{\mathbf T}(\mathcal G)\cong M_m(K)$ where $m$ is the rank of $\mathcal G$. Since $L$ is a semi-simple $K$-algebra, the representation $\alpha_{\mathcal G}:L\to M_m(K)$ decomposes into a direct sum of irreducible representations. This decomposition underlies a decomposition of $\mathcal G$ into direct summands. Since $\mathcal G$ is trivial, the same is true for these summands. Moreover the $K$-algebra $L$ has a unique irreducible representation which is given by the rule
$l\mapsto M(l)$. The claim is now clear.
\end{proof}
Condition $(a)$ now immediately follows from condition $(i)$. Next we show $(b)$. By assumption $(ii)$ there exists an
$\mathcal H\subseteq\mathcal G$ with $\mathcal H_0=a(\mathcal H)$. The action of $L$ via $a_{(L)}(\alpha_{\mathcal G})$ leaves
$\mathcal H_0$ invariant, so $\alpha_{\mathcal G}$ also leaves
$\mathcal H$ invariant. Therefore
$(\mathcal H,\alpha_{\mathcal G}|_{\mathcal H})$ is a sub-object of $(\mathcal G,\alpha_{\mathcal G})$ such that $(\mathcal H_0,\alpha_{\mathcal H_0})=a_{(L)}(\mathcal H,\alpha_{\mathcal G}|_{\mathcal H})$. Condition $(b)$ follows. Finally we show condition $(c)$. By assumption $(iii)$ there is an object $\mathcal H_0$ of $\mathbf B$ such that $\mathcal G$ is a sub-object of
 $a(\mathcal H_0)$. Clearly $\mathcal G^{(L)}$ is a sub-object of $a(\mathcal H_0)^{(L)}=a_{(L)}(\mathcal H_0^{(L)})$ in
$\mathbf A_{(L)}$. So it will be enough to show that $(\mathcal G,\alpha_{\mathcal G})$ is a sub-object of $\mathcal G^{(L)}$ in
$\mathbf B_{(L)}$. However the morphism $\mathcal G\to\mathcal G^{\oplus n}$ given by the vector $(\alpha_{\mathcal G}(e_1),\ldots,\alpha_{\mathcal G}(e_n))$ induces such an embedding.
\end{proof}
\begin{defn} We will continue with the assumptions and notation of Definition \ref{3.8}. Suppose now that $X$ is geometrically connected. Fix a point $x\in X(\mathbb F_{q^n})$. The pull-back with respect to $x$ furnishes a functor from $\textrm{Isoc}^{\dagger}(X/\mathbb L)$ into the category of finite dimensional $\mathbb L_n$-linear vector spaces which makes $\textrm{Isoc}^{\dagger}(X/\mathbb L)$ into a Tannakian category. (See 2.1 of \cite{Cr} on page 438.) Let $\omega_{x}$ denote the corresponding fibre functor on $\textrm{Isoc}^{\dagger}(X/\mathbb L)$. For every strictly full rigid abelian tensor subcategory $\mathbf C$ of $\textrm{Isoc}^{\dagger}(X/\mathbb L)$ let $\textrm{DGal}(\mathbf C,x)$ denote the Tannakian fundamental group of $\mathbf C$ with respect to the fibre functor $\omega_{x}$. Note that the composition of the forgetful functor $(\cdot)^{\wedge}:\textrm{$F_{\sigma}$-Isoc}^{\dagger}(X/\mathbb L)\to\textrm{Isoc}^{\dagger}(X/\mathbb L)$ and $\omega_{x}$, which we will denote by the same symbol by slight abuse of notation, makes $\textrm{$F_{\sigma}$-Isoc}^{\dagger}(X/\mathbb L)$ into a Tannakian category, too. Similarly as above for every strictly full rigid abelian tensor subcategory $\mathbf C$ of
$\textrm{$F_{\sigma}$-Isoc}^{\dagger}(X/\mathbb L)$ let
$\textrm{Gr}(\mathbf C,x)$ denote the Tannakian fundamental group of $\mathbf C$ with respect to the fibre functor $\omega_{x}$. Moreover let $\mathbf C^{\wedge}$ denote strictly full rigid abelian tensor subcategory generated by the image of $\mathbf C$ with respect to $(\cdot)^{\wedge}$ and for simplicity let $\textrm{DGal}(\mathbf C,x)$ denote $\textrm{DGal}(\mathbf C^{\wedge},x)$.
\end{defn}
\begin{notn} For every Tannakian category $\mathbf C$ and every object
$\mathcal F$ of $\mathbf C$ let $\dal\mathcal F\dar$ denote the strictly full rigid abelian tensor subcategory of $\mathbf C$ generated by $\mathcal F$. For every object $\mathcal F$ of $\textrm{$F_{\sigma}$-Isoc}^{\dagger}(X/\mathbb L)$ let
$\textrm{DGal}(\mathcal F,x)$ and  $\textrm{Gr}(\mathcal F,x)$ denote $\textrm{DGal}(\dal\mathcal F^{\wedge}\dar,x)$ and $\textrm{Gr}(\dal\mathcal F\dar,x)$, respectively.  Moreover for every such $\mathcal F$ let $\dal\mathcal F\dar_{const},\mathbf W(\mathcal F,x)$ denote the strictly full rigid abelian tensor subcategory of constant objects of $\dal\mathcal F\dar$ and the Tannakian fundamental group of $\dal\mathcal F\dar_{const}$ with respect to the fibre functor $\omega_{x}$, respectively. Let $\alpha:\textrm{DGal}(\mathcal F,x)\to\textrm{Gr}(\mathcal F,x)$ be the homomorphism induced by the forgetful functor $(\cdot)^{\wedge}:\dal\mathcal F\dar\to\dal\mathcal F^{\wedge}\dar$, and let $\beta:\textrm{Gr}(\mathcal F,x)\to\mathbf W(\mathcal F,x)$ be the homomorphism induced by the inclusion $\dal\mathcal F\dar_{const}\subset\dal\mathcal F\dar$.
\end{notn}
\begin{defn} Let $\mathcal F$ be an object of $\textrm{$F_{\sigma}$-Isoc}^{\dagger}(X/\mathbb L)$. Note that for every $x\in [X]$ the Frobenius $\textrm{Frob}_x(\mathcal F)$ is an automorphism of the the fibre functor $\omega_{x}$, so it furnishes an element of
$\textrm{Gr}(\mathcal F,x)(\mathbb L_n)$ which we will denote by the same symbol by slight abuse of notation. For the reasons we mentioned in Definition \ref{3.8}, the conjugacy class of $\textrm{Frob}_x(\mathcal F)$ only depends on $\phi(x)$, which we will denote by $\textrm{Frob}_{\phi(x)}(\mathcal F)$. Now let $y$ be another point in $[X]$. Then there is an isomorphism between $\textrm{Gr}(\mathcal F,y)$ and $\textrm{Gr}(\mathcal F,x)$ after base change to the algebraic closure $\overline{\mathbb L}$, which is well-defined up to conjugacy. Let $\textrm{Frob}_y(\mathcal F)$ also denote the conjugacy class in $\textrm{Gr}(\mathcal F,x)(\overline{\mathbb L})$ spanned by the image of $\textrm{Frob}_y(\mathcal F)$ under this isomorphism. There is a conjugacy class $\textrm{Frob}_y\subseteq\textrm{\rm Gr}(\textrm{\rm $F_{\sigma}$-Isoc}^{\dagger}(X/\mathbb L),x)(\overline{\mathbb L})$ whose image is $\textrm{Frob}_y(\mathcal F)$ under the canonical surjection $\textrm{\rm Gr}(\textrm{\rm $F_{\sigma}$-Isoc}^{\dagger}(X/\mathbb L),x)\to
\textrm{Gr}(\mathcal F,x)$ for every $\mathcal F$ as above.
\end{defn}
The monodromy group $\textrm{DGal}(\mathcal F,x)$ was introduced by Crew~\cite{Cr}. Next we describe its relationship to $\textrm{Gr}(\mathcal F,x)$. 
\begin{prop}\label{crew_sequence} Assume that $\mathcal F^{\wedge}$ is semi-simple. Then the sequence:
$$\xymatrix { 0 \ar[r] & \textrm{\rm DGal}(\mathcal F,x) \ar[r]^{\alpha} &
\textrm{\rm Gr}(\mathcal F,x) \ar[r]^{\beta} & 
\mathbf W(\mathcal F,x) \ar[r] & 0}$$
is exact.
\end{prop}
\begin{proof} By Proposition \ref{my_variant} we only have to check the following:
\begin{enumerate}
\item[$(i)$] For an object $\mathcal G$ of $\dal\mathcal F\dar$ the object $\mathcal G^{\wedge}$ of $\dal\mathcal F^{\wedge}\dar$ is trivial if and only if $\mathcal G$ is an object of $\dal\mathcal F\dar_{const}$.
\item[$(ii)$] Let $\mathcal G$ be an object of $\dal\mathcal F\dar$, and let $\mathcal H_0\subseteq\mathcal G^{\wedge}$ denote the largest trivial sub-object. Then there exists an $\mathcal H\subseteq\mathcal G$ with $\mathcal H_0=\mathcal H^{\wedge}$.
\item[$(iii)$] Every object $\mathcal G$ of $\dal\mathcal F^{\wedge}\dar$ is a sub-object of an object of the form $\mathcal H^{\wedge}$ with some object
$\mathcal H$ of $\dal\mathcal F\dar$. 
\end{enumerate}
Condition $(i)$ trivially holds: an $F$-isocrystal is constant if and only if it is trivial as an isocrystal. Next we show $(ii)$. The maximal trivial overconvergent sub-isocrystal $\mathcal H_0$ of an overconvergent $F$-isocrystal $\mathcal G$ is generated by (overconvergent) horizontal sections of $\mathcal G$. Since the Frobenius map of $\mathcal G$ respects horizontal sections, the isocrystal
$\mathcal H_0$ underlies an overconvergent $F$-isocrystal. Finally we prove $(iii)$. Because the image of $\dal\mathcal F\dar$ under $(\cdot)^{\wedge}$ is closed under directs sums, tensor products and duals, there is an object $\mathcal H$ of $\dal\mathcal F\dar$ such that $\mathcal G$  is a subquotient of $\mathcal H^{\wedge}$. Since $\mathcal F^{\wedge}$ is semi-stable, so is every object in $\dal\mathcal F^{\wedge}\dar$. Therefore $\mathcal G$  is isomorphic to a sub-object of $\mathcal H^{\wedge}$.
\end{proof}
\begin{cor}\label{crew_sequence2} Assume either that $X$ is realisable, geometrically connected, and $\mathcal F$ is $\iota$-pure, or that $\mathcal F$ is semi-simple. Then the sequence:
$$\xymatrix { 0 \ar[r] & \textrm{\rm DGal}(\mathcal F,x) \ar[r]^{\alpha} &
\textrm{\rm Gr}(\mathcal F,x) \ar[r]^{\beta} & 
\mathbf W(\mathcal F,x) \ar[r] & 0}$$
is exact.
\end{cor}
\begin{proof} The first case follows from Proposition \ref{abe-caro} and Proposition \ref{crew_sequence}, while the second case follows from Proposition \ref{3.9} and Proposition \ref{crew_sequence}.
\end{proof}
\begin{prop}\label{easy_ss} Assume that $X$ is realisable, smooth, and geometrically connected. Also suppose that $\mathcal F$ is $\iota$-pure and $\textrm{\rm Frob}_x(\mathcal F)$ is semi-simple for a closed point $x\in|X|$. Then $\mathcal F$ is semi-simple.
\end{prop}
\begin{proof} Let $n$ be the degree of $x$ and let $X_n$ be the base change of $X$ to Spec$(\mathbb F_{q^n})$. Let $\sigma'$ be the identity of
$\mathbb L_n$; it is a lift of the $q^n$-power automorphism of its residue field. Then $\mathcal F^{(n)}\otimes_{\mathbb L}\mathbb L_n$ is an object of 
$\textrm{$F_{\sigma'}$-Isoc}^{\dagger}(X/\mathbb L_n)=
\textrm{$F_{\sigma'}$-Isoc}^{\dagger}(X_n/\mathbb L_n)$. By Proposition \ref{reductions_for_ss} it will be sufficient to prove that $\mathcal F^{(n)}\otimes_{\mathbb L}\mathbb L_n$ is semi-simple. Let $r:|X_n|\to|X|$ be the map induced by the base change morphism $X_n\to X$ of schemes. Then for every $y\in|X_n|$ the linear map $\textrm{\rm Frob}_y(\mathcal F^{(n)}\otimes_{\mathbb L}\mathbb L_n)$ is the $n$-th power of $\textrm{\rm Frob}_{r(y)}(\mathcal F)$. We get that  $\mathcal F^{(n)}\otimes_{\mathbb L}\mathbb L_n$ is $\iota$-pure (as an object of $\textrm{$F_{\sigma'}$-Isoc}^{\dagger}(X_n/\mathbb L_n)$), and $\textrm{\rm Frob}_x(\mathcal F)$ is semi-simple for every $y\in|X_n|$ such that $r(y)=x$. So we may assume without the loss of generality that $x\in X(\mathbb F_q)$.

In this case $\dal\mathcal F\dar$ is a neutral Tannakian category with respect to the fibre functor $\omega_x$, and hence it will be sufficient to prove that $\textrm{Gr}(\mathcal F,x)$ is reductive. Let $N\subseteq\textrm{Gr}(\mathcal F,y)$ be a connected unipotent normal subgroup. Since $\mathcal F^{\wedge}$ is semi-simple by Proposition \ref{abe-caro}, the intersection $N\cap\textrm{DGal}(\mathcal F,y)$ is trivial. Therefore $N$ injects into $\mathbf W(\mathcal F,y)$ with respect to the quotient map $\beta:\textrm{Gr}(\mathcal F,y)\to\textbf W(\mathcal F,y)$. With respect to the Zariski topology the group $\mathbf W(\mathcal F,x)$ is generated by the element which, as an automorphism of $\omega_x$, is the $q$-linear Frobenius $\textrm{Frob}_x(\mathcal C)$ for every object
$\mathcal C$ of $\dal\mathcal F\dar_{const}$. Since for every object $\mathcal G$ of $\dal\mathcal F\dar$ the linear map $\textrm{Frob}_x(\mathcal G)$ is semi-simple by assumption, we get that the group $\textbf W(\mathcal F,y)$ is reductive. Therefore $N$ is trivial, and hence $\textrm{Gr}(\mathcal F,y)$ is reductive, too.
\end{proof}
\begin{defn} For every $p$-divisible group $G$ over a $\mathbb F_q$-scheme $S$ let $D(G)=\mathbf D(G)\otimes_{\mathbb Z_q}\mathbb Q_q$ be the associated convergent Dieudonn\'e $F$-isocrystal. For every abelian scheme $C$ over $S$ let $C[p^{\infty}]$ and $\mathbf D(C)$, $D(C)$ denote the $p$-divisible group of $C$, and $\mathbf D(C[p^{\infty}])$, $D(C[p^{\infty}])$, respectively. When $S=U$ is a geometrically connected smooth quasi-projective curve defined over the finite field $\mathbb F_q$ of characteristic $p$, as in the introduction, let $D^{\dagger}(C)$ denote the overconvergent crystalline Dieudonn\'e module of $C$ over $U$ (for a construction see \cite{KT}, sections 4.3--4.8). Its key property is that there is a natural isomorphism $D^{\dagger}(C)^{\sim}\cong D(C)$. 
\end{defn} 
\begin{rem} It is easy to see that the claim above applies to the overconvergent Dieudonn\'e crystal of an abelian scheme as follows. Let $A$ be an abelian scheme over $U$ and for every closed point $y\in|U|$ let $A_y$ be the fibre of $A$ over $y$. For every such $y$ (of degree $n$) the fibre of $D^{\dagger}(A)$ at $y$ is isomorphic to $D(A_y)\otimes_{\mathbb Z_{q^n}}\mathbb Q_{q^n}$. The action of the $q^n$-linear Frobenius on the latter is semi-simple by classical Honda-Tate theory. Moreover we also get that $D^{\dagger}(A)$ is $\iota$-pure of weight $1$. So Theorem \ref{semisimple} follows from Corollary \ref{easy_ss} (and Proposition \ref{reductions_for_ss}).
\end{rem}
\begin{prop}\label{3.15} Assume that $\mathcal F$ is a semi-simple object of $\textrm{\rm $F_{\sigma}$-Isoc}^{\dagger}(U/\mathbb L)$. Then $\textrm{\rm DGal}(\mathcal F,x)^o$ is semi-simple, and it is the derived group of $\textrm{\rm Gr}(\mathcal F,x)^o$.
\end{prop}
\begin{proof} The first claim is Corollary 4.10 of \cite{Cr} on page 457. With respect to the Zariski topology the group $\mathbf W(\mathcal F,x)$ is generated by the element which, as an automorphism of $\omega_x$, is the $q$-linear Frobenius $\textrm{Frob}_x(\mathcal C)$ for every object
$\mathcal C$ of $\dal\mathcal F\dar_{const}$. So it is commutative, and hence 
$[\textrm{\rm Gr}(\mathcal F,x),\textrm{\rm Gr}(\mathcal F,x)]\subseteq\textrm{\rm DGal}(\mathcal F,x)$. Therefore $[\textrm{\rm Gr}(\mathcal F,x)^o,\textrm{\rm Gr}(\mathcal F,x)^o]\subseteq\textrm{\rm DGal}(\mathcal F,x)^o$. In order to show the reverse inclusion it will be enough to show that $[\textrm{\rm DGal}(\mathcal F,x)^o,\textrm{\rm DGal}(\mathcal F,x)^o]=\textrm{\rm DGal}(\mathcal F,x)^o$. But this is true by the main result of \cite{Ree}, since $\textrm{\rm DGal}(\mathcal F,x)^o$ is semi-simple. 
\end{proof}
The following result will play an important role in the proofs of Proposition \ref{connected_components} and Theorem \ref{chin_independence}:
\begin{thm}\label{chebotarev} Assume that $\mathcal F$ is a semi-simple $\iota$-pure object of $\textrm{\rm $F_{\sigma}$-Isoc}^{\dagger}(U/\mathbb L)$. Then the set $\bigcup_{x\in|U|}\textrm{\rm Frob}_x(\mathcal F)$ is Zariski dense in $\textrm{\rm Gr}(\mathcal F,x)$.
\end{thm}
\begin{proof} This is a special case of Theorem 10.1 of \cite{HP}.
\end{proof}
It will be also useful to record the following
\begin{lemma}\label{finite_monodromy} Assume that $\mathcal F$ is an object of $\textrm{\rm $F_{\sigma}$-Isoc}^{\dagger}(U/\mathbb L)$ such that $\textrm{\rm Gr}(\mathcal F,x)$ is finite. Then $\mathcal F$ is unit-root.
\end{lemma}
\begin{proof} Recall that $\mathcal F$ is unit-root if for every $y\in|U|$ the eigenvalues of $\textrm{Frob}_y(\mathcal F)$ are $p$-adic units. By definition $\textrm{Frob}_y(\mathcal F)$ is an element of $\textrm{\rm Gr}(\mathcal F,y)$. Since the algebraic groups $\textrm{\rm Gr}(\mathcal F,x)$ and $\textrm{\rm Gr}(\mathcal F,y)$ are isomorphic over some finite extension of $\mathcal L$, we get that $\textrm{\rm Gr}(\mathcal F,y)$ is finite. Therefore the eigenvalues of $\textrm{Frob}_y(\mathcal F)$ are roots of unity, and hence the claim follows. 
\end{proof}
\begin{rem}\label{cc_remark} An important consequence of the above is the following. Let $\mathcal F$ be any object of $\textrm{\rm $F_{\sigma}$-Isoc}^{\dagger}(U/\mathbb L)$. Then there is an object $\mathcal C$ of $\dal\mathcal F\dar$ whose monodromy group $\textrm{\rm Gr}(\mathcal C,x)$ is the quotient $\textrm{\rm Gr}(\mathcal F,x)/\textrm{\rm Gr}(\mathcal F,x)^o$. Since the latter is finite we get from Lemma \ref{finite_monodromy} that $\mathcal C$ is unit-root. The Crew--Katz--Tsuzuki tensor equivalence (see Theorem 1.4 of \cite{Cr} on page 434 and Theorem 1.3.1 of \cite{Ts} on page 387) between the category of unit root $F$-isocrystals and $p$-adic Galois representations identities the monodromy group $\textrm{\rm Gr}(\mathcal C,x)$, and hence $\textrm{\rm Gr}(\mathcal F,x)/\textrm{\rm Gr}(\mathcal F,x)^o$, with a finite quotient of
$\pi_1(U,\overline x)$ where $x\in U(\overline{\mathbb F}_{q^n})$ is a point lying above $x$.
\end{rem}
\begin{lemma}\label{compatible_frobenii} Let $\mathcal F$ be an object of $\textrm{\rm $F_{\sigma}$-Isoc}^{\dagger}(U/\mathbb L)$ such that $\textrm{\rm Gr}(\mathcal F,x)$ is finite, and let $\rho:\pi_1(U,\overline x)\to\textrm{\rm Gr}(\mathcal F,x)(\overline{\mathbb L})$ be the corresponding surjection. Then for every $y\in|U|$ the image of $\textrm{\rm Frob}_y$ with respect to $\rho$ is the conjugacy class of the geometric Frobenius at $y$.
\end{lemma}
\begin{proof} It is enough to check the claim after pull-back to $y$, that is, we may assume that $U$ is a point without the loss of generality. In this case the claim follows from Corollary 4.6 of \cite{HP}.
\end{proof}
\begin{notn} Assume again that $X$ is geometrically connected, and fix a point $x\in X(\mathbb F_{q^n})$. Note that the functor
$$(\,\cdot\,)^{(n)}:\textrm{$F_{\sigma}$-Isoc}^{\dagger}(X/\mathbb L)\longrightarrow\textrm{$F^n_{\sigma^n}$-Isoc}^{\dagger}(X/\mathbb L)$$
induces a map 
$$\phi:\textrm{\rm Gr}(\textrm{$F^n_{\sigma^n}$-Isoc}^{\dagger}(X/\mathbb L),x)\longrightarrow\textrm{\rm Gr}(\textrm{$F_{\sigma}$-Isoc}^{\dagger}(X/\mathbb L),x)$$
of Tannakian fundamental groups. Let $\mathbf C_n$ be the full subcategory of 
$\textrm{$F_{\sigma}$-Isoc}^{\dagger}(X/\mathbb L)$ whose objects $\mathcal F$ are such that $\mathcal F^{(n)}$ is trivial; it is a strictly full rigid abelian tensor subcategory. Let
$$\psi:\textrm{\rm Gr}(\textrm{$F_{\sigma}$-Isoc}^{\dagger}(X/\mathbb L),x)
\longrightarrow\textrm{\rm Gr}(\mathbf C_n,x)$$
denote the map induced by the inclusion functor $\mathbf C_n\to\textrm{$F_{\sigma}$-Isoc}^{\dagger}(X/\mathbb L)$.
\end{notn}
\begin{prop}\label{n-power_sequence} Assume that $X$ is smooth. Then the sequence:
$$\xymatrix { 0 \to \textrm{\rm Gr}(\textrm{\rm $F^n_{\sigma^n}$-Isoc}^{\dagger}(X/\mathbb L),x) \ar[r]^{\ \phi} &
\textrm{\rm Gr}(\textrm{\rm $F_{\sigma}$-Isoc}^{\dagger}(X/\mathbb L),x)
\ar[r]^{\ \ \ \psi} & 
\textrm{\rm Gr}(\mathbf C_n,x)\to 0}$$
is exact, and $\textrm{\rm Gr}(\mathbf C_n,x)\cong\mathbb Z/n\mathbb Z$.
\end{prop}
\begin{proof} We will argue very similarly to the proof of Proposition \ref{crew_sequence}. By Proposition \ref{my_variant} we only have to check the following:
\begin{enumerate}
\item[$(i)$] For an object $\mathcal G$ of $\textrm{\rm $F_{\sigma}$-Isoc}^{\dagger}(X/\mathbb L)$ the object $\mathcal G^{(n)}$ of $\textrm{\rm $F^n_{\sigma^n}$-Isoc}^{\dagger}(X/\mathbb L)$ is trivial if and only if $\mathcal G$ is an object of $\mathbf C_n$.
\item[$(ii)$] Let $\mathcal G$ be an object of $\textrm{\rm $F_{\sigma}$-Isoc}^{\dagger}(X/\mathbb L)$, and let $\mathcal H_0\subseteq\mathcal G^{(n)}$ denote the largest trivial sub-object. Then there exists an $\mathcal H\subseteq\mathcal G$ with $\mathcal H_0=\mathcal H^{(n)}$.
\item[$(iii)$] Every object $\mathcal G$ of $\textrm{\rm $F^n_{\sigma^n}$-Isoc}^{\dagger}(X/\mathbb L)$ is a sub-object of an object of the form $\mathcal H^{(n)}$ with some object $\mathcal H$ of $\textrm{\rm $F_{\sigma}$-Isoc}^{\dagger}(X/\mathbb L)$. 
\end{enumerate}
Condition $(i)$ holds by definition. Next we show $(ii)$. We already saw in the proof Proposition \ref{crew_sequence} that the maximal trivial overconvergent sub-isocrystal $\mathcal H_0'$ of $\mathcal G^{\wedge}$ is underlies an overconvergent $F$-isocrystal $\mathcal G'\subseteq\mathcal G$. Since $\mathcal G'$ is generated by (overconvergent) horizontal sections, its sub $F$-isocrystal generated by those horizontal sections which are fixed by the $n$-th power of the Frobenius underlies $\mathcal H_0$. Now we prove $(iii)$. Let $\mathbf F_{\mathcal G}:(F^n_{\sigma})^*(\mathcal G)\to\mathcal G$ be the Frobenius of $\mathcal G$, let $\mathcal F$ be the direct sum:
$$\mathcal F=\mathcal G\oplus F_{\sigma}^*(\mathcal G)\oplus\cdots
\oplus (F^{n-1}_{\sigma})^*(\mathcal G)$$
in the category $\textrm{\rm Isoc}^{\dagger}(X/\mathbb L)$, and consider the map
$$\mathbf F_{\mathcal F}:F_{\sigma}^*(\mathcal F)\cong
F_{\sigma}^*(\mathcal G)\oplus(F^2_{\sigma})^*(\mathcal G)\oplus
\cdots\oplus (F^n_{\sigma})^*(\mathcal G)
\longrightarrow\mathcal F$$
given by the matrix:
$$\begin{bmatrix}
    0 & \textrm{id}_{F_{\sigma}^*(\mathcal G)} & 0 & \ldots
    & 0 \\
    0 & 0 & \textrm{id}_{(F^2_{\sigma})^*(\mathcal G)} & \ldots
    & 0 \\
  \vdots & \ddots & \ddots & \ddots
    & \vdots \\
    0 &  \ldots &  & \cdots
    & \textrm{id}_{(F^{n-1}_{\sigma})^*(\mathcal G)} \\
    \mathbf F_{\mathcal G} & 0 & \ldots
    &  & 0
  \end{bmatrix}.$$
Since $\mathbf F_{\mathcal F}$ is the linear combination of the composition of horizontal maps, it is horizontal, too. Moreover the map:
$$\mathbf F_{\mathcal F}^{[n]}\!=\!\mathbf F_{\mathcal F}\circ F_{\sigma}^*
(\mathbf F_{\mathcal F})\circ\cdots\circ (F_{\sigma}^{n-1})^*(\mathbf F_{\mathcal F})\!:\!(F^n_{\sigma})^*(\mathcal F)
\!\cong\!(F^n_{\sigma})^*(\mathcal G)\oplus
\cdots\oplus (F^{2n-1}_{\sigma})^*(\mathcal G)
\!\to\!\mathcal F$$ 
is given by the diagonal matrix
$$\begin{bmatrix}
    \mathbf F_{\mathcal G} & 0 & 0& \ldots
    & 0 \\
   0 & F_{\sigma}^*(\mathbf F_{\mathcal G}) & 0 & \ldots
    & 0 \\
  \vdots & \ddots & \ddots & \ddots
    & \vdots \\
    \vdots &  & \ddots & \ddots & \vdots \\
    0 & \ldots &     & 0 &(F^{n-1}_{\sigma})^*(\mathbf F_{\mathcal G})
  \end{bmatrix},$$
so it is an isomorphism. We get that $\mathcal F$ equipped with $\mathbf F_{\mathcal F}$ is an object of the category $\textrm{\rm $F_{\sigma}$-Isoc}^{\dagger}(X/\mathbb L)$, and $\mathcal G$ is a direct summand of $\mathcal F^{(n)}$. Claim $(iii)$ follows. 

Finally we prove that $\textrm{\rm Gr}(\mathbf C_n,x)\cong\mathbb Z/n\mathbb Z$. Let $\mathcal G$ be an arbitrary object of $\mathbf C_n$. Note that for every point $y:\textrm{Spec}(\mathbb F_{q^d})\to X$ of degree $d$ the pull-back $y^*(\mathcal G^{(n)})\cong y^*(\mathcal G)^{(n)}$ is trivial. Therefore the eigenvalues of $\textrm{Frob}_z(\mathcal G)$ for any closed point $z\in|X|$ are
$n$-th roots of unity, and hence $\mathcal G$ is unit-root. The claim now follows from the Crew--Katz--Tsuzuki tensor equivalence.
\end{proof}
\begin{notn} Let $\pi:Y\rightarrow X$ a finite, \'etale, Galois map of geometrically connected smooth schemes over $\mathbb F_q$ with Galois group $G$ and assume that there is a $y\in Y(\mathbb F_{q^n})$ such that $\pi(y)=x$. The pull-back functor:
$$\pi^*:\textrm{$F_{\sigma}$-Isoc}^{\dagger}(X/\mathbb L)\longrightarrow\textrm{$F_{\sigma}$-Isoc}^{\dagger}(Y/\mathbb L)$$
induces a map 
$$\rho:\textrm{\rm Gr}(\textrm{$F_{\sigma}$-Isoc}^{\dagger}(Y/\mathbb L),y)\longrightarrow\textrm{\rm Gr}(\textrm{$F_{\sigma}$-Isoc}^{\dagger}(X/\mathbb L),x)$$
of Tannakian fundamental groups. (For the definition of $\pi_*$ and $\pi^*$ in this setting see section 1.7 of \cite{Cr}.) Let $\mathbf C(\pi)$ be the full subcategory of 
$\textrm{$F_{\sigma}$-Isoc}^{\dagger}(X/\mathbb L)$ whose objects $\mathcal F$ are such that $\pi^*(\mathcal F)$ is trivial; it is a strictly full rigid abelian tensor subcategory. Let
$$\sigma:\textrm{\rm Gr}(\textrm{$F_{\sigma}$-Isoc}^{\dagger}(X/\mathbb L),x)
\longrightarrow\textrm{\rm Gr}(\mathbf C(\pi),x)$$
denote the map induced by the inclusion functor $\mathbf C(\pi)\to\textrm{$F_{\sigma}$-Isoc}^{\dagger}(X/\mathbb L)$.
\end{notn}
\begin{prop}\label{cover_sequence} The sequence:
$$\xymatrix { 0 \to \textrm{\rm Gr}(\textrm{\rm $F_{\sigma}$-Isoc}^{\dagger}(Y/\mathbb L),y) \ar[r]^{\ \rho} &
\textrm{\rm Gr}(\textrm{\rm $F_{\sigma}$-Isoc}^{\dagger}(X/\mathbb L),x)
\ar[r]^{\ \ \ \sigma} & 
\textrm{\rm Gr}(\mathbf C(\pi),x)\to 0}$$
is exact, and $\textrm{\rm Gr}(\mathbf C(\pi),x)\cong G$.
\end{prop}
\begin{proof} Our proof is again along the same line as the proofs for Propositions \ref{crew_sequence} and \ref{n-power_sequence}. First we prove that $\textrm{\rm Gr}(\mathbf C(\pi),x)\cong G$. Let $\mathcal G$ be an arbitrary object of $\mathbf C_n$. Note that for every closed point $z\in|X|$ there is a close point $\widetilde z\in|Y|$ such that $\pi(\widetilde z)=z$. Then some power of the eigenvalues of
$\textrm{Frob}_z(\mathcal G)$ with positive exponent are all eigenvalues of
$\textrm{Frob}_{\widetilde z}(\pi^*(\mathcal G))$. Since $\textrm{Frob}_{\widetilde z}(\pi^*(\mathcal G))$ is the identity we get that the eigenvalues of $\textrm{Frob}_z(\mathcal G)$ are roots of unity, and hence $\mathcal G$ is unit-root. The claim now follows from the Crew--Katz--Tsuzuki tensor equivalence and its commutativity with the pull-back functor. Since $\mathbf C(\pi)$ is a subcategory of $\textrm{\rm $F_{\sigma}$-Isoc}^{\dagger}(X/\mathbb L)$, we only have to check the following:
\begin{enumerate}
\item[$(i)$] For an object $\mathcal G$ of $\textrm{\rm $F_{\sigma}$-Isoc}^{\dagger}(X/\mathbb L)$ the object $\pi^*(\mathcal G)$ of $\textrm{\rm $F_{\sigma}$-Isoc}^{\dagger}(Y/\mathbb L)$ is trivial if and only if $\mathcal G$ is an object of $\mathbf C(\pi)$.
\item[$(ii)$] Let $\mathcal G$ be an object of $\textrm{\rm $F_{\sigma}$-Isoc}^{\dagger}(X/\mathbb L)$, and let $\mathcal H_0\subseteq\pi^*(\mathcal G)$ denote the largest trivial sub-object. Then there exists $\mathcal H\subseteq\mathcal G$ with $\mathcal H_0=\pi^*(\mathcal H)$.
\item[$(iii)$] Every object $\mathcal G$ of $\textrm{\rm $F_{\sigma}$-Isoc}^{\dagger}(Y/\mathbb L)$ is a sub-object of an object of the form
$\pi^*(\mathcal H)$ with some object $\mathcal H$ of $\textrm{\rm $F_{\sigma}$-Isoc}^{\dagger}(X/\mathbb L)$. 
\end{enumerate}
Condition $(i)$ holds by definition. Since for every $\mathcal G$ of $\textrm{\rm $F_{\sigma}$-Isoc}^{\dagger}(Y/\mathbb L)$ the adjunction map $\mathcal G\to
\pi_*(\pi^*(\mathcal G))$ is injective, claim $(iii)$ also follows. Finally we show $(ii)$. As a pull-back, the $F$-isocrystal $\pi^*(\mathcal G)$ is equipped with a descent data with respect to $\pi$. For every $g\in G$ let
$\iota_g:g^*(\pi^*(\mathcal G))\to\pi^*(\mathcal G)$ be the isomorphism in this descent data. For every $g\in G$ the $F$-isocrystal $g^*(\mathcal H_0)$ is trivial, and hence so its image under $\iota_g$. Therefore $\iota_g(g^*(\mathcal H_0))\subseteq\mathcal H_0$. As the ranks of these isocrystals are the same, we get that the restriction of $\iota_g$ onto $g^*(\mathcal H_0)$ furnishes an isomorphism $\iota_g|_{g^*(\mathcal H_0)}:g^*(\mathcal H_0)\to\mathcal H_0$. Since these maps are restrictions of a descent data, they also satisfy the same cocycle condition, and hence $\{\iota_g|_{g^*(\mathcal H_0)}:g^*(\mathcal H_0)\to\mathcal H_0|g\in G\}$ is a descent data on $\mathcal H_0$ with respect to $\pi$. Therefore by Th\'eor\`eme 1 of \cite{Et1} on page 593 there is an $F$-isocrystal $\mathcal H\subseteq\mathcal G$ such that $\mathcal H_0=\pi^*(\mathcal H)$, so claim $(iii)$ holds.
\end{proof}
\begin{cor}\label{4.20} Let $\pi:V\to U$ be a finite \'etale cover of geometrically connected smooth schemes over $\mathbb F_q$, and let $\mathcal F$ be an object of $\textrm{\rm $F_{\sigma}$-Isoc}^{\dagger}(U/\mathbb L)$. Then for every point $x:\textrm{\rm Spec}(\mathbb F_{q^n})\to V$ of degree $n$ the homomorphism $\textrm{\rm Gr}(\pi^*(\mathcal F),x)\to\textrm{\rm Gr}(\mathcal F,\pi(x))$ induced by pull-back with respect to $\pi$ is an open immersion.
\end{cor}
\begin{proof} Let $\rho:W\to U$ be a finite, \'etale Galois cover which factorises as the composition of a finite, \'etale map $\sigma:W\to V$ and $\pi$. Then $\sigma$ is Galois, too. Fix a point $y:\textrm{Spec}(\mathbb F_{q^n})\to W$ of degree $n$. The homomorphisms $\textrm{Gr}(\rho^*(\mathcal F),y)^o
\to\textrm{Gr}(\mathcal F,\rho(y))^o$ and 
$\textrm{Gr}(\rho^*(\mathcal F),y)^o=
\textrm{Gr}(\sigma^*(\pi^*(\mathcal F)),y)^o
\to\textrm{Gr}(\pi^*(\mathcal F),\sigma(y))^o$ induced by
$\rho$ and $\sigma$, respectively, are isomorphisms by Proposition \ref{cover_sequence}. So the same holds for the homomorphism $\textrm{Gr}(\pi^*(\mathcal F),\sigma(y))^o \to\textrm{Gr}(\mathcal F,\rho(y))^o$ induced by $\pi$, too. Since it is enough to prove the claim for just one point $x$, the claim follows.
\end{proof}
\begin{cor}\label{4.20b} Let $\pi:V\to U$ and $\mathcal F$ be as above. If $\mathcal F$ is semi-simple, then so is
$\pi^*(\mathcal F)$.
\end{cor}
\begin{proof} Since $\pi^*(\mathcal F^{(n)})\cong
\pi^*(\mathcal F)^{(n)}$, we may assume that $x$ above has degree one by switching to $\mathcal F^{(n)}$ and using part $(i)$ of Proposition \ref{reductions_for_ss}. By assumption $\textrm{Gr}(\mathcal F,\pi(x))$ is reductive, therefore its open subgroup 
$\textrm{Gr}(\pi^*(\mathcal F),x)$ is also reductive. The claim now follows from Tannakian duality.
\end{proof}
Note that the pull-back functor:
$$\pi^*:\textrm{$F_{\sigma}$-Isoc}^{\dagger}(U/\mathbb L)\longrightarrow\textrm{$F_{\sigma}$-Isoc}^{\dagger}(V/\mathbb L)$$
also induces a map 
$$\textrm{\rm DGal}(\pi^*(\mathcal F),x)\longrightarrow\textrm{\rm DGal}(\mathcal F,\pi(x))$$
of Tannakian fundamental groups for every point $x:\textrm{\rm Spec}(\mathbb F_{q^n})\to V$ of degree $n$. 
\begin{cor}\label{4.21} Assume in addition that $\mathcal F$ is semi-simple. Then the homomorphism $\textrm{\rm DGal}(\pi^*(\mathcal F),x)\longrightarrow\textrm{\rm DGal}(\mathcal F,\pi(x))$ is an open immersion.
\end{cor}
\begin{proof} Since $\pi^*(\mathcal F)$ is semi-simple by Corollary \ref{4.20b}, this follows at once from Corollary \ref{4.20} and Proposition \ref{3.15} applied to both $\mathcal F$ and
$\pi^*(\mathcal F)$.
\end{proof}
Let $\mathbb L'/\mathbb L$ be a finite totally ramified extension. In this case the identity map $\sigma':\mathbb L'\to\mathbb L'$ is still a lift of the $q$-power automorphism of $k$, so there is a functor $\cdot\otimes_{\mathbb L}\mathbb L':\textrm{$F_{\sigma}$-Isoc}^{\dagger}(X/\mathbb L)\to\textrm{$F_{\sigma'}$-Isoc}^{\dagger}(X/\mathbb L')$. 
\begin{lemma}\label{coefficients} Let $X$ be a geometrically irreducible smooth scheme over $\mathbb F_q$, let $\mathcal F$ be an object of $\textrm{\rm $F_{\sigma}$-Isoc}^{\dagger}(X/\mathbb L)$ and let $x$ be a closed point of $X$ of degree $n$. Then we have: $\textrm{\rm Gr}(\mathcal F
\otimes_{\mathbb L}\mathbb L',x)\cong\textrm{\rm Gr}(\mathcal F,x)\otimes_{\mathbb L_n}\mathbb L'_n$. If we also assume that $\mathcal F^{\wedge}$ is semi-simple, then we have: $\textrm{\rm DGal}(\mathcal F
\otimes_{\mathbb L}\mathbb L',x)\cong
\textrm{\rm DGal}(\mathcal F,x)\otimes_{\mathbb L_n}\mathbb L'_n$.
\end{lemma}
\begin{proof} Let $\pi:U\to X$ be an open affine sub-scheme containing $x$. Then by Theorem \ref{extension_thm} the pull-back map $\pi^*$ induces a pair of isomorphisms $\textrm{\rm Gr}(\mathcal F,x)\cong\textrm{\rm Gr}(\pi^*(\mathcal F),x)$ and
$\textrm{\rm Gr}(\mathcal F
\otimes_{\mathbb L}\mathbb L',x)\cong\textrm{\rm Gr}(\pi^*(\mathcal F
\otimes_{\mathbb L}\mathbb L'),x)$. Moreover $\pi^*$ induces an $\mathbb L$-linear, respectively $\mathbb L'$-linear tensor equivalence between $\dal\mathcal F\dar_{const}$ and $\dal\pi^*(\mathcal F)\dar_{const}$, respectively between $\dal
\mathcal F\otimes_{\mathbb L}\mathbb L'\dar_{const}$ and $\dal\pi^*(\mathcal F\otimes_{\mathbb L}\mathbb L')\dar_{const}$, and hence induces a pair of isomorphisms $\textrm{\rm DGal}(\mathcal F,x)\cong\textrm{\rm DGal}(\pi^*(\mathcal F),x)$ and
$\textrm{\rm DGal}(\mathcal F
\otimes_{\mathbb L}\mathbb L',x)\cong\textrm{\rm DGal}(\pi^*(\mathcal F
\otimes_{\mathbb L}\mathbb L'),x)$ by Proposition
\ref{crew_sequence} when $\mathcal F^{\wedge}$ is semi-simple.

Therefore we may assume without the loss of generality that $X$ is affine, and hence it has a compactification $Y$ which has a formal lift $\mathfrak Y$ to $\textrm{\rm Spf}(\mathcal O_{\mathbb L})$ smooth in the neighbourhood of $X$ which can be equipped with a lift of the $q$-power Frobenius compatible with $\sigma$. Let 
$\mathfrak Y'$ be the base change of $\mathfrak Y$ to Spf$(\mathcal O_{\mathbb L'})$. The smooth and proper frames $X\subseteq Y\hookrightarrow\mathfrak Y$ and $X\subseteq Y
\hookrightarrow\mathfrak Y'$ determine dagger algebras $A,A'$ over $\mathbb L$ and over $\mathbb L'$, respectively. Clearly $A'=A\otimes_{\mathbb L}\mathbb L'$. By assumption there is a lift of the $q$-power Frobenius $F_{\sigma}:A\to A$ compatible with $\sigma$, and the unique $\sigma'$-linear extension $F_{\sigma'}:A'\to A'$ of $F_{\sigma}$ is also a lift of the $q$-power Frobenius.

As in the proof of Proposition \ref{my_variant} above, for any $\mathbb L$-linear Tannakian category $\mathbf T$ let $\mathbf T_{(\mathbb L')}$ denote the category of $\mathbb L'$-modules in $\mathbf T$. Using the construction in Proposition 3.11 of \cite{DM} we may attach an $\mathbb L_n'$-valued fibre functor $\omega_x\otimes_{\mathbb L_n}\mathbb L_n'$ on $\textrm{\rm $F_{\sigma}$-Isoc}^{\dagger}(X/\mathbb L)_{(\mathbb L')}$ to the $\mathbb L_n$-linear fibre functor $\omega_x$ on $\textrm{\rm $F_{\sigma}$-Isoc}^{\dagger}(X/\mathbb L)$. In order to show the first claim it will be enough to show that there is an $\mathbb L'$-linear tensor-equivalence $\epsilon:\textrm{\rm $F_{\sigma}$-Isoc}^{\dagger}(X/\mathbb L')\to\textrm{\rm $F_{\sigma}$-Isoc}^{\dagger}(X/\mathbb L)_{(\mathbb L')}$ such that the composition $\omega_x\otimes_{\mathbb L_n}\mathbb L_n'\circ\epsilon$ is just the fibre functor at $x$. 

Recall that $\textrm{\rm $F_{\sigma}$-Isoc}^{\dagger}(X/\mathbb L)$ and $\textrm{\rm $F_{\sigma}$-Isoc}^{\dagger}(X/\mathbb L')$ is the category of integrable $(F_{\sigma},\nabla)$-modules over $A$, and the category of integrable $(F_{\sigma'},\nabla)$-modules over $A'$, respectively. Under these identifications
$\epsilon$ is just the factor forgetful functor attaching to an integrable $(F_{\sigma'},\nabla)$-module over $A'$ the underlying
$(F_{\sigma},\nabla)$-module over $A$, and using the $\mathbb L'$-module structure to define the $\mathbb L'$-multiplication. The first claim is now clear. Note that $\epsilon$ induces an equivalence between $\dal\mathcal F\otimes_{\mathbb L}\mathbb L'\dar_{const}$ and $(\dal\mathcal F\dar_{const})_{(\mathbb L')}$, so the second claim follows from Proposition
\ref{crew_sequence} when $\mathcal F^{\wedge}$ is semi-simple.
\end{proof}
\begin{prop}\label{geometric} Let $\mathcal F$ be a semi-simple object of $\textrm{\rm $F_{\sigma}$-Isoc}^{\dagger}(U/\mathbb L)$. Then there is a geometrically connected finite \'etale cover $\pi:V\to U$ such that $\textrm{\rm DGal}(\pi^*(\mathcal F),x)$ is connected for any closed point $x$ of $V$.
\end{prop}
\begin{rem} It would be interesting to see a proof of a similar claim for overconvergent isocrystals, without using the Frobenius, but I do not know how to do this. In particular is it true that an overconvergent isocrystal with finite monodromy can be trivialised with a finite, \'etale map?
\end{rem}
\begin{proof} First note that we only need to prove a similar claim $\mathcal F^{(d)}$ and $U^{(d)}$ for some positive integer $d$. Indeed let $\pi':V'\to U^{(d)}$ be a geometrically connected finite \'etale cover such that $\textrm{\rm DGal}(\pi^*(\mathcal F),x)$ is connected for any closed point $x$ of $V$. Because the \'etale fundamental group of the base change of $U$ to $\overline{\mathbb F}_q$ is topologically finitely generated, its open characteristic subgroups are cofinal, and hence there is a geometrically connected finite \'etale map $\pi:V\to U$ such that its base change $\pi^{(d)}:V^{(d)}\to U^{(d)}$ to $\mathbb F_{q^d}$ can be factorised as the composition of a finite \'etale map $\rho:V^{(d)}\to V'$ and $\pi'$. By Corollary \ref{4.21} the group $\textrm{\rm DGal}(\pi^*(\mathcal F^{(d)}),x)$ is connected for any closed point $x$ of $V$. Since $\pi^*(\mathcal F^{(d)})\cong\pi^*(\mathcal F)^{(d)}$, we get that the same holds for $\mathcal F$, too.

By the above we may assume that $U$ has a degree one point $x$. Since $\textrm{\rm DGal}(\mathcal F,x)^o$ is an open characteristic subgroup of $\textrm{\rm DGal}(\mathcal F,x)$, we get that $\textrm{\rm DGal}(\mathcal F,x)^o$ is a closed normal subgroup of $\textrm{\rm Gr}(\mathcal F,x)$. By Tannaka duality there is an object of $\dal\mathcal F\dar$ whose monodromy group is the quotient $\textrm{\rm Gr}(\mathcal F,x)/\textrm{\rm DGal}(\mathcal F,x)^o$. So we may assume without the loss of generality that $\textrm{\rm DGal}(\mathcal F,y)$ is finite for any closed point $y$ of $U$. By Corollary \ref{4.20} the same will hold for $\pi^*(\mathcal F)$ where $\pi:V\to U$ is any geometrically connected finite \'etale cover. Therefore we may assume that $\textrm{\rm Gr}(\mathcal F,x)$ is connected, by taking a suitable finite \'etale cover $\pi:V\to U$ and using Proposition \ref{cover_sequence}. 

Since $\textrm{\rm Gr}(\mathcal F,x)$ is connected, its derived group is also connected. Therefore by Proposition \ref{3.15} this group is trivial, so the reductive group $\textrm{\rm Gr}(\mathcal F,x)$ must be a torus. After switching to $\mathcal F^{(d)}$ for a suitable $d$, then taking a suitable finite extension $\mathbb L'/\mathbb L$ such that the identity map $\sigma':\mathbb L'\to\mathbb L'$ is a lift of the $q^d$-power automorphism of $k$, and switching to $\mathcal F^{(d)}\otimes_{\mathbb L}\mathbb L'$ and using Lemma \ref{coefficients} (recall that $\mathcal F^{\wedge}$ is semi-simple by Proposition \ref{3.9}), we may even assume that $\textrm{\rm Gr}(\mathcal F,x)$ is a split torus. Using Tannaka duality we get that without the loss of generality we may assume that $\mathcal F=\mathcal F_1\oplus\mathcal F_2\oplus\cdots\oplus\mathcal F_m$, where $\mathcal F_i$ is a rank one object of $\textrm{$F_{\sigma}$-Isoc}^{\dagger}(U/\mathbb L)$ for every index $i$. Note that it will be enough to show the proposition for each $\mathcal F_i$ individually. In fact the fibre product of the covers which we have constructed for the $\mathcal F_i$ will have the requited property for each index $i$ (by Corollary \ref{4.20}), and hence for $\mathcal F$, too. So we may assume without the loss of generality that $\mathcal F$ has rank one. 

Let $x:\textrm{Spec}(\mathbb F_q)\to U$ be a closed point of degree one, and let $\mathcal F_x$ be the fibre of $\mathcal F$ over $x$. Let $c:U\to\textrm{Spec}(\mathbb F_q)$ be the unique map, and let $\mathcal H$ be the pull-back of $\mathcal F_x$ onto $U$. If $\mathcal H^{\wedge}$ denotes the dual of $\mathcal H$, then $\mathcal G=\mathcal F\otimes\mathcal H^{\wedge}$ is a rank one object of $\textrm{$F_{\sigma}$-Isoc}^{\dagger}(U/\mathbb L)$ whose fibre over $x$ is trivial. Therefore by the Grothendieck-Katz theorem (see \cite{Ka}) the $F$-isocrystal $\mathcal G$ is unit-root. Therefore by the Crew--Katz--Tsuzuki tensor equivalence and  Proposition \ref{cover_sequence} there is a geometrically connected finite \'etale cover $\pi:V\to U$ such that $\pi^*(\mathcal G)$ is a constant unit-root $F$-isocrystal. Since $\mathcal F^{\vee}\cong\mathcal G^{\vee}$, we get that $\pi(\mathcal F)^{\vee}\cong\pi(\mathcal F^{\vee})\cong\pi^*(\mathcal G^{\vee})
\cong\pi^*(\mathcal G)^{\vee}$ is also trivial, and the claim follows. 
\end{proof}

\section{The isogeny and semi-simplicity conjectures via $p$-divisible groups}

We are going to use the notation introduced in the introduction and in the previous section.
\begin{proof}[Proof of Theorem \ref{isogeny}] In the commutative diagram below we let $\beta$ and $\gamma$ be induced by the functoriality of convergent Dieudonn\'e modules for abelian varieties and for $p$-divisible groups, respectively, the map $\psi$ is furnished by the forgetful map from the category of overconvergent $F$-isocrystals into the category of convergent $F$-isocrystals, the map
$\phi$ is induced by the functoriality of taking the $p$-divisible group of abelian varieties, while $\rho$ is furnished by base change.
$$\xymatrixcolsep{5pc}\xymatrix{
\textrm{Hom}(A,B)\otimes\mathbb Q_p
\ar[d]^{\phi} \ar[rdd]^{\beta} \ar[rd]^{\alpha} & \\
\textrm{Hom}(A[p^{\infty}],B[p^{\infty}])\otimes_{\mathbb Z_p}\mathbb Q_p
\ar[d]^{\rho} \ar[rd]^{\gamma} &
\textrm{Hom}(D^{\dagger}(A),D^{\dagger}(B))\ar[d]^{\psi} \\
\textrm{Hom}(A_L[p^{\infty}],B_L[p^{\infty}])\otimes_{\mathbb Z_p}\mathbb Q_p&
\textrm{Hom}(D(A),D(B))}$$
By de Jong's theorem (Theorem 2.6 of \cite{J2} on page 305) the composition $\rho\circ\phi$ is an isomorphism. Since the map $\rho$ is injective, we get that $\phi$ is an isomorphism, too. By another theorem of de Jong (Main Theorem 2 of \cite{J1} on page 6) the map $\gamma$ is bijective. Therefore the composition $\beta$ is an isomorphism, too. By the global version of Kedlaya's full faithfulness theorem (see \cite{Ke1}) the map $\psi$ is also an isomorphism. Therefore the map $\alpha$ must be an isomorphism, too. 
\end{proof}
\begin{lemma}\label{descent} Let $\pi:X\rightarrow Y$ a finite, \'etale, Galois map of smooth schemes over $\mathbb F_q$ with Galois group $G$. Let $\mathcal F$ be an object of the category $\textrm{\rm $F_{\sigma}$-Isoc}^{\dagger}(Y/\mathbb L)$ and let $s$ be a horizontal section of $\pi^*(\mathcal F)$ which is invariant with respect to the natural $G$-action. Then $s$ is the pull-back of a horizontal section of $\mathcal F$ with respect to $f$.
\end{lemma}
\begin{proof} There is a short exact sequence:
$$\CD
0@>>>\mathcal F@>>>{\pi_*(\pi^*(\mathcal F))}@>>>\pi_*(\pi^*(\mathcal F))/\mathcal F@>>>0
\endCD$$
of overconvergent $F$-isocrystals which gives rise to the cohomological exact sequence:
$$\CD
0@>>>H^0_{rig}(Y,\mathcal F)@>{\pi_*\circ\pi^*}>>H^0_{rig}(Y,\pi_*(\pi^*(\mathcal F)))@>{\gamma}>>H^0_{rig}(Y,\pi_*(\pi^*(\mathcal F))/\mathcal F).
\endCD$$
Since the map $\pi_*:H^0_{rig}(X,\pi^*(\mathcal F))\rightarrow H^0_{rig}(Y,\pi_*(\pi^*(\mathcal F)))$ is injective, it will be enough to show that the image $\pi_*(s)$ under the map $\gamma$ is zero. It is enough to verify that
$\gamma(\pi_s(s))$ vanishes in the fibre of $\pi^*(\mathcal F))/\mathcal F$ over each closed point of $Y$. Therefore by the naturality of the maps $\pi_*$ and $\pi^*$ we may reduce to the situation when $X$ and $Y$ are zero-dimensional. In this case the claim is obvious.
\end{proof}
For every function field $K$ of transcendence degree one over $\mathbb F_q$ let $\mathcal C_K$ denote the unique irreducible smooth projective curve whose function field is $K$. When $K=L$ we will sometimes let $\mathcal C$ denote $\mathcal C_L$, for the sake of avoiding overburdened notation.  
\begin{lemma}\label{basechange} We may assume that $A_L$ has semi-stable reduction over $\mathcal C_L$ in the course of the proof of Theorem \ref{semisimple} without the loss of generality.
\end{lemma}
\begin{proof} By the semi-stable reduction theorem there is a finite separable extension $K$ of $L$ such that $A_K$ has semi-stable reduction over $\mathcal C_K$. We may assume that $K/L$ is a finite Galois extension, since the base change $A_J$ of $A_K$ will have semi-stable reduction over $\mathcal C_J$ for any finite separable extension $J$ of $K$, too. Let $G$ denote the Galois group of $K/L$. Let $\pi:\mathcal C_K\rightarrow\mathcal C_L$ be the map corresponding to the extension $K/L$ and let $V=\pi^{-1}(U)$. Let $A'$ denote the base change of $A$ to $V$. By the naturality of the Dieudonn\'e functor we have $D^{\dagger}(A')=\pi^*(D^{\dagger}(A))$. Let $\mathcal F\subseteq D^{\dagger}(A)$ be a sub $F$-isocrystal. There is a projection operator $P:D^{\dagger}(A')\rightarrow D^{\dagger}(A')$ with image $\pi^*(\mathcal F)$ by assumption. Moreover there is a natural action $G$ on $D^{\dagger}(A')$ covering the Galois action on $\mathcal C_K$. Since this action leaves $\pi^*(\mathcal F)$ invariant, we may argue similarly to the proof of Proposition \ref{reductions_for_ss} to conclude that the operator $P'$ which we get from $P$ by averaging out the action of $G$ will be a projection operator with image $\pi^*(\mathcal F)$, too. Since $P'$ is invariant under the action of $G$, by Lemma \ref{descent} applied to the overconvergent $F$-isocrystal $\textrm{End}(D^{\dagger}(A))$ and the map $\pi$ we get that $P$ is the pull-back of an endomorphism of $D^{\dagger}(A)$ with respect to $\pi$. This must be a projection operator with kernel $\mathcal F$.
\end{proof}
\begin{defn} Recall that for every function field $K$ of transcendence degree one over $\mathbb F_q$ the set of closed points $|\mathcal C_K|$ of $\mathcal C_K$ is in a natural bijection with the set $|K|$ of places of $K$. We say that a $p$-divisible subgroup $G$ over $U$ has semi-stable reduction if for every $x\in|\mathcal C_L|$ in the complement of $U$ the base change of $G$ to $L_x$ has semi-stable reduction. Similarly we say that an abelian scheme $A$ over $U$ has semi-stable reduction such that for every $x\in|\mathcal C_L|$ in the complement of $U$ the base change of $A$ to $L_x$ has semi-stable reduction.
\end{defn}
\begin{prop}\label{extension} Let $A$ be an abelian scheme over $U$ with semi-stable reduction, and let $\mathcal F\subseteq D^{\dagger}(A)$ be a overconvergent sub $F$-isocrystal. Then there is a $p$-divisible subgroup $G\subset A[p^{\infty}]$ which has semi-stable reduction such that $D(G)=
\mathcal F^{\sim}$.
\end{prop}
\begin{proof} Note that there is a sub Dieudonn\'e-module $\mathbf F\subseteq
\mathbf D(A)$ such that $\mathbf F\otimes\mathbb Q_q=\mathcal F^{\sim}$. Because the Dieudonn\'e functor is an equivalence of categories over regular schemes which are of finite type over a finite field (see Main Theorem 1 of \cite{J1} on page 6) there is a $p$-divisible subgroup $G\subset A[p^{\infty}]$ such that $\mathbf D(G)=\mathbf F$, and hence $D(G)=\mathcal F^{\sim}$. We only need to show that $G$ has semi-stable reduction. By 2.5 of \cite{J2} on page 304 the base change $H=A_{L_x}[p^{\infty}]$ of the $p$-divisible group $A[p^{\infty}]$ to $L_x$ has semi-stable reduction, for every $x\in|\mathcal C_L|$ in the complement of $U$. The claim now follows from Theorem \ref{overconvergent-condition}.
\end{proof}
For every abelian scheme $A$ over $U$ we are going to introduce a vector bundle $\omega_A$ as follows. Let $\mathcal A$ be the N\'eron model of $A_L$ over $\mathcal C_L$ extending $A$. Let $\omega_A$ denote the $\mathcal O_{\mathcal C_L}$-dual of the pull-back of the sheaf of K\"ahler differentials $\Omega^1_{\mathcal A/\mathcal C_L}$ with respect to the identity section $\mathcal C_L\rightarrow\mathcal A$.
\begin{thm}\label{height} Let $A$ be an abelian scheme over $U$ with semi-stable reduction, and let $G\subset A[p^{\infty}]$ be a $p$-divisible subgroup which has semi-stable reduction. For every $n\in\mathbb N$ let $B_n=A/G[p^n]$ be the abelian scheme that is the quotient of $A$ by the $p^n$-torsion of $G$. Then the line bundles $\det(\omega_{B_n})$ are all isomorphic.
\end{thm}
\begin{proof} First we are going to construct a generalisation of the sheaf $\omega_A$ for $p$-divisible subgroups with semi-stable reduction, following de Jong. For every place $v$ of $L$ not in $U$ let $R_v$ denote the complete local ring of $\mathcal C_L$ at $v$. For every such $v$ and for every scheme or ind-scheme $S$ over $U$ let $S_v$ denote the base change of $S$ to Spec$(L_v)$. Let $H$ be a $p$-divisible group over $U$ with semi-stable reduction. For every $v$ as above let
$$0\subseteq H_v^{\mu}\subseteq H_v^f\subseteq H_v$$
be a filtration by $p$-divisible groups of the type considered in Definition \ref{filtrationdef} and let $H_{v,1}$ denote the unique $p$-divisible group over Spec$(R_v)$ extending $H_v^f$. Let $\omega_H$ denote the coherent sheaf on $\mathcal C_L$ which we get by gluing Lie$(H)$, a coherent sheaf on $U$, with the sheaves Lie$(H_{v,1})$, coherent sheaves on Spec$(R_v)$, over Spec$(L_v)$ for every $v$ as above. This is possible because the map $\textrm{Lie}(H_v^f)\to\textrm{Lie}(H_v)$ induced by the inclusion $H_v^f\subseteq H_v$ is an isomorphism, since the quotient $H_v/H_v^f$ is \'etale. It is easy to see that
$\omega_H$ is independent of any choices made and for every abelian scheme $A$ over $U$ with semi-stable reduction we have $\omega_A=\omega_{A[p^{\infty}]}$. 

Now let us start the proof of the theorem; we will essentially repeat the argument in \cite{J2}, Theorem 2.6. Note that we have an exact sequence of truncated Barsotti-Tate group schemes of level $1$ over $U$ as follows:
$$\CD0@>>>G[p]@>>>A[p]@>>>B_n[p]@>>>G[p]@>>>0,\endCD$$
so there is an exact sequence
$$\CD0@>>>\textrm{Lie}(G)@>>>\textrm{Lie}(A)
@>>>\textrm{Lie}(B_n)@>>>\textrm{Lie}(G)@>>>0\endCD$$
of coherent $\mathcal O_U$-sheaves. In order to avoid overloading the notation let
$\mathbf A,\mathbf B_n$ denote the $p$-divisible group $A[p^{\infty}],B_n[p^{\infty}]$ of $A$ and $B_n$, respectively. For every $v$ as above choose a filtration:
\begin{equation}\label{hahasok}
0\subseteq\mathbf A_v^{\mu}\subseteq\mathbf A_v^f\subseteq\mathbf A_v
\end{equation}
by $p$-divisible groups of the type considered in Definition \ref{filtrationdef}. As we noted in the proof of Theorem \ref{overconvergent-condition} the filtration:
$$0\subseteq G_v\cap\mathbf A_v^{\mu}\subseteq G_v\cap\mathbf A_v^f
\subseteq G_v\cap\mathbf A_v$$
which we get by scheme-theoretical intersection is also of the type considered in Definition \ref{filtrationdef}. Moreover the image of the filtration (\ref{hahasok}) under the map
$\mathbf A\to\mathbf B_n$ induced by the isogeny $A\to B_n$ is a filtration:
$$0\subseteq(\mathbf B_n)_v^{\mu}\subseteq(\mathbf B_n)_v^f
\subseteq(\mathbf B_n)_v$$
by $p$-divisible groups of the type considered in Definition \ref{filtrationdef}. Let $G_{v,1},
\mathbf A_{v,1},\mathbf B_{n,v,1}$ denote the unique $p$-divisible group over Spec$(R_v)$ extending $G_v\cap\mathbf A_v^f,\mathbf A_v^f,(\mathbf B_n)_v^f$, respectively. By construction there is an exact sequence of truncated Barsotti-Tate group schemes of level $1$ over Spec$(R_v)$ as follows:
$$\CD0\to G_{v,1}[p]@>>>\mathbf A_{v,1}[p]@>>>\mathbf B_{n,v,1}[p]
@>>>G_{v,1}[p]\to0,\endCD$$
so there is an exact sequence
$$\CD0\to\textrm{Lie}(G_{v,1})@>>>\textrm{Lie}(\mathbf A_{v,1})
@>>>\textrm{Lie}(\mathbf B_{n,v,1})@>>>\textrm{Lie}(G_{v,1})\to0\endCD$$
of coherent sheaves on Spec$(R_v)$. By patching these exact sequences of coherent sheaves together we get that there is an exact sequence
$$\CD0@>>>\omega_G@>>>\omega_A
@>>>\omega_{B_n}@>>>\omega_G@>>>0\endCD$$
of coherent $\mathcal O_{\mathcal C_L}$-sheaves. The latter implies that det$(\omega_{B_n})\cong\det(\omega_A)$ is independent of $n$, and so the theorem follows.
\end{proof}
As explained in Section 5 of \cite{Fa2} that the theorem above has the following
\begin{cor}\label{finiteness} Let $A,G$ and $B_n=A/G[p^n]$ be as above. Then there is an infinite set $S\subseteq\mathbb N$ such that the abelian varieties $\{B_n|n\in S\}$ are all isomorphic.\qed
\end{cor}
\begin{proof}[Proof of Theorem \ref{semisimple}] This argument is essentially the same as Faltings's and Zar\-hin's, so we include it for the reader's convenience. We may assume that $A_L$ has semi-stable reduction over $\mathcal C_L$ without the loss of generality by Lemma \ref{basechange}. Let $\mathcal F\subseteq D^{\dagger}(A)$ be an  overconvergent sub $F$-isocrystal. We only need to show that there is an endomorphism of $D^{\dagger}(A)$ whose kernel is $\mathcal F$. There is a semi-stable $p$-divisible subgroup $G\subset A[p^{\infty}]$ such that $D(G)=\mathcal F^{\sim}$ by Proposition \ref{extension}. For every $n\in\mathbb N$ let $B_n=A/G[p^n]$ be the abelian variety that is the quotient of $A$ by the $p^n$-torsion of $G$. Let $f_n:A\longrightarrow B_n$ be the quotient map for every $n$. Then there is an infinite set $S\subseteq\mathbb N$ such that the abelian varieties $\{B_n|n\in S\}$ are all isomorphic by Corollary \ref{finiteness}. Let $i$ be the smallest element of $S$. For each $n\in S$ choose an isomorphism $v_n:B_n\rightarrow B_i$. Because $f_i$ is an isogeny the composition $u_n=f_i^{-1}
\circ v_n\circ f_n$ is well-defined as an element of End$(A)\otimes\mathbb Q$, and hence as an element of End$(A)\otimes\mathbb Q_p$. Choose a positive integer $d$ such that $U(\mathbb F_{q^d})\neq\emptyset$ and fix an $\mathbb F_{q^d}$-rational point $x\in U(\mathbb F_{q^d})$. For every abelian scheme $C$ over $U$ let $C_x$ denote the fibre of $C$ over $x$. By the above for every $n\in S$ the map $u_n$ induces an endomorphism of the $F$-isocrystal $D(A_x)$ over $\mathbb F_{q^d}$ and the image of the $\mathbb Z_{q^d}$-lattice $\mathbf D(A_x)\subset D(A_x)$ under this homomorphism lies in $(f_i)_*^{-1}\mathbf D((B_i)_x)\subset\mathbf D(A_x)\otimes_{\mathbb Z_{q^d}}\mathbb Q_{q^d}$. Therefore as a subset of End$(D(A_x)\otimes\mathbb Q_{q^d})$ the set $\{u_n|n\in S\}$ is bounded, so after possibly replacing $S$ with an infinite subset, we may assume that the sequence $\{u_n|n\in S\}$ converges to a limit $u$ in End$(D(A_x))$. However End$(A)\otimes\mathbb Q_p$ is a
$\mathbb Q_p$-linear subspace of End$(D(A_x))$, and hence it is closed. Since every $u_n$ lies in End$(A)\otimes\mathbb Q_p$, so does their limit $u$, too.

As we noted above for every $n\in S$ the map $u_n$ furnishes a homomorphism:
$$\mathbf D(u_n):\mathbf D(A_x)\longrightarrow\mathbf D((B_i)_x)\cong (f_i)_*^{-1}\mathbf D((B_i)_x),$$
and by taking a limit we get that the action of $u$ restricted to the lattice $\mathbf D(A_x)$ is a map $\mathbf D(A_x)\rightarrow\mathbf D((B_i)_x)$. Let $m$ be any positive integer. Then for every $n\in S$ such that $n\geq m$ the map:
$$\mathbf D(A_x)/p^m\mathbf D(A_x)\longrightarrow\mathbf D((B_i)_x)/p^m\mathbf D((B_i)_x)$$
induced by $\mathbf D(u_n)$ above has kernel $(\mathbf D(G_x)+p^m\mathbf D(A_x))/p^m\mathbf D(A_x)$ where $G_x$ denotes the fibre of $G$ over $x$. By taking the limit we get that the same holds for the map induced by $u$. Since $m$ is arbitrary we can conclude that the kernel of the map $\mathbf D(A_x)\rightarrow \mathbf D((B_i)_x)$ induced by $u$ has kernel $\mathbf D(G_x)$. Therefore the kernel of the action of $u$ on the fibre $\omega_x(D^{\dagger}(A))\cong\mathbf D(A_x)\otimes_{\mathbb Z_{q^d}}\mathbb Q_{q^d}$ is $\omega_x(\mathcal F)\cong \mathbf D(G_x)\otimes_{\mathbb Z_{q^d}}\mathbb Q_{q^d}$. Because $\omega_x$ is faithful we get that the kernel of $u$ as an endomorphism of $D^{\dagger}(A)$ is $\mathcal F$, and hence the claim is true.
\end{proof}

\section{Cycle classes into rigid cohomology over function fields}
\label{cycle-class}

\begin{notn} For every quasi-projective variety $X$ over an arbitrary field $k$ let $Z_r(X)$ denote the group of algebraic cycles of dimension $r$ on $X$, that is, the free abelian group generated by prime cycles of dimension $r$, where recall that a prime cycle on $X$ is an irreducible closed subvariety. Let $CH_r(X)$ denote the homological Chow group of algebraic cycles of $X$ of dimension $r$ modulo rational equivalence in the sense of Fulton (see \cite{Fu}). Assume now that $X$ is smooth. Let $CH^r(X)$ denote the cohomological Chow group of algebraic cycles of $X$ of codimension $r$ modulo usual rational equivalence, and let $CH^*(X)=\oplus_{r=0}^{\infty}CH^r(X)$ denote the Chow ring of $X$; it is a graded ring with respect to the intersection product:
$$\cap:CH^r(X)\times CH^s(X)\longrightarrow CH^{r+s}(X).$$
\end{notn}
\begin{notn} For every smooth projective variety $X$ over an arbitrary field $k$ let $CH_A^r(X)$ denote the quotient of $CH^r(X)$ by the subgroup $ACH^r(X)$ generated by the rational equivalence classes of cycles algebraically equivalent to zero. Let $a_X:CH^r(X)\to CH_A^r(X)$ denote the quotient map. Similarly let $CH_{SN}^r(X)$ denote the quotient of $CH^r(X)$ by the subgroup $SNCH^r(X)$ generated by the rational equivalence classes of cycles smash-nilpotent to zero, and let $s_X:CH^r(X)\to CH_{SN}^r(X)$ denote the quotient map. By a theorem Voevodsky, who introduced this concept, we have $ACH^r(X)\subset SNCH^r(X)$ (see \cite{Vo}). Note that $CH^1(X)=\textrm{Pic}(X)$, the Picard group of $X$, while $CH_A^1(X)=NS(X)$, the N\'eron--Severi group of $X$.  According to the N\'eron--Severi theorem the abelian group $NS(X)$ is finitely generated. 
\end{notn}
\begin{defn} Again let $k$ be a perfect field of characteristic $p$ and let $\mathbb K$ denote the field of fractions of the ring of Witt vectors $\mathcal O$ of $k$. For every quasi-projective variety $X$ over $k$ let $H^n_{rig,c}(X/\mathbb K)$ denote Berthelot's $n$-th rigid cohomology of $X$ with compact support and having coefficients in $\mathbb K$. Following \cite{Pe} we define the $n$-th rigid homology group of $X$ as
$$H^{rig}_n(X/\mathbb K)=\textrm{Hom}_{\mathbb K}(H^n_{rig,c}(X/\mathbb K),\mathbb K),$$
where $\textrm{Hom}_{\mathbb K}$ denotes the group of $\mathbb K$-linear maps. Rigid homology furnishes a covariant functor from the category of quasi-projective varieties over $k$ into finite dimensional $\mathbb K$-linear vector spaces. 
 \end{defn} 
\begin{defn} Let $Z$ be an integral quasi-projective variety of dimension $d$ over $k$. In \cite{Be1} a trace homomorphism:
$$\textrm{Tr}_Z:H^{2d}_{rig,c}(Z/\mathbb K)\to \mathbb K$$
is constructed. We will let $\eta_Z\in H^{rig}_{2d}(Z/\mathbb K)$ denote the corresponding element and call it the fundamental class of $Z$. For every cycle $z=\sum_in_iT_i\in Z_r(X)$, where
$n_i$ are integers and $T_i\subseteq X$ are prime cycles, let
$$\gamma_X(z)=\sum_in_i\alpha_{i*}(\eta_{T_i})\in H^{rig}_{2r}(X/\mathbb K),$$
where $\alpha_i:T_i\to X$ are the closed immersions of these prime cycles into $X$ and
$\alpha_{i*}:H^{rig}_{2r}(T_i/\mathbb K)\to H^{rig}_{2r}(X/\mathbb K)$ are the induced maps.
\end{defn}
We will need the following result of Petrequin:
\begin{prop}\label{petrequin1} Let $X$ be as above, and let $z\in Z_r(X)$ be a cycle rationally equivalent to zero in the sense of Fulton. Then $\gamma_X(z)=0$. Therefore the cycle class map $\gamma_X$ factors through the quotient map $Z_r(X)\to CH_r(X)$, and hence furnishes a homomorphism: 
$$CH_r(X)\longrightarrow H^{rig}_{2r}(X/\mathbb K),$$
which we will denote by the same symbol $\gamma_X$ by slight abuse of notation. 
\end{prop} 
\begin{proof} This is the content of Proposition 6.10 of \cite{Pe} on page 111.
\end{proof}
\begin{defn} Assume now that $X$ is smooth and equidimensional of dimension $d$. By slight abuse of notation let $\gamma_X:CH^r(X)\rightarrow H_{rig}^{2r}(X/\mathbb K)$ denote the
composition of the identification $CH^r(X)\cong CH_{d-r}(X)$, the cycle class map
$\gamma_X:CH_{d-r}(X)\rightarrow H^{rig}_{2(d-r)}(X/\mathbb K)$ and the isomorphism $H^{rig}_{2(d-r)}(X/\mathbb K)\rightarrow H_{rig}^{2r}(X/\mathbb K)$ furnished by Poincar\'e duality (see \cite{Be1}).
\end{defn}
We will also need the following two additional results of Perlequin:
\begin{thm}\label{petrequin2} Let $X$ be as above. Then
\begin{enumerate}
\item[$(a)$] for every $x\in CH^r(X)$ and $y\in CH^s(X)$ we have:
$$\gamma_X(x\cap y) = \gamma_X(x)\cup\gamma_X(y)\in H^{2r+2s}_{rig}(X/\mathbb K).$$
In other words the map
$$\gamma_X:CH^*(X)\longrightarrow H^*_{rig}(X/\mathbb K)$$
is a ring homomorphism.
\item[$(b)$] for every morphism $f:Y\to X$ between smooth equidimensional varieties over $k$, and for every $z\in CH^*(X)$ we have:
$$f^*(\gamma_X(z))=\gamma_{Y}(f^*(z)).$$
\end{enumerate}
\end{thm} 
\begin{proof} The rigid cycle class map is a ring homomorphism by Corollaire 7.6 of \cite{Pe} on page 115, and it is natural by Proposition 7.7 of \cite{Pe} on the same page.
\end{proof}
\begin{prop}\label{voevodsky} Let $X$ be as above. Then $\gamma_X:CH^r(X)\to H^{2r}_{rig}(X/\mathbb K)$ factors through $s_X:CH^r(X)\to CH_{SN}^r(X)$, and hence furnishes a homomorphism:
$$\sigma_X:CH_{SN}^r(X)\longrightarrow H^{2r}_{rig}(X/\mathbb K).$$
\end{prop} 
\begin{proof} We include this well-known argument for the reader's convenience. Let $z$ be a cycle of codimension $r$ on $X$ smash-nilpotent to zero. By definition there is a positive integer $n$ such that the $n$-fold smash product $z^{\otimes n}$ on the $n$-fold direct product $X^n=X\times X\times\cdots\times X$ is rationally equivalent to zero. By Theorem \ref{petrequin2} we have:
$$\gamma_{X^n}(z^{\otimes n})=\pi^*_1(\gamma_{X}(z))\cup\pi^*_2(\gamma_{X}(z))
\cup\cdots\cup\pi^*_n(\gamma_{X}(z)),$$
where $\pi_i:X^n\to X$ is the projection onto the $i$-th factor. Since $z^{\otimes n}$ is rationally equivalent to zero, the left hand side is zero by Proposition \ref{petrequin1}. By the K\"unneth formula for rigid cohomology (see \cite{Be1}) the map:
$$H^{2r}_{rig}(X/\mathbb K)\otimes_{\mathbb K}\cdots\otimes_{\mathbb K}
H^{2r}_{rig}(X/\mathbb K)\longrightarrow
H^{2rn}_{rig}(X^n/\mathbb K)$$
given by the rule:
$$x_1\otimes_{\mathbb K}x_2\otimes_{\mathbb K}\cdots\otimes_{\mathbb K}
x_n\mapsto\pi^*_1(x_1)\cup\pi^*_2(x_2)\cup\cdots\cup\pi^*_n(x_n)$$
is injective. The claim now follows.
\end{proof}
Note that when $X$ is projective then $H^{2d}_{rig,c}(X/\mathbb K)=H^{2d}_{rig}(X/\mathbb K)$, and hence we have a trace homomorphism:
$$\textrm{Tr}_X:H^{2d}_{rig}(X/\mathbb K)\to \mathbb K.$$ 
Because rigid cohomology is a Weil cohomology with respect to this cycle map, we have the following
\begin{prop}\label{rigid_trace} Let $X$ be as above. Then the diagram:
$$\xymatrix{
CH_0(X) \ar^-{\gamma_X}[r] \ar^{\deg}[d] & H^{2d}_{rig}(X/\mathbb K) \ar^{\textrm{\rm Tr}_X}[d] \\
\mathbb Z  \ar@{^{(}->}[r] & \mathbb K}$$
commutes, where $\deg$ denotes the degree map.
\end{prop} 
\begin{proof} This follows immediately from Proposition 2.8 of \cite{Pe} on page 69.
\end{proof}
\begin{defn}\label{4.10} Equip $\mathbb K$ with natural lift of the $p$-power Frobenius. For every quasi-projective variety $V$ over $k$ let $\mathcal O^{\dagger}_V$ denote the trivial overconvergent $F$-isocrystal (with coefficients in $\mathbb K$). Let $\mathcal C$ be a smooth, projective, geometrically irreducible curve over $k$ and let $L$ denote its function field. Let $X$ be a smooth, projective variety over $L$. Let $\pi:\mathfrak X\to U$ be a projective, smooth morphism of $k$-schemes where $U$ is a non-empty Zariski-open sub-curve of $\mathcal C$ and the generic fibre of $\pi$ is isomorphic to $X$. (Note that such a map $\pi:\mathfrak X\to U$ always exists.) Let $\mathcal H^n(X/\mathbb K)$ denote $H^0_{rig}(U,R^n\pi_*(\mathcal O_{\mathfrak X}^{\dagger}))$. Because $g$ is smooth and projective the overholonomic $F\textrm{-}\mathcal D$-module $R^n\pi_*(\mathcal O_{\mathfrak X}^{\dagger})$ is an overconvergent $F$-isocrystal, and hence the cohomology group above is a finite dimensional $\mathbb K$-linear vector space.
\end{defn}
\begin{prop}\label{4.8} The $\mathbb K$-linear vector space $\mathcal H^n(X/\mathbb K)$ is independent of the choice of $\pi:\mathfrak X\to U$, up to a natural isomorphism, and the correspondence $X\mapsto \mathcal H^n(X/\mathbb K)$ is a contravariant functor.
\end{prop} 
\begin{proof} We are going to define a category $\mathcal I$ as follows. Its objects $\textrm{Ob}(\mathcal I)$ are quadruples $(U,\mathfrak X,\pi,\phi)$, where $U$ is a non-empty Zariski-open sub-curve of $\mathcal C$, moreover $\pi:\mathfrak X\to U$ is a projective, smooth morphism of $k$-schemes, and $\phi$ is an isomorphism $\mathfrak X_L\to X$ of $L$-schemes between the generic fibre $\mathfrak X_L$ of $\pi$ and $X$. A morphism of $\mathcal I$ between two objects $(U,\mathfrak X,\pi,\phi)$ and $(U',\mathfrak X',\pi',\phi')$ is an isomorphism
$$\iota:i^*(\mathfrak X)\longrightarrow\mathfrak X'$$
of schemes over $U'$, where $U$ contains $U'$ as an open sub-curve, the morphism $i:U'\to U$ is the inclusion map, and the $U'$-scheme $i^*(\mathfrak X)$ is the base-change of $\mathfrak X$ with respect to $i$, such that the diagram:
$$\xymatrix{
\mathfrak X_L \ar^-{\iota_L}[r] \ar^{\phi}[d] & \mathfrak X'_L \ar^{\phi'}[d] \\
X  \ar^{\textrm{id}_X}@{=}[r] & X}$$
is commutative, where $\iota_L$ is the base change of $\iota$ to Spec$(L)$. The composition $\kappa\circ\iota$ of a morphism $\iota$ from $(U,\mathfrak X,\pi,\phi)$ to $(U',\mathfrak X',\pi',\phi')$ and a morphism $\kappa$ from $(U',\mathfrak X',\pi',\phi')$ to $(U'',\mathfrak X'',\pi'',\phi'')$ is defined as the composition of
$$\iota_{U''}:(i'\circ i)^*(\mathfrak X)\longrightarrow(i')^*(\mathfrak X'),$$
the base change of $\iota$ to $U''$ with respect to the inclusion map $i':U''\to U'$ and
$\kappa$. Then there is a functor from $\mathcal I$ into the category of finite-dimensional $\mathbb K$-linear vector spaces which assigns to every object $(U,\mathfrak X,\pi,\phi)$ to the vector space $H^0_{rig}(U,R^n\pi_*(\mathcal O_{\mathfrak X}^{\dagger}))$, and to every morphism $\iota$ from $(U,\mathfrak X,\pi,\phi)$ to $(U',\mathfrak X',\pi',\phi')$ the composition $\mathcal H^n(\iota)$ of the restriction map:
$$H^0_{rig}(U,R^n\pi_*(\mathcal O_{\mathfrak X}^{\dagger}))\longrightarrow
H^0_{rig}(U',R^n\pi_*(\mathcal O_{\mathfrak X}^{\dagger}))$$
and the map:
$$H^0_{rig}(U',R^n\pi_*(\mathcal O_{\mathfrak X}^{\dagger}))\longrightarrow
H^0_{rig}(U',R^n\pi'_*(\mathcal O_{\mathfrak X'}^{\dagger}))$$
induced by $\iota$. By (the global version of) de Jong's full faithfulness theorem (see Theorem \ref{local-faithful}) for every $\iota$ as above the $\mathbb K$-linear map $\mathcal H^n(\iota)$ is an isomorphism. Clearly $\mathcal I$ is filtering, so the limit:
\begin{equation}\label{4.11.1-}
\mathcal H^n(X/\mathbb K)=\lim_{(V,\mathfrak Y,\rho,\lambda)\in\textrm{Ob}(\mathcal I)}H^0_{rig}(V,R^n\rho_*(\mathcal O_{\mathfrak Y}^{\dagger})),
\end{equation}
with $\mathcal I$ as an index category, is well-defined, and for every $(U,\mathfrak X,\pi,\phi)\in\textrm{Ob}(\mathcal I)$ the natural map:
\begin{equation}\label{4.11.1}
H^0_{rig}(U,R^n\pi_*(\mathcal O_{\mathfrak X}^{\dagger}))\longrightarrow\lim_{(V,\mathfrak Y,\rho,\lambda)\in\textrm{Ob}(\mathcal I)}H^0_{rig}(V,R^n\rho_*(\mathcal O_{\mathfrak Y}^{\dagger})),
\end{equation}
is an isomorphism. The first half of the proposition follows. Let $f:X\to Y$ be map between smooth, projective varieties over $L$. Then there is a non-empty Zariski-open sub-curve $U$ of $\mathcal C$, two smooth, projective morphisms $\pi:\mathfrak X\to U$, $\rho:\mathfrak Y\to U$ of $k$-schemes, and a morphism $\mathfrak f:\mathfrak X\to\mathfrak Y$ of $U$-schemes such that the base change of $\mathfrak f$ to Spec$(L)$ is isomorphic to $f$. The map $\mathfrak f$ induces a homomorphism:
$$H^0_{rig}(U,R^n\pi_*(\mathcal O_{\mathfrak X}^{\dagger}))\longrightarrow
H^0_{rig}(U,R^n\rho_*(\mathcal O_{\mathfrak Y}^{\dagger}))$$
which via the isomorphism (\ref{4.11.1}) furnishes a homomorphism:
$$\mathcal H^n(f):\mathcal H^n(X/\mathbb K)\longrightarrow\mathcal H^n(Y/\mathbb K)$$
which is independent of the choices of $U,\mathfrak X,\mathfrak Y,\mathfrak f$ and only depends on $f$. Equipped with these morphisms the correspondence $X\mapsto\mathcal H^n(X/\mathbb K)$ acquires the structure of a functor, so the second half of the proposition is also true.
\end{proof}
\begin{rem} Strictly speaking one should think of (\ref{4.11.1-}) as the definition of $\mathcal H^*(X/\mathbb K)$. Perhaps it is also worth remarking that although it has many of its additional structures, as we will see below, the functor $X\mapsto\mathcal H^*(X/\mathbb K)$ is not a Weil cohomology theory. 
\end{rem}
\begin{defn}\label{4.9} Let $\pi:\mathfrak X\to U$ be as above. Note that the usual multiplication map $\mathcal O^{\dagger}_{\mathfrak X}\otimes\mathcal O^{\dagger}_{\mathfrak X}\to\mathcal O^{\dagger}_{\mathfrak X}$, where $\otimes$ denotes the tensor product of overconvergent $F$-isocrystals, induces a map:
$$R^j\pi_*(\mathcal O_{\mathfrak X}^{\dagger})\otimes
R^n\pi_*(\mathcal O_{\mathfrak X}^{\dagger})\longrightarrow
R^{j+n}\pi_*(\mathcal O_{\mathfrak X}^{\dagger})$$
of derived functors, which by taking sections induces a $\mathbb K$-bilinear pairing:
$$H^0_{rig}(U,R^j\pi_*(\mathcal O_{\mathfrak X}^{\dagger}))\otimes_{\mathbb K}
H^0_{rig}(U,R^n\pi_*(\mathcal O_{\mathfrak X}^{\dagger}))\longrightarrow
H^0_{rig}(U,R^{j+n}\pi_*(\mathcal O_{\mathfrak X}^{\dagger})),$$
which, via the identification $H^0_{rig}(U,R^i\pi_*(\mathcal O_{\mathfrak X}^{\dagger}))\cong
\mathcal H^i(X/\mathbb K)$, furnishes a $\mathbb K$-bilinear pairing:
$$\cup:\mathcal H^j(X/\mathbb K)\otimes_{\mathbb K}\mathcal H^n(X/\mathbb K)
\longrightarrow\mathcal H^{j+n}(X/\mathbb K).$$
\end{defn}
\begin{lemma} The pairing $\cup$ is independent of the choice of $\pi:\mathfrak X\to U$ and the isomorphism $\mathfrak X_L\cong X$.\end{lemma} 
We will call $\cup$ the cup product on $\mathcal H^*(X/\mathbb K)=\oplus_{n=0}^{\infty}
\mathcal H^n(X/\mathbb K)$, which makes $\mathcal H^*(X/\mathbb K)$ into a graded ring.
\begin{proof} Let $\mathcal I$ be the same category as in the proof of Proposition \ref{4.8}. Then the lemma follows from the following remark: for every morphism
$\iota$ of $\mathcal I$ from $(U,\mathfrak X,\pi,\phi)$ to $(U',\mathfrak X',\pi',\phi')$ the diagram:
$$\xymatrix{
H^0_{rig}(U,R^j\pi_*(\mathcal O_{\mathfrak X}^{\dagger}))\otimes_{\mathbb K}
H^0_{rig}(U,R^n\pi_*(\mathcal O_{\mathfrak X}^{\dagger})) \ar[r] 
\ar^{\mathcal H^j(\iota)\otimes\mathcal H^n(\iota)}[d] & H^0_{rig}(U,R^{j+n}\pi_*(\mathcal O_{\mathfrak X}^{\dagger})) \ar^{\mathcal H^{j+n}(\iota)}[d] \\
H^0_{rig}(U',R^j\pi'_*(\mathcal O_{\mathfrak X'}^{\dagger}))\otimes_{\mathbb K}
H^0_{rig}(U',R^n\pi'_*(\mathcal O_{\mathfrak X'}^{\dagger}))  \ar[r] &
H^0_{rig}(U',R^{j+n}\pi'_*(\mathcal O_{\mathfrak X'}^{\dagger}))}$$
is commutative, where the horizontal maps are the bilinear maps introduced in Definition \ref{4.9}.
\end{proof}
\begin{defn} Assume now that $X$ is also equidimensional of dimension $d$ and let $z\in Z_r(X)$ be a cycle. Let $\pi:\mathfrak X\to U$ be a projective, smooth morphism of $k$-schemes where $U$ is a non-empty Zariski-open sub-curve of $\mathcal C$ and the generic fibre of $\pi$ is isomorphic to $X$. By shrinking $U$, if it is necessary, we may assure that there is a cycle $\mathfrak z\in Z_{r+1}(\mathfrak X)$ whose base change to $X$ is $z$. Let
$\gamma_X(z)\in\mathcal H^{2d-2r}(X/\mathbb K)$ denote the image of $\gamma_{\mathfrak X}(\mathfrak z)\in H^{2d-2r}_{rig}(\mathfrak X,\mathcal O_{\mathfrak X}^{\dagger})$ under the map:
$$H^{2d-2r}_{rig}(\mathfrak X,\mathcal O_{\mathfrak X}^{\dagger})\longrightarrow
H^0_{rig}(U,R^{2d-2r}\pi_*(\mathcal O_{\mathfrak X}^{\dagger}))$$
furnished by the Leray spectral sequence $H^i_{rig}(U,R^j\pi_*(\mathcal O_{\mathfrak X}^{\dagger}))\Rightarrow H^{i+j}_{rig}(\mathfrak X,\mathcal O_{\mathfrak X}^{\dagger})$.
\end{defn}
\begin{lemma} The cycle class $\gamma_X(z)\in\mathcal H^{2d-2r}(X/\mathbb K)$ is well-defined, that is, it is independent of the choices made. 
\end{lemma} 
\begin{proof} We are going to define a category $\mathcal I(z)$ as follows. Its objects
$\textrm{Ob}(\mathcal I(z))$ are quintuples $(U,\mathfrak X,\pi,\phi,\mathfrak z)$, where
$(U,\mathfrak X,\pi,\phi)$ is an object of the category $\mathcal I$ introduced in the proof of
Proposition \ref{4.8}, and $\mathfrak z\in Z_{r+1}(\mathfrak X)$ is a cycle whose base change to $X$ is $z$ with respect to the isomorphism $\phi:\mathfrak X_L\to X$. A morphism of
$\mathcal I(z)$ between two objects $(U,\mathfrak X,\pi,\phi,\mathfrak z)$ and
$(U',\mathfrak X',\pi',\phi',\mathfrak z')$ is a morphism
$$\iota:i^*(\mathfrak X)\longrightarrow\mathfrak X'$$
from $(U,\mathfrak X,\pi,\phi)$ to $(U',\mathfrak X',\pi',\phi')$ in $\mathcal I$, where
$i:U'\to U$ is the inclusion map, which maps the pull-back cycle $i^*(\mathfrak z)$ to
$\mathfrak z'$. 

Let $(U,\mathfrak X,\pi,\phi,\mathfrak z)$ and $(U',\mathfrak X',\pi',\phi',\mathfrak z')$ be two objects of $\mathcal I(z)$. Then there is a non-empty open sub-curve $V\subset U\cap U'$ such that there is an isomorphism $\iota:\mathfrak X_V\to\mathfrak X'_V$, where $\mathfrak X_V$ and $\mathfrak X_V'$ are the pull-back of $\mathfrak X$ and $\mathfrak X'$ with respect to the inclusions $V\subset U$ and $V\subset U'$, respectively, such that the pull-back of $\mathfrak z$ onto $\mathfrak X_V$ maps to the pull-back of $\mathfrak z'$ onto $\mathfrak X'_V$ with respect to $\iota$, or in other words the category $\mathcal I(z)$ is filtering.

For every object $(U,\mathfrak X,\pi,\phi,\mathfrak z)$ of $\mathcal I(z)$ let
$\gamma(U,\mathfrak X,\pi,\phi,\mathfrak z)\in H^0_{rig}(U,R^{2d-2r}\pi_*(\mathcal O_{\mathfrak X}^{\dagger}))$ denote the image of
$\gamma_{\mathfrak X}(\mathfrak z)\in H^{2d-2r}_{rig}(\mathfrak X,\mathcal O_{\mathfrak X}^{\dagger})$ under the map:
$$H^{2d-2r}_{rig}(\mathfrak X,\mathcal O_{\mathfrak X}^{\dagger})\longrightarrow
H^0_{rig}(U,R^{2d-2r}\pi_*(\mathcal O_{\mathfrak X}^{\dagger}))$$
furnished by the Leray spectral sequence $H^i_{rig}(U,R^j\pi_*(\mathcal O_{\mathfrak X}^{\dagger}))\Rightarrow H^{i+j}_{rig}(\mathfrak X,\mathcal O_{\mathfrak X}^{\dagger})$. For every
morphism $\iota$ of $\mathcal I(z)$ from $(U,\mathfrak X,\pi,\phi,\mathfrak z)$ to $(U',\mathfrak X',\pi',\phi',\mathfrak z')$ the homomorphism
$$\mathcal H^{2d-2r}(\iota):H^0_{rig}(U,R^{2d-2r}\pi_*(\mathcal O_{\mathfrak X}^{\dagger}))\longrightarrow H^0_{rig}(U',R^{2d-2r}\pi'_*(\mathcal O_{\mathfrak X'}^{\dagger}))$$
introduced in the proof of Pro\-po\-sition \ref{4.8} maps the section $\gamma(U,\mathfrak X,\pi,\phi,\mathfrak z)$ to the section $\gamma(U',\mathfrak X',\pi',\phi',\mathfrak z')$, therefore the limit:
$$\lim_{(V,\mathfrak Y,\rho,\lambda,\mathfrak o)\in\textrm{Ob}(\mathcal I(z))}
\!\!\!\!\!\!\!\!\!\!\gamma(V,\mathfrak Y,\rho,\lambda,\mathfrak o)
\in\!\!\!\!\!\!\!\!\!\!
\lim_{(V,\mathfrak Y,\rho,\lambda,\mathfrak o)\in\textrm{Ob}(\mathcal I(z))}
\!\!\!\!\!\!\!\!\!\!
H^0_{rig}(V,R^{2d-2r}\rho_*(\mathcal O_{\mathfrak Y}^{\dagger}))=\mathcal H^{2d-2r}(X/\mathbb K)$$
is well-defined. As we already saw in the proof of Proposition \ref{4.8} the natural map:
$$
H^0_{rig}(U,R^{2d-2r}\pi_*(\mathcal O_{\mathfrak X}^{\dagger}))\longrightarrow
\lim_{(V,\mathfrak Y,\rho,\lambda,\mathfrak o)\in\textrm{Ob}(\mathcal I(z))}
H^0_{rig}(V,R^{2d-2r}\rho_*(\mathcal O_{\mathfrak Y}^{\dagger}))$$
for every $(U,\mathfrak X,\pi,\phi,\mathfrak z)\in\textrm{Ob}(\mathcal I(z))$ is an isomorphism. The claim is now clear.
\end{proof}
\begin{prop}\label{chow_new} Let $X$ be as above, and let $z\in Z_r(X)$ be a cycle rationally equivalent to zero. Then $\gamma_X(z)=0$. Therefore the cycle class map $\gamma_X$ factors through the quotient map $Z_r(X)\to CH_r(X)$, and hence furnishes a homomorphism: 
$$\gamma_X:CH^r(X)\longrightarrow\mathcal H^{2r}(X/\mathbb K),$$
which we will denote by the same symbol by slight abuse of notation. 
\end{prop} 
\begin{proof} Without the loss of generality we may assume that there is an irreducible subvariety $Y\subset X$ of dimension $r+1$ and a rational function $\phi$ on $Y$ such that $z$ is the divisor of $\phi$. Let $\pi:\mathfrak X\to U$ be a projective, smooth morphism of $k$-schemes where $U$ is a non-empty Zariski-open sub-curve of $\mathcal C$ and the generic fibre of $\pi$ is $X$. Let $\mathfrak Y\subset\mathfrak X$ be the Zariski-closure of $Y$ in $\mathfrak X$. Then $\mathfrak Y$ is an irreducible subvariety of $X$ of dimension $r+2$ and the divisor of $\phi$, considered as a rational function on $\mathfrak Y$, is a cycle $w\in Z_{r+1}(\mathfrak X)$ whose base change to $X$ is $z$. Therefore
$\gamma_{\mathfrak X}(w)$ is also zero by Proposition \ref{petrequin1}, and the claim is now clear.
\end{proof}
\begin{thm}\label{functorial_new} Let $X$ be as above. Then
\begin{enumerate}
\item[$(a)$] for every $x\in CH^r(X)$ and $y\in CH^s(X)$ we have:
$$\gamma_X(x\cap y) = \gamma_X(x)\cup\gamma_X(y)\in\mathcal H^{2r+2s}(X/\mathbb K).$$
In other words the map
$$\gamma_X:CH^*(X)\longrightarrow\mathcal H^*(X/\mathbb K)$$
is a ring homomorphism.
\item[$(b)$] for every morphism $f:Y\to X$ between smooth equidimensional varieties over $L$, and for every $z\in CH^*(X)$ we have:
$$f^*(\gamma_X(z))=\gamma_{Y}(f^*(z)).$$
\end{enumerate}
\end{thm} 
\begin{proof} We first prove $(a)$. Let $\pi:\mathfrak X\to U$ be a projective, smooth morphism of $k$-schemes where $U$ is a non-empty Zariski-open sub-curve of $\mathcal C$ and the generic fibre of $\pi$ is isomorphic to $X$, as above. By shrinking $U$ if it is necessary we may assume that there are two cycles $\mathfrak x,\mathfrak y$ on $\mathfrak X$ whose pull-backs to $X$ represent $x$ and $y$, respectively. Using moving lemma (see 2.3 of \cite{Fu} on page 156) we immediately reduce to the case when $\mathfrak x$ and $\mathfrak y$ intersect properly. Then $\mathfrak x\cap\mathfrak y$ represents the cycle class $x\cap y$ in $X$ and the claim follows from part $(a)$ of Theorem \ref{petrequin2}.

Let $f:Y\to X$ be as in part $(b)$ above. Then there is a non-empty Zariski-open sub-curve
$U$ of $\mathcal C$, two smooth, projective morphisms $\pi:\mathfrak X\to U$, $\rho:\mathfrak Y\to U$ of $k$-schemes, and a morphism $\mathfrak f:\mathfrak Y\to\mathfrak X$ of $U$-schemes such that the base change of $\mathfrak f$ to Spec$(L)$ is isomorphic to $f$. By shrinking $U$, if it is necessary, we may assure that there is a cycle $\mathfrak z\in Z_{r+1}(\mathfrak X)$ whose base change to $X$ is $z$. Note that there is a commutative diagram:
$$\xymatrix{H^{2d-2r}_{rig}(\mathfrak X,\mathcal O_{\mathfrak X}^{\dagger})
\ar^-{\mathfrak f^*}[r] \ar[d] & H^{2d-2r}_{rig}(\mathfrak Y,\mathcal O_{\mathfrak Y}^{\dagger})
 \ar[d] \\
H^0_{rig}(U,R^{2d-2r}\pi_*(\mathcal O_{\mathfrak X}^{\dagger}))  \ar[r] & H^0_{rig}(U,R^{2d-2r}\rho_*(\mathcal O_{\mathfrak Y}^{\dagger})),}$$
where the vertical maps are furnished by the Leray spectral sequences
$$H^i_{rig}(U,R^j\pi_*(\mathcal O_{\mathfrak X}^{\dagger}))\Rightarrow H^{i+j}_{rig}(\mathfrak X,\mathcal O_{\mathfrak X}^{\dagger})\textrm{ and }
H^i_{rig}(U,R^j\rho_*(\mathcal O_{\mathfrak Y}^{\dagger}))\Rightarrow H^{i+j}_{rig}(\mathfrak Y,\mathcal O_{\mathfrak Y}^{\dagger}),$$
respectively, the upper horizontal map is the pull-back with respect to $\mathfrak f$, and the lower horizontal map corresponds to the homomorphism:
$$\mathcal H^{2d-2r}(f):\mathcal H^k(X/\mathbb K)\longrightarrow\mathcal H^{2d-2r}(Y/\mathbb K)$$
under the the isomorphism (\ref{4.11.1}). Since $\mathfrak f^*(\mathfrak z)$ is a cycle in
$Z_{r+1}(\mathfrak Y)$ whose base change to $Y$ is $f^*(z)$, now claim $(b)$ follows from part $(b)$ of Theorem \ref{petrequin2}. 
\end{proof}
\begin{notn}\label{4.18} Let $X_1,X_2,\ldots,X_n$ be smooth, projective varieties over $L$ and let 
$$\pi_i:X_1\times X_2\times\cdots\times X_n\longrightarrow X_i$$
be the projection onto the $i$-th factor. Moreover let $\pi^*_{i_1,i_2,\ldots,i_n}$ denote the map:
$$\mathcal H^{i_1}(X_1/\mathbb K)\otimes_{\mathbb K}\mathcal H^{i_2}(X_2/\mathbb K)\otimes_{\mathbb K}
\cdots\otimes_{\mathbb K}\mathcal H^{i_n}(X_n/\mathbb K)
\to
\mathcal H^{i_1+i_2+\cdots+i_n}(X_1\times\cdots\times X_n/\mathbb K)$$
given by the rule:
$$x_1\otimes_{\mathbb K}x_2\otimes_{\mathbb K}\cdots\otimes_{\mathbb K}
x_n\mapsto\pi^*_1(x_1)\cup\pi^*_2(x_2)\cup\cdots\cup\pi^*_n(x_n).$$
\end{notn}
\begin{lemma}\label{kunneth_new} The direct sum:
$$\bigoplus_{i_1+\cdots+i_n=r}\!\!\!\!\!\!\!\!\pi^*_{i_1,i_2,\ldots,i_n}:
\!\!\!\!\!\!\!\!\!\!\bigoplus_{i_1+i_2+\cdots+i_n=r}\!\!\!\!\!\!\!\!\!\!
\mathcal H^{i_1}(X_1/\mathbb K)\otimes_{\mathbb K}\cdots\otimes_{\mathbb K}
\mathcal H^{i_n}(X_n/\mathbb K)
\to
\mathcal H^r(X_1\times\cdots\times X_n/\mathbb K)$$
is an isomorphism.
\end{lemma} 
\begin{proof} There is a non-empty Zariski-open sub-curve $U$ of $\mathcal C$ and for every $i=1,2,\ldots,n$ a smooth, projective morphism $\rho_i:\mathfrak X_i\to U$ such that the base change of $\mathfrak X_i$ to Spec$(L)$ is isomorphic to $X_i$ for every index $i$. Let
$$\rho:\mathfrak X_1\times_U\cdots\times_U\mathfrak X_n\longrightarrow U$$
be the fibre product of the morphisms $\rho_1,\rho_2,\ldots,\rho_n$ over $U$. Let
$\rho_{i_1,i_2,\ldots,i_n*}$ denote the map
$$R^{i_1}\!\rho_{1*}(\mathcal O^{\dagger}_{\mathfrak X_1})
\otimes R^{i_2}\!\rho_{2*}(\mathcal O^{\dagger}_{\mathfrak X_2})
\otimes\cdots\otimes R^{i_n}\!\rho_{n*}(\mathcal O^{\dagger}_{\mathfrak X_n})
\longrightarrow R^{i_1+i_2+\cdots+i_n}
\rho_*(\mathcal O^{\dagger}_{\mathfrak X_1\times_U\cdots\times_U\mathfrak X_n})$$
of overconvergent $F$-isocrystals induced by the fibre-wise exterior cup product. Note that the direct sum:
$$\bigoplus_{i_1+\cdots+i_n=r}\!\!\!\!\!\!\!\!\rho_{i_1,i_2,\ldots,i_n*}:
\!\!\!\!\!\!\!\!\!\!\bigoplus_{i_1+i_2+\cdots+i_n=r}\!\!\!\!\!\!\!\!\!\!
R^{i_1}\!\rho_{1*}(\mathcal O^{\dagger}_{\mathfrak X_1})
\otimes\cdots\otimes R^{i_n}\!\rho_{n*}(\mathcal O^{\dagger}_{\mathfrak X_n})
\to
R^r
\rho_*(\mathcal O^{\dagger}_{\mathfrak X_1\times_U\cdots\times_U\mathfrak X_n})$$
is an isomorphism. The latter follows at once from the proper base change theorem and the K\"unneth formula applied to the fibre of the $U$-scheme $\mathfrak X_1\times_U\cdots\times_U\mathfrak X_n$ at a closed point $x$ of $U$. By taking global sections we get the lemma.
\end{proof}
\begin{prop} Let $X$ be as above. Then $\gamma_X:CH^r(X)\to\mathcal H^{2r}(X/\mathbb K)$ factors through $s_X:CH^r(X)\to CH_{SN}^r(X)$, and hence furnishes a homomorphism:
$$\sigma_X:CH_{SN}^r(X)\longrightarrow\mathcal H^{2r}(X/\mathbb K).$$
\end{prop} 
\begin{proof} The proof is very similar to the argument presented for Proposition \ref{voevodsky}. Let $z$ be a cycle of codimension $r$ on $X$ smash-nilpotent to zero. By definition there is a positive integer $n$ such that the $n$-fold smash product $z^{\otimes n}$ on the $n$-fold direct product $X^n=X\times X\times\cdots\times X$ is rationally equivalent to zero. By Theorem \ref{functorial_new} we have:
$$\gamma_{X^n}(z^{\otimes n})=\pi^*_1(\gamma_{X}(z))\cup\pi^*_2(\gamma_{X}(z))
\cup\cdots\cup\pi^*_n(\gamma_{X}(z)),$$
where $\pi_i:X^n\to X$ is the projection onto the $i$-th factor. Since $z^{\otimes n}$ is rationally equivalent to zero, the left hand side is zero by Proposition \ref{chow_new}. By Lemma \ref{kunneth_new} above the map:
$$\mathcal H^{2r}(X/\mathbb K)\otimes_{\mathbb K}\cdots\otimes_{\mathbb K}
\mathcal H^{2r}(X/\mathbb K)\longrightarrow
\mathcal H^{2rn}(X^n/\mathbb K)$$
given by the rule:
$$x_1\otimes_{\mathbb K}x_2\otimes_{\mathbb K}\cdots\otimes_{\mathbb K}x_n
\mapsto\pi^*_1(x_1)\cup\pi^*_2(x_2)\cup\cdots\cup\pi^*_n(x_n)$$
is injective. The claim now follows.
\end{proof}
By the above $\gamma_X:CH^r(X)\to\mathcal H^{2r}(X/\mathbb K)$ also factors through the quotient map $a_X:CH^r(X)\to CH_{A}^r(X)$, and hence furnishes a homomorphism:
$$\alpha^r_X:CH_{A}^r(X)\longrightarrow\mathcal H^{2r}(X/\mathbb K).$$
\begin{prop}\label{my_trace} Let $X$ be as above. Then there is a homomorphism:
$$\textrm{\rm Tr}_X:\mathcal H^{2d}(X/\mathbb K)\longrightarrow \mathbb K$$ 
such that diagram:
$$\xymatrix{
CH^0_A(X) \ar^-{\alpha^0_X}[r] \ar^{\deg}[d] & \mathcal H^{2d}(X/\mathbb K)
 \ar^{\textrm{\rm Tr}_X}[d] \\
\mathbb Z  \ar@{^{(}->}[r] & \mathbb K}$$
commutes.
\end{prop} 
\begin{proof} Because the group $CH^0_A(X)$ is finitely generated, there is a $\pi:\mathfrak X\to U$ as above such that for every $z\in CH^0_A(X)$ there is a cycle $\mathfrak z$ on $\mathfrak X$ whose pull-back to $X$ represents $z$. We may even assume that for every $x\in|U|$ the fibre of $\pi$ over $x$ intersects $\mathfrak z$ properly, by removing irreducible components which $\pi$ does not map onto $U$. Fix a closed point $x\in|U|$ whose residue field $k'$ is a finite extension of $k$. Let $\mathbb K'$ be the unique unramified extension of $\mathbb K$ with residue field $k'$. Let $R^{2d}\pi_*(\mathcal O^{\dagger}_{\mathfrak X})_x$ be the fibre of $R^{2d}\pi_*(\mathcal O^{\dagger}_{\mathfrak X})$ at $x$; forgetting the Frobenius structure we get a finite dimensional $\mathbb K'$-linear vector space. Let $\mathfrak X_x$ be the fibre of $\pi$ over $x$, and let
$$b:R^{2d}\pi_*(\mathcal O^{\dagger}_{\mathfrak X})_x\longrightarrow H^{2d}_{rig}(\mathfrak X_x/\mathbb K')$$
be the map furnished by the proper base change theorem. We define
$$\textrm{\rm Tr}_X:\mathcal H^{2d}(X/\mathbb K)\longrightarrow \mathbb K$$
as the composition of the restriction map
$$\mathcal H^{2d}(X/\mathbb K)\cong H^0_{rig}(U,R^{2d}\pi_*(\mathcal O^{\dagger}_{\mathfrak X}))
\longrightarrow R^{2d}\pi_*(\mathcal O^{\dagger}_{\mathfrak X})_x,$$
the homomorphism $b$, the trace homomorphism
$$\textrm{\rm Tr}_{\mathfrak X_x}:H^{2d}_{rig}(\mathfrak X_x/\mathbb K')\longrightarrow
\mathbb K'$$
in Proposition \ref{rigid_trace}, and $\deg(\mathbb K'/\mathbb K)^{-1}\cdot\textrm{Tr}_{\mathbb K'/\mathbb K}$, where
$\textrm{Tr}_{\mathbb K'/\mathbb K}$, $\deg(\mathbb K'/\mathbb K)$ is the trace map and the degree of the field extension $\mathbb K'/\mathbb K$, respectively. Note that for every cycle $\mathfrak z$ on $\mathfrak X$ of codimension $d$ the pull-backs of $\mathfrak z$ to $X$ and $\mathfrak X_v$ have the same degree. So by Proposition \ref{rigid_trace} the map $\textrm{\rm Tr}_X$ has the required properties.
\end{proof}
For any $n\in\mathbb N$ and for any abelian scheme $A$ over a scheme $S$ let $[n]:A\to A$ be the multiplication by $n$ map on $A$.
\begin{lemma}\label{homogenous} Let $A$ be an abelian variety over $L$. Then $\mathcal H^k([n]):
\mathcal H^k(A/\mathbb K)\to\mathcal H^k(A/\mathbb K)$ is the multiplication by $n^k$ map.
\end{lemma}
\begin{proof} Let $U$ be a non-empty Zariski-open set of $\mathcal C$ and let
$\pi:\mathfrak A\to U$ be an abelian scheme whose generic fibre is $A$. Using the proper base change theorem, we only need to check a similar claim on the fibres of $\pi$, that is, we only need to prove the similar claim for abelian varieties over (finite extensions of) $k$. So let $B$ be an abelian variety over $k$; we want to show that the map $H^k_{rig}([n]):H_{rig}^k(B/\mathbb K)\to H_{rig}^k(B/\mathbb K)$ induced by $[n]$ is the multiplication by $n^k$ map. Let $\mathbf B$ be an abelian scheme over Spec$(\mathcal O)$ which is a lift of $B$ to Spec$(\mathcal O)$. Let $\mathcal B$ denote the base change of $\mathbf B$ to Spec$(\mathbb K)$. It is well-known that the map $H^k_{dR}([n]):H_{dR}^k(\mathcal B/\mathbb K)\to H_{dR}^k(\mathcal B/\mathbb K)$ induced by $[n]$ on the $k$-th de Rham cohomology of $\mathcal B$ over $\mathbb K$ is the multiplication by $n^k$ map. Since the base change of $[n]:\mathbf B\to\mathbf B$ to Spec$(k)$ and Spec$(\mathbb K)$ are the multiplication by $n$ maps on $B$ and $\mathcal B$, respectively, under the isomorphism $H_{rig}^k(B/\mathbb K)\cong H_{dR}^k(\mathcal B/\mathbb K)$ (furnished by the lift $\mathbf B$) the maps $H^k_{rig}([n])$ and $H^k_{dR}([n])$ correspond to each other. The claim follows.
\end{proof}

\section{The $p$-adic Tate conjecture over function fields}
\label{tate}

\begin{defn} We continue to use the notation of the previous section. Let $X$ be again a smooth, projective variety over the function field $L$ and let $\pi:\mathfrak X\to U$ be a projective, smooth morphism of $k$-schemes where $U$ is a non-empty Zariski-open sub-curve of $\mathcal C$ and the generic fibre of $\pi$ is isomorphic to $X$. The Frobenius of the $F$-isocrystal $R^r\pi_*(\mathcal O^{\dagger}_{\mathfrak X})$ furnishes a map:
$$F:H^0_{rig}(U,R^r\pi_*(\mathcal O^{\dagger}_{\mathfrak X}))\longrightarrow
H^0_{rig}(U,R^r\pi_*(\mathcal O^{\dagger}_{\mathfrak X}))$$
on global sections which, via the identification in (\ref{4.11.1}), makes $\mathcal H^r(X/\mathbb K)$ into an $F$-isocrystal over $k$. Since the comparison maps appearing in the proof of Proposition \ref{4.8} are induced by maps between $F$-isocrystals, this additional structure is independent of the choice of  $\pi:\mathfrak X\to U$ and the isomorphism $\mathfrak X_L\cong X$. 
\end{defn}
\begin{prop} Assume now that $X$ is also equidimensional of dimension $d$ and let $z\in Z_r(X)$ be a cycle. Then $F(\gamma_X(z))=p^{r}\gamma_X(z)$.
\end{prop}
\begin{proof} Note that the map:
$$H^{2d-2r}(\mathfrak X,\mathcal O_{\mathfrak X}^{\dagger})\longrightarrow
H^0_{rig}(U,R^{2d-2r}\pi_*(\mathcal O_{\mathfrak X}^{\dagger}))$$
furnished by the Leray spectral sequence $H^i_{rig}(U,R^j\pi_*(\mathcal O_{\mathfrak X}^{\dagger}))\Rightarrow H^{i+j}_{rig}(\mathfrak X,\mathcal O_{\mathfrak X}^{\dagger})$ respects the Frobenii. Therefore the claim follows at once from Proposition 5.34 of \cite{Pe} on page 104. 
\end{proof}
\begin{defn} Assume now that $k$ is a finite field $\mathbb F_q$ and $q=p^f$. Let $X$ be as above and let $l$ be any prime number different from $p$. Recall that the $l$-adic Tate conjecture $T(X,r,l)$ for $r$ codimensional algebraic cycles of $X$  claims that the map
$$CH^{r}_A(X)\otimes\mathbb Q_l\longrightarrow H^{2r}(X,\mathbb Q_l(r))^{\textrm{Gal}(\overline L/L)}$$
furnished by the $l$-adic cycle class map is surjective. The $p$-adic analogue $T(X,r,p)$ of this conjecture  says that the $\mathbb K$-linearisation of the map $\alpha^r_X$:
$$\overline{\alpha}^r_X:CH^{r}_A(X)\otimes\mathbb K\longrightarrow
\mathcal H^{2r}(X/\mathbb K)^{F^f=q^r\cdot\textrm{id}}$$
is surjective. (Note that $F^f$ is $\mathbb K$-linear and hence the range of the map above is a $\mathbb K$-linear subspace of $\mathcal H^r(X/\mathbb K)$.)
\end{defn} 
\begin{prop}\label{equivalence1} Assume that $X$ is geometrically irreducible. The following claims are equivalent:
\begin{enumerate}
\item[$(a)$] the claim $T(X,1,l)$ is true for some prime number $l$,
\item[$(b)$] the claim $T(X,1,l)$ is true for every prime number $l$.
\end{enumerate}
\end{prop}
We will prove the proposition above through a sequence of other claims. But before doing so, it will be convenient to make some basic definitions.
\begin{defn} Let $F:U\to U$ be again the $q$-power Frobenius, and let $\mathbb L$ be a totally ramified finite extension of $\mathbb Q_q$. Let $\sigma$ be the identity of $\mathbb L$ and let $\mathcal F$ be an object of $\textrm{$F_{\sigma}$-Isoc}^{\dagger}(U/\mathbb L)$. By definition the $L$-function of $\mathcal F$ on $U$ is the product:
$$L(U,\mathcal F,t)=\prod_{x\in|U|}\det(1-t^{\deg(x)}\cdot
\textrm{Frob}_x(\mathcal F))^{-1}.$$
Similarly for every prime number $l\neq p$, a finite extension $\mathbb E$ of
$\mathbb Q_l$, and for every lisse $\mathbb E$-sheaf $\mathcal L$ on $U$ let $\textrm{Frob}_x(\mathcal L)$ denote the geometric Frobenius of the fibre
$\mathcal L_x$ of $\mathcal L$ over $x$. By definition the $L$-function of $\mathcal L$ on $U$ is the product:
$$L(U,\mathcal L,t)=\prod_{x\in|U|}\det(1-t^{\deg(x)}\cdot
\textrm{Frob}_x(\mathcal L))^{-1}.$$
\end{defn}
\begin{proof}[Proof of Proposition \ref{equivalence1}] Let $\pi:\mathfrak X\to U$ be a projective, smooth morphism of $k$-schemes where $U$ is a non-empty Zariski-open sub-curve of $\mathcal C$ and the generic fibre of $\pi$ is isomorphic to $X$, as usual. For every prime number $l\neq p$ and for every natural number $n$ let $H^n(\mathfrak X)_l$ denote the $n$-th higher direct image $R^n\pi_*(\mathbb Q_l)$ of the constant $l$-adic sheaf $\mathbb Q_l$. It is a lisse $\mathbb Q_l$-sheaf on $U$. Similarly let $H^n(\mathfrak X)_p$ denote the oveconvergent $F$-isocrystal $R^n\pi_*(\mathcal O^{\dagger}_{\mathfrak X})$.
\begin{lemma}\label{6.6} We have $L(U,H^n(\mathfrak X)_l,t)\in\mathbb Q[[t]]$ for every $l$, and these functions are all equal. 
\end{lemma}
\begin{proof} It will be sufficient to prove that for every $x\in|U|$ the local factors 
$$\det(1-t^{\deg(x)}\cdot\textrm{Frob}_x(H^n(\mathfrak X)_l))^{-1}$$
have coefficients in $\mathbb Q$, and are equal. However these claims follows at once from the respective proper base change theorems and the main result of \cite{KaMe} applied to the fibre of $\pi$ over $x$. 
\end{proof}
Set $n=2d-2$ and let $\rho$ be the common order of pole of these $L$-functions at $t=q^{-d}$. Proposition \ref{equivalence1} above follows immediately from the claim below.
\end{proof}
\begin{prop} For every prime $l$ the following claims are equivalent:
\begin{enumerate}
\item[$(a)$] the claim $T(X,1,l)$ is true,
\item[$(b)$] the rank of $NS(X)$ is $\rho$.
\end{enumerate}
\end{prop}
\begin{proof} We are only going to prove the claim for $l=p$. The proof for other primes is similar, and rather well-known (see \cite{Ta2}, for example). Let
$$\langle\cdot,\cdot\rangle:CH^1_A(X)\otimes\mathbb K\times CH_A^{d-1}\otimes \mathbb K\longrightarrow CH^0_A(X)\otimes \mathbb K\longrightarrow \mathbb K$$
be the composition of the $\mathbb K$-linearisations of the intersection pairing and the degree map. By Theorem \ref{functorial_new} and Proposition \ref{my_trace} the diagram:
$$\xymatrix{
CH^1_A(X)\otimes \mathbb K\times CH_A^{d-1}\otimes \mathbb K
\ar^{\overline{\alpha}^1_X\times\overline{\alpha}^{d-1}_X}[d] 
\ar^{\quad\quad\quad\quad\quad\langle\cdot,\cdot\rangle}[r] & \mathbb K
 \\
\mathcal H^{2}(X/\mathbb K)\times \mathcal H^{2d-2}(X/\mathbb K)
\ar^{\quad\quad\quad\cup}[r]  & 
 \ar^{\textrm{\rm Tr}_X}[u]\mathcal H^{2d}(X/\mathbb K)}$$
is commutative. By Matsusaka's theorem (see \cite{Ma}) algebraic equivalence and numerical equivalence coincide on $CH^1(X)$ up to torsion. Therefore we get that
\begin{equation}\label{the_map}
\overline{\alpha}^1_X:CH^1_A(X)\otimes \mathbb K\longrightarrow\mathcal H^{2}(X/\mathbb K)^{F^f=q\cdot\textrm{id}}
\end{equation}
is injective. 

Let $W\subset\mathcal H^{2}(X/\mathbb K)$ be the generalised eigenspace of $F^f$ with eigenvalue $q$, that is, the union $\bigcup_{n=1}^{\infty}
\textrm{Ker}(F^f-q\cdot\textrm{id})^n$. Clearly $\mathcal H^{2}(X/\mathbb K)^{F^f=q\cdot\textrm{id}}\subseteq W$, so by Lemma \ref{rank} below the range of the map in (\ref{the_map}) has dimension at most $\rho$. Therefore $(b)$ implies $(a)$. Assume now that $(a)$ holds and let $V\subset W$ be the subspace annihilated by $\textrm{Im}(\overline{\alpha}^{d-1}_X)$ with respect to the cup product. Then $V$ is an $F^f$-invariant subspace, since $\textrm{Im}(\overline{\alpha}^{d-1}_X)$ is $F^f$-invariant. By Matsusaka's theorem the intersection of $V$ and $\textrm{Im}(\overline{\alpha}^{1}_X)$ is the zero vector. If $(b)$ were false then $V$ would have positive dimension, so it would contain an eigenvector for $F^f$ with eigenvalue $q$. This is a contradiction.
\end{proof}
\begin{lemma}\label{rank} The $\mathbb K$-dimension of $W$ is $\rho$.
\end{lemma}
\begin{proof} By the Etesse-Le~Stum trace formula (see Theorem 6.3 of \cite{ES} on pages 570--571):
$$L(U,R^{2d-2}\pi_*(\mathcal O^{\dagger}_{\mathfrak X}),t)=\prod_{i=0}^2\det(1-t\cdot F^f|H^i_{rig,c}(U,R^{2d-2}\pi_*(\mathcal O^{\dagger}_{\mathfrak X})))^{(-1)^{i+1}}.$$
Since $R^{2d-2}\pi_*(\mathcal O^{\dagger}_{\mathfrak X})$ is pure of weight $2d-2$, for $i=0,1$ the cohomology group $H^i_{rig,c}(U,R^{2d-2}\pi_*(\mathcal O^{\dagger}_{\mathfrak X}))$ is mixed of weight $\leq 2d-2+i$ by Theorem 5.3.2 of \cite{Ke2} on page 1445; in particular it has weights $<2d$. Therefore the first two factors of the product above do not contribute to the order of the pole at $t=q^{-d}$. Let
$\alpha_1,\alpha_2,\ldots,\alpha_m$ be the reciprocal roots of the characteristic polynomial of $F^f$ acting on $H^2_{rig,c}(U,R^{2d-2}\pi_*(\mathcal O^{\dagger}_{\mathfrak X}))$, where $m$ is the dimension of $H^2_{rig,c}(U,R^{2d-2}\pi_*(\mathcal O^{\dagger}_{\mathfrak X}))$. Since there is a perfect pairing:
$$R^{2}\pi_*(\mathcal O^{\dagger}_{\mathfrak X})\otimes
R^{2d-2}\pi_*(\mathcal O^{\dagger}_{\mathfrak X})\longrightarrow
R^{2d}\pi_*(\mathcal O^{\dagger}_{\mathfrak X})\cong
\mathcal O^{\dagger}_U(d),$$
by Poincar\'e duality (see Theorem 9.5 of \cite{Cr2} on pages 753--754) the reciprocal roots of the characteristic polynomial of $F^f$ acting on $H^0_{rig}(U,R^2\pi_*(\mathcal O^{\dagger}_{\mathfrak X})))$ are $\frac{q^{d+1}}{\alpha_1},\frac{q^{d+1}}{\alpha_2},\ldots,\frac{q^{d+1}}{\alpha_m}$. Since $\mathcal H^2(X/\mathbb K)=H^0_{rig}(U,R^2\pi_*(\mathcal O^{\dagger}_{\mathfrak X}))$ the order of vanishing of the polynomial $\det(1-t\cdot F^f|H^2_{rig,c}(U,R^{2d-2}\pi_*(\mathcal O^{\dagger}_{\mathfrak X})))$ at $t=q^{-d}$ is the dimension of $W$.
\end{proof}
Recall that $U$ is a non-empty Zariski-open sub-curve of $\mathcal C$. Let $\mathfrak A$ and $\mathfrak B$ be two abelian schemes over $U$, and let $A,B$ denote their base change to $L$, respectively. For every prime $l$ different from $p$ the $l$-adic isogeny conjecture $I(\mathfrak A,\mathfrak B,l)$ is that the map
$$\textrm{\rm Hom}(\mathfrak A,\mathfrak B)\otimes\mathbb Q_l
\longrightarrow\textrm{\rm Hom}(T_l(\mathfrak A),T_l(\mathfrak B))$$
induced by the functoriality of $l$-adic Tate module is an isomorphism. The $p$-adic analogue $I(\mathfrak A,\mathfrak B,p)$ is that the map:
$$D^{\dagger}_{\mathfrak A,\mathfrak B}:\textrm{\rm Hom}(\mathfrak A,\mathfrak B)\otimes \mathbb Q_p
\longrightarrow
\textrm{\rm Hom}(D^{\dagger}(\mathfrak A),D^{\dagger}(\mathfrak B))$$
induced by the functoriality of overconvergent Dieudonn\'e modules is an isomorphism. 
\begin{rem} Above $\textrm{\rm Hom}(D^{\dagger}(\mathfrak A),D^{\dagger}(\mathfrak B))$ means homomorphisms of overconvergent $F$-isocrystals. Let 
$\textrm{\rm Hom}_{F^f}(D^{\dagger}(\mathfrak A),D^{\dagger}(\mathfrak B))$ denote the group of homomorphisms as overconvergent $F^f$-isocrystals. Since
$$\textrm{\rm Hom}(D^{\dagger}(\mathfrak A),D^{\dagger}(\mathfrak B))
\otimes_{\mathbb Q_p}\mathbb K\cong
\textrm{\rm Hom}_{F^f}(D^{\dagger}(\mathfrak A),D^{\dagger}(\mathfrak B))$$
we get that $I(\mathfrak A,\mathfrak B,p)$ is equivalent to the claim that
$$D^{\dagger}_{\mathfrak A,\mathfrak B}\otimes_{\mathbb Q_p}\mathbb K:\textrm{\rm Hom}(\mathfrak A,\mathfrak B)\otimes \mathbb K
\longrightarrow
\textrm{\rm Hom}_{F^f}(D^{\dagger}(\mathfrak A),D^{\dagger}(\mathfrak B))$$
is an isomorphism. Moreover note that in order to prove $I(\mathfrak A,\mathfrak B,l)$ it is enough to prove $I(\mathfrak A\times\mathfrak B,\mathfrak A\times\mathfrak B,l)$.
\end{rem}
\begin{prop}\label{equivalence} Let $\mathfrak B$ be an abelian scheme over $U$, and let $B$ denote its base change to $L$. For every prime number $l$ the following holds:
\begin{enumerate}
\item[$(a)$] the conjecture $T(B\times B,1,l)$ implies the conjecture $I(\mathfrak B,\mathfrak B,l)$,
\item[$(b)$] the conjecture $I(\mathfrak B,\mathfrak B,l)$ implies the conjecture $T(B,1,l)$. 
\end{enumerate}
\end{prop}
\begin{proof} This claim is well-known when $l\neq p$ (see for example the proof of Theorem 4 of \cite{Ta} on page 143.) We will adopt this proof to the case $l=p$. Recall that a divisorial correspondence between two pointed varieties $(S,s)$ and $(T,t)$ (over $L$) is an invertible sheaf $\mathcal L$ on $S\times T$ whose restrictions to $S\times\{t\}$ and $\{s\}\times T$ are both trivial. These form a subgroup of the Picard group $S\times T$. Let $DC(S,T)$ denote this group by slight abuse notation. (This is justified since usually the base points can be easily guessed.) Let $ADC(S,T)$ denote the subgroup of divisorial correspondences algebraically equivalent to zero and let $DC_A(S,T)$ denote the quotient of $DC(S,T)$ by $ADC(S,T)$. Clearly $DC_A(S,T)$ is a subgroup of $NS(S\times T)$.
\begin{lemma} Assume that $S,T$ are smooth and projective. Then the image of $DC_A(S,T)$ under $\alpha_{S\times T}^1$ lies in $\mathcal H^1(S/\mathbb K)\otimes_{\mathbb K}
\mathcal H^1(T/\mathbb K)$.
\end{lemma}
Here we consider $\mathcal H^1(S/\mathbb K)\otimes_{\mathbb K}\mathcal H^1(T/\mathbb K)$ as a subgroup of
$\mathcal H^2(S\times T/\mathbb K)$ via the map introduced in Notation \ref{4.18}.
\begin{proof} Let $\pi_S:S\times T\to S$ and $\pi_T:S\times T\to T$ be the projection onto the first and onto the second factor, respectively. Let $i_S:S\to S\times T$ be the imbedding identifying $S$ with $S\times\{t\}$, and similarly let $i_T:T\to S\times T$ be the imbedding identifying $S$ with $\{s\}\times T$. Note that by part $(b)$ of Theorem \ref{functorial_new} the image of $\alpha^1_{S\times T}(DC_A(S,T))$ with respect to the map:
$$\mathcal H^2(i_S)\oplus\mathcal H^2(i_T):\mathcal H^2(S\times T/\mathbb K)\longrightarrow\mathcal H^2(S/\mathbb K)\oplus\mathcal H^2(T/\mathbb K),$$
is zero. Therefore it will be enough to show that the kernel of $\mathcal H^2(i_S)\oplus\mathcal H^2(i_T)$ is $\mathcal H^1(S/\mathbb K)\otimes_{\mathbb K}\mathcal H^1(T/\mathbb K)$. Since $\pi_S\circ i_S=\textrm{id}_S$ and $\pi_T\circ i_T=\textrm{id}_T$, we get that the restriction of $\mathcal H^2(i_S)\oplus\mathcal H^2(i_T)$ onto the direct summand $\mathcal H^2(S/\mathbb K)\oplus\mathcal  H^2(T/\mathbb K)$ in the decomposition of $\mathcal H^2(S\times T/\mathbb K)$ in Lemma \ref{kunneth_new} is an isomorphism, and hence it will be enough to show that the kernel of $\mathcal H^2(i_S)\oplus\mathcal H^2(i_T)$ contains $\mathcal H^1(S/\mathbb K)\otimes_{\mathbb K}\mathcal H^1(T/\mathbb K)$. The latter is immediate however; because $\mathcal H^*(i_S)$ is a homomorphism of graded rings, the image of $\mathcal H^1(S/\mathbb K)\otimes_{\mathbb K}\mathcal H^1(T/\mathbb K)$ under $\mathcal H^2(i_S)$ lies in
$$\mathcal H^1(S/\mathbb K)\otimes_{\mathbb K}
\mathcal H^1(\{\mathfrak t\}/\mathbb K)\subset
\mathcal H^2(S\times_V\{\mathfrak t\}/\mathbb K)\cong\mathcal H^2(S/\mathbb K) ,$$
but by the proper base change theorem $\mathcal H^1(\{\mathfrak t\}/\mathbb K)=0$. We may argue similarly for $\mathcal H^2(i_T)$.
\end{proof}
Assume now that there are abelian schemes $\mathbf{s}:\mathfrak S\to U$ and $\mathbf{t}:\mathfrak T\to U$ such that their base change to $L$ are $S,T$, respectively. 
\begin{lemma}\label{5.12} The conjecture $I(\mathfrak S,\mathfrak T,p)$ holds if and only if the image of the vector space $DC_A(S,T)\otimes \mathbb K$ under $\overline{\alpha}_{S\times T}^1$ is $(\mathcal H^1(S/\mathbb K)\otimes_{\mathbb K}
\mathcal H^1(T/\mathbb K))^{F^f=q\cdot\textrm{\rm id}}$.
\end{lemma}
\begin{proof} Note that both
$$\overline{\alpha}_{S\times T}^1:DC_A(S,T)\otimes \mathbb K
\longrightarrow
(\mathcal H^1(S/\mathbb K)\otimes_{\mathbb K}
\mathcal H^1(T/\mathbb K))^{F^f=q\cdot\textrm{\rm id}}$$
and
$$D^{\dagger}_{\mathfrak S,\mathfrak T}:\textrm{\rm Hom}(\mathfrak S,\mathfrak T)\otimes \mathbb K
\longrightarrow
\textrm{\rm Hom}_{F^f}(D^{\dagger}(\mathfrak S),D^{\dagger}(\mathfrak T))$$
are injective maps. In the first case this follows from by Matsusaka's theorem, while in the second case it holds because of the faithfulness of the overconvergent Dieudonn\'e module functor. Therefore it will be enough to show that there are isomorhisms:
$$DC_A(S,T)\otimes \mathbb K\cong
\textrm{\rm Hom}(\mathfrak S,\mathfrak T)\otimes \mathbb K,\textrm{ and}$$
$$(\mathcal H^1(S/\mathbb K)\otimes_{\mathbb K}
\mathcal H^1(T/\mathbb K))^{F^f=q\cdot\textrm{\rm id}}\cong
\textrm{\rm Hom}_{F^f}(D^{\dagger}(\mathfrak S),D^{\dagger}(\mathfrak T)).$$
Since the first is very well known, we only need to show the second. Recall that $D^{\dagger}(\mathfrak S)$ and $D^{\dagger}(\mathfrak T)$ are overconvergent $F$-isocrystals on $U$, and hence:
\begin{eqnarray}\nonumber
\textrm{\rm Hom}_{F^f}(D^{\dagger}(\mathfrak S),D^{\dagger}(\mathfrak T))
&\cong & H^0_{rig}(U,Hom(D^{\dagger}(\mathfrak S),D^{\dagger}(\mathfrak T)))^{F^f=
\textrm{\rm id}}\\
\nonumber & \cong &
H^0_{rig}(U,D^{\dagger}(\mathfrak S)^{\vee}\otimes D^{\dagger}(\mathfrak T))^{F^f=\textrm{\rm id}}\\
\nonumber & \cong &
H^0_{rig}(U,D^{\dagger}(\mathfrak S)^{\vee}(1)\otimes D^{\dagger}(\mathfrak T))^{F^f=q\cdot\textrm{\rm id}},
\end{eqnarray}
where $Hom(\cdot,\cdot)$ is the internal Hom in the category of overconvergent $F$-isocrystals, for every overconvergent $F$-isocrystal $\mathcal F$ on $U$ we let $\mathcal F^{\vee}$ denote the dual overconvergent $F$-isocrystal $Hom(\mathcal O^{\dagger}_U,\mathcal F)$ and we let $\mathcal F(1)$ denote its Tate-twist. The latter is an overconvergent $F$-isocrystal with the same underlying vector bundle and connection as $\mathcal F$, but with a new Frobenius which is $q$ times the old Frobenius. On the other hand:
\begin{eqnarray}\nonumber
(\mathcal H^1(S/\mathbb K)\otimes_{\mathbb K}
\mathcal H^1(T/\mathbb K))^{F^f=q\cdot\textrm{\rm id}}\!\!\!\!
&\cong &\!\!\!\!\!\left(H^0_{rig}(U,R^1\mathbf{s}_*(\mathcal O_{\mathfrak S}))
\otimes_{\mathbb K}
H^0_{rig}(U,R^1\mathbf{t}_*(\mathcal O_{\mathfrak T}))\right)^{F^f=q\cdot\textrm{\rm id}}\\
\nonumber & \cong &\!\!\!
H^0_{rig}(U,R^1\mathbf{s}_*(\mathcal O_{\mathfrak S})\otimes
R^1\mathbf{t}_*(\mathcal O_{\mathfrak T}))^{F^f=q\cdot\textrm{\rm id}},
\end{eqnarray}
so in order to conclude the proof, we only need to show Lemma
\ref{d-module_comparison} below.
\end{proof}
\begin{lemma}\label{d-module_comparison} For every abelian scheme $\pi:\mathfrak A\to U$ the overconvergent $F$-isocrys\-tals $R^1\pi_*(\mathcal O_{\mathfrak A}),D^{\dagger}(\mathfrak A)$, and $D^{\dagger}(\mathfrak A)^{\vee}(1)$ are all isomorphic. 
\end{lemma}
\begin{proof} Since $D^{\dagger}(\mathfrak A)^{\sim}\cong D(\mathfrak A)$, $D^{\dagger}(\mathfrak A)^{\vee}(1)^{\sim}\cong D(\mathfrak A)^{\vee}(1)$ and $D(\mathfrak A)\cong D(\mathfrak A)^{\vee}(1)$, we get that $D^{\dagger}(\mathfrak A)$ and $D^{\dagger}(\mathfrak A)^{\vee}(1)$ are isomorphic by Kedlaya's full faithfulness theorem. Similarly in order to show that $R^1\pi_*(\mathcal O_{\mathfrak A})$ and $D^{\dagger}(\mathfrak A)$ are isomorphic, it will be enough to prove that $R^1\pi_*(\mathcal O_{\mathfrak A})^{\sim}$ and $D(\mathfrak A)\cong D^{\dagger}(\mathfrak A)^{\sim}$ are isomorphic. Let $\mathfrak U$ be a smooth formal lift of $U$ to Spf$(\mathcal O)$. Let $f:\mathfrak A\to\mathfrak U$ be the morphism of formal $\mathcal O$-schemes which is the composition of the structure map $\mathfrak A\to U$ and the closed immersion $U\to\mathfrak U$. We will let $R^*f_{\mathfrak U,\textrm{conv}*}$ denote the higher direct images of convergent $F$-isocrystals with respect to $f$, similarly to section 2 of \cite{Laz}. By Th\'eor\`eme 2.5.6 of \cite{BBM} on page 104 and Proposition 3.3.7 of \cite{BBM} on page 144 we know that $D(\mathfrak A)\cong R^1f_{\mathfrak U,\textrm{conv}*}(\mathcal O_{\mathfrak A,\mathbb Q})$. Since $R^1\pi_*(\mathfrak A)^{\sim}$ and $R^1f_{\mathfrak U,\textrm{conv}*}(\mathcal O_{\mathfrak A,\mathbb Q})$ are isomorphic by Lemmas 5.1 and 5.5 of \cite{Laz}, the claim follows. 
\end{proof}
Now let us start the proof of Proposition \ref{equivalence} in earnest! First assume that $T(B\times B,1,p)$ is true. Using that $\overline{\alpha}_B^1$ and
$\overline{\alpha}_{B\times B}^1$ are injective, we get that
\begin{eqnarray}\nonumber
\dim_{\mathbb K}(DC_A(B,B)\otimes\mathbb K) \!\!\!\!\!&=&\!\!\!\!\!
\dim_{\mathbb K}(NS(B,B)\otimes\mathbb K)-2\dim_{\mathbb K}(NS(B)
\otimes\mathbb K)\\
\nonumber \!\!\!\!\!&\geq&\!\!\!\!\!
\dim_{\mathbb K}(\mathcal H^2(B\times B/\mathbb K)^{F^f=
q\cdot\textrm{\rm id}})-
2\dim_{\mathbb K}(\mathcal H^2(B/\mathbb K)^{F^f=
q\cdot\textrm{\rm id}})\\
\nonumber \!\!\!\!\!&=&\!\!\!\!\!
\dim_{\mathbb K}(\left(\mathcal H^1(B/\mathbb K)\otimes_{\mathbb K}\mathcal H^1(B/\mathbb K)\right)^{F^f=q\cdot\textrm{\rm id}}),
\end{eqnarray}
and so $I(B,B,p)$ holds by Lemma \ref{5.12}. Assume now that $I(B,B,p)$ is true. Let $\Delta:B\to B\times B$ be the diagonal embedding, and let
$m,p_1,p_2:B\times B\to B$ denote the addition on $B$, and the projections onto the first and second factors, respectively. Consider the diagram
$$\xymatrix{
NS(B)\otimes \mathbb K \ar@<1ex>^-{\alpha}[r] \ar^{\overline{\alpha}_B^1}[d] &
DC_A(B, B)\otimes \mathbb K 
\ar^-{\beta}[l]\ar^{\overline{\alpha}_{B\times B}^1}[d] \\
\mathcal H^2(B/\mathbb K)^{F^f=q\cdot\textrm{\rm id}}
\ar@<1ex>^-{\alpha'}[r]&
\left(\mathcal H^1(B/\mathbb K)\otimes_{\mathbb K}\mathcal H^1(B/\mathbb K)\right)^{F^f=q\cdot\textrm{\rm id}}
\ar^-{\beta'}[l].}$$
where
$$\alpha=m^*-p_1^*-p_2^*,\ \alpha'=\mathcal H^2(m)-\mathcal H^2(p_1)-\mathcal H^2(p_2),\ \beta=\Delta^*,\ \beta'=\mathcal H^2(\Delta),$$
where the subscript ${}^*$ means the pull-back on N\'eron--Severi groups. The diagram is commutative in the sense that $\alpha'\circ\overline{\alpha}_{B}^1=\overline{\alpha}_{B\times B}^1\circ\alpha$ and
$\beta'\circ\overline{\alpha}_{B\times B}^1=\overline{\alpha}_{B}^1
\circ\beta$, and 
$$\beta'\circ\alpha'=\mathcal H^2(m\circ\Delta)-\mathcal H^2(p_1\circ\Delta)-\mathcal H^2(p_2\circ\Delta)=\mathcal H^2([2])-\mathcal H^2([1])-\mathcal H^2([1])$$
is the multiplication by $2$ map by Lemma \ref{homogenous}. By Matsusaka's theorem we get that $\beta\circ\alpha$ is also the multiplication by $2$ map. Therefore the left side of the diagram above is a direct summand of the right side; the claim now follows from Lemma \ref{5.12}.
\end{proof}
\begin{rem} Note that Theorem \ref{isogeny} implies de Jong's theorem (Theorem 2.6 of \cite{J2} on page 305) by reversing the argument in the first proof of Theorem \ref{isogeny} in chapter 4. So we've got a new, although not entirely independent proof of the latter.
\end{rem}

\section{Independence results}

For a moment let $K$ be a field of characteristic zero and denote by $ch:\textrm{GL}_n\to\mathbb G_m\times\mathbb A^{n-1}$ the morphism over $K$ associating to a matrix the coefficients of its characteristic polynomial. 
\begin{prop}[Larsen--Pink]\label{7.1} Let $G\subseteq\textrm{\rm GL}_n$ be a reductive algebraic subgroup over $K$, let $G^o$ be its identity component, and let $g\in G(K)$. Then the Zariski-closures of $ch(gG^o)$ and of $ch(G^o)$ are equal if and only if $g\in G^o(K)$.
\end{prop}
\begin{proof} This is Proposition 4.11 of \cite{LP1b} on page 618.
\end{proof}
For any $n$-uple $\underline d=(d_1,d_2,\ldots,d_n)\in\mathbb N^n$ let
$$\textrm{GL}_{\underline d,K}=\prod_{i=1}^n\textrm{GL}_{d_i,K}$$
denote $n$-fold product whose $i$-th factor $\textrm{GL}_{d_i,K}$ is the general linear group of rank $d_i$ over $K$. As usual let $\mathbb G_{m,K}$ denote $GL_{1,K}$. Now also assume that $K$ is an algebraically closed field.
\begin{thm}[Larsen--Pink]\label{larsen-pink} Let $G$ be a connected semi-simple algebraic subgroup of $\textrm{\rm GL}_{\underline d,K}$ such that each standard representation $G\to\textrm{\rm GL}_{d_i,K}$ (induced by the projection onto the $i$-th factor) is irreducible. Then the data assigning $\dim(U^G)$ to every $K$-linear representation of $\textrm{\rm GL}_{\underline d,K}$ on a finite dimensional vector space $U$ determines $G$ up to conjugation in $\textrm{\rm GL}_{\underline d,K}$. 
\end{thm}
\begin{proof} This is Theorem 4.2 of \cite{LP1b} on page 574, although there it is only stated in the case when $K=\mathbb C$. Also note that the condition in the theorem is not only sufficient, but necessary, too. 
\end{proof}
\begin{rem}\label{8.3a} We would like to make sense of the following principle of Larsen and Pink: "all representations of $\textrm{\rm GL}_{\underline d,K}$ are given by linear algebra." Let $K$ be again an arbitrary field of characteristic zero. Let $W$ be the vector-space underlying the tautological $K$-linear representation of $GL_{n,K}$ (i.e.~the $K$-dimension of $W$ is $n$). Then we have a $K$-linear action of $GL_{n,K}$ on
$$W^{\otimes m}=\underbrace{W\otimes_KW\otimes\cdots\otimes_KW}_{\textrm{$m$-times}},$$
the $m$-fold tensor product power of the tautological representation of $GL_{n,K}$. The permutation group $S_m$ acts $K$-linearly on $W^{\otimes m}$ via permuting the factors of the tensor product, and this action commutes with the action of $GL_{n,K}$. For every irreducible $K$-linear representation $\rho$ of $S_m$ let $\pi_{\rho}$ be the corresponding idempotent in the group ring $K[S_m]$. Then the image $W^{\rho}$ of the action of $\pi_{\rho}$ on $W^{\otimes m}$ is a $GL_{n,K}$-invariant subspace which is an isomorphic to the $r$-fold direct sum of an irreducible $K$-linear representation of $GL_{n,K}$ for some $r$ (depending on $n$ and $d$, of course). Moreover every irreducible $K$-linear representation of $GL_{n,K}$ arises this way. Since every $K$-linear representation of $GL_{n,K}$ is semi-simple, for some positive integer $d$ its $d$-fold direct sum is isomorphic to some direct sum of the representations $W^{\rho}$ (for various $m$ and $\rho$). A similar construction can be carried out for the direct product $GL_{\underline d,K}$. 
\end{rem}
\begin{rem}\label{8.3b} The relevance of the fact above to our problem of Tannakian nature is the following. Now let $\mathbf T$ be a $K$-linear Tannakian category and let $\mathcal F$ be an object of $\mathbf T$ of rank $n$. Then $S_m$ acts on $m$-fold tensor product power $\mathcal F^{\otimes m}$ of $\mathcal F$, which in turn induces a $K$-algebra embedding $K[S_m]\hookrightarrow\textrm{End}_{\mathbf T}(\mathcal F^{\otimes m})$. As above for every irreducible $K$-linear representation $\rho$ of $S_m$ let $\pi_{\rho}$ be the corresponding idempotent in $K[S_m]$, and let $\mathcal F^{\rho}$ be the image of the action of $\pi_{\rho}$ on $\mathcal F^{\otimes m}$. Now for every field extension $L/K$ and $L$-valued fibre functor $\omega$ on
$\mathbf T$ the action of $\pi(\mathbf T,\omega)$ on
$\omega(\mathcal F^{\rho})$ is the $L$-linear extension of the composition of $r$-fold direct sum of the representation $W_{\rho}$ with the homomorphism of $\pi(\mathbf T,\omega)$ into the $L$-linear automorphism group $GL_{n,L}$ of $\omega(\mathcal F)$ (where $r$ is the same as above). Moreover as we have remarked above we have a similar construction for every $K$-linear representation of $GL_{\underline d,K}$.
\end{rem}
\begin{defn}\label{7.2} We are going to extend Serre's definition (see \cite{Se}) of a strictly compatible system of $l$-adic Galois representations to involve overconvergent $F$-isocrystals. For every number field $E$ let $|E|_q$ denote the set of irreducible factors of the semi-simple $\mathbb Q_l$-algebra $E\otimes_{\mathbb Q}\mathbb Q_l$, where $l$ is either a prime number different from $p$, or $l=q$. For every $\lambda\in|E|_q$ let $E_{\lambda}$ denote the corresponding factor: it is a finite  extension of $\mathbb Q_l$ containing $E$. Let $\mathcal O^E_{\lambda}$ denote the valuation ring of $E_{\lambda}$. We will drop the supersript if the choice of $E$ is clear. Let $E^{(d)}_{\lambda}$ denote the unique unramified extension of $E_{\lambda}$ of degree $d$ (for every $d\in\mathbb N$). Note that the factors of $E\otimes_{\mathbb Q}\mathbb Q_l$ correspond to the places of $E$ over $l$ when $l$ is a prime number. 

Let $E$ be a number field such that factors of $E\otimes_{\mathbb Q}\mathbb Q_q$ are all totally ramified extensions of $\mathbb Q_q$. For every $\lambda\in|E|_q$ which is a factor of $E\otimes_{\mathbb Q}\mathbb Q_q$ let
$\sigma:E_{\lambda}\to E_{\lambda}$ be the identity map, let $\mathfrak U_{\lambda}$ be a smooth formal lift  of $U$ to Spf$(\mathcal O_{\lambda})$, and assume that there is a lift $F_{\sigma}$ of the $q$-power Frobenius of $U$, compatible with $\sigma$. By an $E$-compatible system (over $U$) we mean a collection $\{\mathcal F_{\lambda}|\lambda\in|E|_q\}$, where $\mathcal F_{\lambda}$ is a lisse $E_{\lambda}$-sheaf on $U$ for every place $\lambda\in|E|_q$ of $E$ not lying over $p$, and $\mathcal F_{\lambda}$ is an object of $\textrm{$F_{\sigma}$-Isoc}^{\dagger}(U/E_{\lambda})$, otherwise, and these sheaves are $E$-compatible with each another in the sense that for every $x\in|U|$ the polynomial
$\det(1-t^{\deg(x)}\cdot\textrm{Frob}_x(\mathcal F_{\lambda}))$ has coefficients in $E$ and it is independent of the choice of $\lambda$. Clearly this condition implies that the rank $n$ of $\mathcal F_{\lambda}$ is independent of $\lambda$, which we will call the rank of the $E$-compatible system. We say that an $E$-compatible system $\{\mathcal F_{\lambda}|\lambda\in|E|_q\}$ is pure of weight $w$, where $w\in\mathbb Z$, if the eigenvalues of $\textrm{Frob}_y(\mathcal F_{\lambda})$ have absolute value $q^{dw/2}$ for every complex embedding $\overline E_{\lambda}\to\mathbb C$, for every $\lambda\in|E|_q$, and for every $y\in|U|$ of degree $d$. We say that an $E$-compatible system $\{\mathcal F_{\lambda}|\lambda\in|E|_q\}$ is semi-simple if each $\mathcal F_{\lambda}$  is a semi-simple object in its respective category. 
\end{defn}
\begin{example}\label{7.3} Let $\pi:\mathfrak X\to U$ be a projective, smooth morphism of $\mathbb F_q$-schemes. For every prime number $l\neq p$ and for every natural number $n$ let $H^n(\mathfrak X)_l$ denote the $n$-th higher direct image $R^n\pi_*(\mathbb Q_l)$ of the constant $l$-adic sheaf $\mathbb Q_l$, and let $H^n(\mathfrak X)_q$ denote the oveconvergent $\mathbb Q_q$-linear $F$-isocrystal $R^n\pi_*(\mathcal O^{\dagger}_{\mathfrak X})$. Then $\{H^n(\mathfrak X)_l|l\in|\mathbb Q|_q\}$ is a pure $\mathbb Q$-compatible system of weight $n$ over $U$, as we saw in the proof of Lemma \ref{6.6}.
\end{example}
\begin{rem}\label{7.3b} We say that an $E$-compatible system $\{\mathcal F_{\lambda}|\lambda\in|E|_q\}$ is mixed, if for every $\lambda\in|E|_q$ the object $\mathcal F_{\lambda}$ has a filtration:
$$0=\mathcal F^{(0)}_{\lambda}\subset\mathcal F^{(1)}_{\lambda}\subset\cdots\subset\mathcal F^{(m_{\lambda})}_{\lambda}=
\mathcal F_{\lambda}$$
such that each successive quotient $\mathcal F^{(i+1)}_{\lambda}/\mathcal F^{(i)}_{\lambda}$ is pure of some weight in the sense defined above. When $\{\mathcal F_{\lambda}|\lambda\in|E|_q\}$ is both semi-simple and mixed then we have a unique direct sum decomposition:
$$\mathcal F_{\lambda}=\bigoplus_{w\in\mathbb Z}\mathcal F^{[w]}_{\lambda}$$
for every $\lambda\in|E|_q$ such that $\mathcal F^{[w]}_{\lambda}$ is the largest sub-object of $\mathcal F_{\lambda}$ which is pure of weight $w$. Note that $\{\mathcal F^{[w]}_{\lambda}|\lambda\in|E|_q\}$ is a pure $E$-compatible system of weight $w$.
\end{rem}
\begin{defn}\label{7.4} Let $E$ be a number field of the type above and let $\{\mathcal F_{\lambda}|\lambda\in|E|_q\}$ be an $E$-compatible system over $U$. Fix a point $x\in U(\mathbb F_{q^n})$ and let $\overline{x}\in U(\overline{\mathbb F}_{q^n})$ be lying over $x$; then for every $\lambda\in|E|_q$ which is a place of $E$ not lying above $p$ let $\rho_{\lambda}:\pi_1(U,\overline x)\to\textrm{GL}_n(E_{\lambda})$ be the representation corresponding to $\mathcal F_{\lambda}$, and let $G_{\lambda}$ denote the Zariski-closure of the image of $\rho_{\lambda}$. By construction $G_{\lambda}/G^o_{\lambda}$ is a finite quotient of the profinite group
$\pi_1(U,\overline x)$. When $\lambda\in|E|_q$ is a factor of $E\otimes_{\mathbb Q}\mathbb Q_q$ then let $G_{\lambda}$ denote the monodromy group $\textrm{Gr}(\mathcal F_{\lambda},x)$. As we explained in Remark \ref{cc_remark} the quotient $G_{\lambda}/G_{\lambda}^o$ can be considered as a  finite quotient of $\pi_1(U,\overline x)$, too. We will call $G_{\lambda}$ the $\lambda$-adic (arithmetic) monodromy group of the $E$-compatible system (with respect to the base point $x$).
\end{defn}
\begin{notn}\label{7.5} Let $\{\mathcal F_{\lambda}|\lambda\in|E|_q\}$ and $x,\overline x$ be as above. For every $\lambda\in|E|_q$ which is a place of $E$ not lying above $p$ let $V_{\lambda}$ denote the fibre $\overline x^*(\mathcal F_{\lambda})$ of $\mathcal F_{\lambda}$ over $\overline x$. When $\lambda\in|E|_q$ is a factor of $E\otimes_{\mathbb Q}\mathbb Q_q$ then let $V_{\lambda}$ denote the fibre $\omega_x(\mathcal F_{\lambda})$. We have a representation
$\rho_{\lambda}:\textrm{\rm Gr}(\mathcal F_{\lambda},x)\rightarrow
\textrm{Aut}(V_{\lambda})$ by definition. For every $\lambda\in|E|_q$ let
$\rho_{\lambda}^{\textrm{alg}}$ denote the representation of $G_{\lambda}^o$ on $V_{\lambda}$. Note that the isomorphism class of the triple
$(G^o_{\lambda},V_{\lambda},\rho^{alg}_{\lambda})$ is independent of the choice of $x$ (and $\overline x$) for every $\lambda\in|E|_q$ which is a place of $E$ not lying above $p$, and similarly the isomorphism class of the triple
$(G^o_{\lambda}\otimes_{E^{(n)}_{\lambda}}\overline E_{\lambda},V_{\lambda}\otimes_{E^{(n)}_{\lambda}}\overline E_{\lambda},\rho^{alg}_{\lambda}
\otimes_{E^{(n)}_{\lambda}}\overline E_{\lambda})$ is independent of the choice of $x$ for every factor $\lambda$ of $E\otimes_{\mathbb Q}\mathbb Q_q$. Therefore it is justified to drop $x$ from the notation, as far as the proofs of our results on independence are concerned. We will call the isomorphism class of $(G^o_{\lambda},V_{\lambda},\rho^{alg}_{\lambda})$ (when $\lambda$ is not lying above $p$) and $(G^o_{\lambda}\otimes_{E^{(n)}_{\lambda}}\overline E_{\lambda},V_{\lambda}\otimes_{E^{(n)}_{\lambda}}\overline E_{\lambda},\rho^{alg}_{\lambda}\otimes_{E^{(n)}_{\lambda}}\overline E_{\lambda})$ (when $\lambda$ is a factor of $E\otimes_{\mathbb Q}\mathbb Q_q$) the $\lambda$-adic monodromy triple of the compatible system.
\end{notn}
\begin{rem}\label{7.6} Assume for a moment that $X$ is any smooth over
$\mathbb F_q$ and let $\mathcal F$ be an object of $\textrm{\rm $F_{\sigma}$-Isoc}^{\dagger}(X/\mathbb L)$ (for some choice of $\mathbb L$, $\sigma$ and $F$). Let $\pi:Y\to X$ be an open immersion with Zariski-dense image and let $y\in Y(\mathbb F_{q^n})$. The the pull-back functor $\pi^*:\dal\mathcal F\dar\to\dal\pi^*(\mathcal F)\dar$ with respect to $\pi$ is a tensor equivalence between $\dal\mathcal F\dar$ and $\dal\pi^*(\mathcal F)\dar$ by Theorem \ref{extension_thm} and the global version of Kedlaya's full faithfulness theorem (see \cite{Ke1}), so the induced map $\textrm{Gr}(\pi^*(\mathcal F),y)\to\textrm{Gr}(\mathcal F,\pi(x))$ is an isomorphism. Since the same conservativity property holds for the monodromy group of lisse $K$-sheaves, where $K$ is a finite extension of $\mathbb Q_l$, with $l\neq p$, we may shrink the curve $U$ while we study the monodromy groups of $E$-compatible systems without the loss of generality. In particular our convenient assumption on the existence of the lift $F_{\sigma}$ is harmless.
\end{rem}
\begin{example}\label{7.7} Let $D/E$ be a finite extension such that the factors of $D\otimes_{\mathbb Q}\mathbb Q_q$ are all totally ramified extensions of $\mathbb Q_q$. Note that for every index $l$, where $l$ is either a prime number $l$ different from $p$, or $l=q$, we have $D\otimes_{\mathbb Q}\mathbb Q_l=D\otimes_E(E\otimes_{\mathbb Q}\mathbb Q_l)$, so every factor $D\otimes_{\mathbb Q}\mathbb Q_l$ is contained by the tensor product of $D$ and a factor of $E\otimes_{\mathbb Q}\mathbb Q_l$. Let $r:|D|_q\to|E|_q$ denote the corresponding map. For every $\lambda\in|D|_q$ which is a factor of $D\otimes_{\mathbb Q}\mathbb Q_q$ let $\sigma:D_{\lambda}\to D_{\lambda}$ also denote the identity map by slight abuse of notation, let $\mathfrak U_{\lambda}$ be the smooth formal lift $\mathfrak U_{r(\lambda)}\times_{\text{Spf}(\mathcal O_{r(\lambda)})}\text{Spf}(\mathcal O^D_{\lambda})$ of $U$ to Spf$(\mathcal O^D_{\lambda})$; then there is a lift $F_{\sigma}$ of the $q$-power Frobenius of $U$ to $\mathfrak U_{\lambda}$, compatible with $\sigma$. For every place $\lambda\in|D|_q$ of $D$ not lying over $p$ let $\mathcal F^D_{\lambda}$ denote the $D_{\lambda}$-linear extension $\mathcal F_{r(\lambda)}\otimes_{E_{r(\lambda)}}D_{\lambda}$ in the category of lisse sheaves, and let $\mathcal F^D_{\lambda}$ be the $D_{\lambda}$-linear extension $\mathcal F_{r(\lambda)}\otimes_{E_{r(\lambda)}}D_{\lambda}$ of $\mathcal F_{r(\lambda)}$ corresponding to the extension $D_{\lambda}/E_{r(\lambda)}$ and the choices of Frobenii which we have specified, otherwise. More explicitly $\mathcal F^D_{\lambda}$ is an object of $\textrm{$F_{\sigma}$-Isoc}^{\dagger}(U/D_{\lambda})$ for every factor of $D\otimes_{\mathbb Q}\mathbb Q_q$ (which we get from base change from $\mathfrak U_{r(\lambda)}$ to $\mathfrak U_{\lambda}$). Then $\{\mathcal F^D_{\lambda}|\lambda\in|D|_q\}$ is a $D$-compatible system of rank $n$ over $V$ which we will call the $D$-linear extension of the $E$-compatible system $\{\mathcal F_{\lambda}|\lambda\in|E|_q\}$. 
\end{example}
\begin{examples}\label{7.8} Let $E$ and $\{\mathcal F_{\lambda}|\lambda\in|E|_q\}$ be as above, and let $\mathbf{p}:V\to U$ be a finite \'etale cover of geometrically connected curves. Then $\{\mathbf{p}^*(\mathcal F_{\lambda})|\lambda\in|E|_q\}$ is an $E$-compatible system of rank $n$ over $V$. We will call this $E$-compatible system the pull-back of $\{\mathcal F_{\lambda}|\lambda\in|E|_q\}$ with respect to $\mathbf{p}$. Let $d$ be a positive integer and let $U^{(d)}$ let be base change of $U$ to Spec$(\mathbb F_{q^d})$. Then $\{\mathcal F_{\lambda}^{(d)}|\lambda\in|E|_q\}$ is an $E$-compatible system of rank $n$ over $U^{(d)}$. Note that the factors of $E\otimes_{\mathbb Q}\mathbb Q_{q^d}$ are all totally ramified extensions of $\mathbb Q_{q^d}$, so this definition makes sense. Also note that every factor of $E\otimes_{\mathbb Q}\mathbb Q_{q^d}=(E\otimes_{\mathbb Q}\mathbb Q_{q})\otimes_{\mathbb Q_q}\mathbb Q_{q^d}$ is equal to the tensor product of a factor of $E\otimes_{\mathbb Q}\mathbb Q_q$ and $\mathbb Q_{q^d}$, so there is a natural bijection between the sets $|E|_q$ and $|E|_{q^d}$. We will not distinguish between these in all that follows. It is a useful fact that these operations, along with $D$-linear extension, preserve the properties of being semi-simple and pure of a given weight.
\end{examples}
\begin{lemma}\label{reduction_lemma} The $\lambda$-adic monodromy triples of $\{\mathcal F_{\lambda}|\lambda\in|E|_q\}$ and $\{\mathbf{p}^*(\mathcal F_{\lambda})|\lambda\in|E|_q\}$ are isomorphic. The same holds for $\{\mathcal F_{\lambda}|\lambda\in|E|_q\}$ and $\{\mathcal F_{\lambda}^{(d)}|\lambda\in|E|_q\}$, too. 
\end{lemma}
\begin{proof} The claims are known for those $\lambda\in|E|_q$ which are places of $E$ not lying above $p$, and the first claim follows from Corollary \ref{4.20} for factors of $E\otimes_{\mathbb Q}\mathbb Q_q$. Now fix a point $y:\textrm{Spec}(\mathbb F_{q^n})\to U$ of degree $n$. The homomorphism $\phi$ of Proposition \ref{n-power_sequence} furnishes an isomorphism $\textrm{Gr}(\mathcal F_{\lambda}^{(d)},y)^o
\to\textrm{Gr}(\mathcal F_{\lambda},y)^o$. The second claim is now clear.
 \end{proof}
\begin{rem}\label{7.9} Note that a similar claim holds for $D$-linear extensions, where $D/E$ is a finite extension of the type considered above. More precisely by slight abuse of notation for every place $\lambda\in|D|_q$ of $D$ not lying over $p$ let
$(G^o_{\lambda},V_{\lambda},\rho^{alg}_{\lambda})$ denote the $\lambda$-adic triple of 
$\{\mathcal F^D_{\lambda}|\lambda\in|D|_q\}$, and for every factor $\lambda$ of $D\otimes_{\mathbb Q}\mathbb Q_q$ let $(G^o_{\lambda}\otimes_{D^{(n)}_{\lambda}}\overline D_{\lambda},V_{\lambda}\otimes_{D^{(n)}_{\lambda}}\overline D_{\lambda},\rho^{alg}_{\lambda}\otimes_{D^{(n)}_{\lambda}}\overline D_{\lambda})$ denote the $\lambda$-adic triple of $\{\mathcal F^D_{\lambda}|\lambda\in|D|_q\}$. Then for every place $\lambda\in|D|_q$ of $D$ not lying over $p$ the triples
$(G^o_{\lambda},V_{\lambda},\rho^{alg}_{\lambda})$ and $(G^o_{\lambda}\otimes_{E_{r(\lambda)}}D_{\lambda},V_{\lambda}\otimes_{E_{r(\lambda)}}D_{\lambda},\rho^{alg}_{\lambda}\otimes_{E_{r(\lambda)}} D_{\lambda})$ are isomorphic, and for every factor $\lambda$ of $D\otimes_{\mathbb Q}\mathbb Q_q$ the triples 
$(G^o_{\lambda}\otimes_{D^{(n)}_{\lambda}}\overline D_{\lambda},V_{\lambda}\otimes_{D^{(n)}_{\lambda}}\overline D_{\lambda},\rho^{alg}_{\lambda}\otimes_{D^{(n)}_{\lambda}}\overline D_{\lambda})$ and $(G^o_{r(\lambda)}\otimes_{E^{(n)}_{r(\lambda)}}\overline E_{r(\lambda)},V_{r(\lambda)}\otimes_{E^{(n)}_{r(\lambda)}}\overline E_{r(\lambda)},\rho^{alg}_{r(\lambda)}\otimes_{E^{(n)}_{r(\lambda)}}\overline E_{r(\lambda)})$ are isomorphic, too. This is trivial when $\lambda$ is not lying over $p$, and follows from Lemma \ref{coefficients}, otherwise.
\end{rem}
The following proposition, our first on independence, is an extension of a classical result of Serre to $p$-adic monodromy groups.
\begin{prop}\label{connected_components} The quotient $G_{\lambda}/G^o_{\lambda}$ is independent of $\lambda$. In particular if $G_{\lambda}$ is connected for some $\lambda$ then it is so for all $\lambda$.
\end{prop}
\begin{proof} It will be sufficient to prove that for every factor $\lambda$ of $E\otimes_{\mathbb Q}\mathbb Q_q$ and for every place $\kappa\in|E|_q$ not lying over $p$ the quotients $G_{\lambda}/G^o_{\lambda}$ and $G_{\kappa}/G^o_{\kappa}$ are the same as quotients of $\textrm{\rm Gr}(\textrm{\rm $F_{\sigma}$-Isoc}^{\dagger}(U/E_{\lambda}),x)$. So fix $\lambda$ now. The affine group scheme $\textrm{\rm Gr}(\textrm{\rm $F_{\sigma}$-Isoc}^{\dagger}(U/E_{\lambda}),x)$ is the projective limit of the monodromy groups $\textrm{\rm Gr}(\mathcal F,x)$ where $\mathcal F$ runs through all objects of $\textrm{\rm $F_{\sigma}$-Isoc}^{\dagger}(U/E_{\lambda})$. We will define a topology on $G=\textrm{\rm Gr}(\textrm{\rm $F_{\sigma}$-Isoc}^{\dagger}(U/E_{\lambda}),x)(\overline E_{\lambda})$ as follows. A subset $Z\subseteq G$ if and only if  there is an $\mathcal F$ as above such that $Z$ is the pre-image of a Zariski closed subset $Z'\subseteq\textrm{\rm Gr}(\mathcal F,x)(\overline E_{\lambda})$ with respect to
the projection
$$\rho_{\mathcal F}:G=\textrm{\rm Gr}(\textrm{\rm $F_{\sigma}$-Isoc}^{\dagger}(U/E_{\lambda}),x)(\overline E_{\lambda})\to\textrm{\rm Gr}(\mathcal F,x)(\overline E_{\lambda}).$$
\begin{lemma}\label{7.12} Let $g\in G$ and let $H\subseteq G$ be an open subgroup. Then the set
$$\mathcal H(g)=\{gH\cap\textrm{\rm Frob}_y|y\in|U|\}$$
is dense in $gH$. 
\end{lemma}
\begin{proof} It will be sufficient to prove that for every such object $\mathcal F$ of
$\textrm{\rm $F_{\sigma}$-Isoc}^{\dagger}(U/E_{\lambda})$ the image of $\mathcal H(g)$ is Zariski-dense in the image of $gH$ with respect to the surjective homomorphism $\rho_{\mathcal F}:G\to\textrm{\rm Gr}(\mathcal F,x)(\overline E_{\lambda})$. By assumption there is an open subgroup scheme $\mathbf H\subseteq\textrm{\rm Gr}(\textrm{\rm $F_{\sigma}$-Isoc}^{\dagger}(U/E_{\lambda}),x)$ of finite index such that $H=\mathbf H(\overline E_{\lambda})$. So there is an object $\mathcal C$ of
$\textrm{\rm $F_{\sigma}$-Isoc}^{\dagger}(U/E_{\lambda})$ which corresponds to a faithful finite-dimensional representation of the quotient $G/H$, that is, the kernel of the projection $\rho_{\mathcal C}$ is $H$. Since $\rho_{\mathcal F}$ factors through $\rho_{\mathcal F\oplus\mathcal C}$, it is enough to prove the claim above for $\mathcal F\oplus\mathcal C$. Since $\rho_{\mathcal C}$ also factors through $\rho_{\mathcal F\oplus\mathcal C}$, the kernel of $\rho_{\mathcal F\oplus\mathcal C}$ is contained by $H$. So we may assume without the loss of generality that kernel of $\rho_{\mathcal F}$ is contained by $H$. In this case:
$$\rho_{\mathcal F}(\mathcal H(g))=\{\rho_{\mathcal F}(g)\rho_{\mathcal F}(H)\cap\textrm{\rm Frob}_y(\mathcal F)|y\in|U|\},$$
so the claim follows from Theorem \ref{chebotarev}. 
\end{proof}
Now for every $\kappa$ which is either a place $\in|E|_q$ not lying over $p$ or is $\lambda$ let $\Gamma_{\kappa}\subseteq G$ be the kernel of the surjective homomorphism $G\to G_{\kappa}/G_{\kappa}^o$. It will be enough to show that for every $g\in G$ whether $g\in\Gamma_{\kappa}$ is independent of $\kappa$. So in order to conclude the proof of Proposition \ref{connected_components} it will be enough, by Proposition \ref{7.1}, to show the following\noqed
\end{proof}
\begin{lemma}\label{7.13} The Zariski-closure of $ch(\rho_{\kappa}(g)G_{\kappa}^o)$ is independent of the choice of $\kappa$. 
\end{lemma}
\begin{proof} Let $\kappa$ be any place in $|E|_q$ not lying over $p$ and set $H=\Gamma_{\lambda}\cap\Gamma_{\kappa}$. It is an open subgroup of $G$ contained by both $\Gamma_{\lambda}$ and $\Gamma_{\kappa}$. By Lemma \ref{7.12} above the set
$$\mathcal H(g)=\{gH\cap\textrm{\rm Frob}_y|y\in|U|\}$$
is dense in $gH$. So in particular $\rho_{\mathcal F_{\lambda}}(\mathcal H(g))$ is Zariski-dense in $\rho_{\mathcal F_{\lambda}}(gH)$. Since $H$ is an open subgroup of finite index in $\Gamma_{\lambda}$, and $G^o_{\lambda}$ is connected, we get that the Zariski-closure of $\rho_{\mathcal F_{\lambda}}(\mathcal H(g))$ is $gG^o_{\lambda}$. By
Remark \ref{cc_remark} the profinite group $\pi_1(U,\overline x)$ is the quotient of $G$, and the quotient map $\pi^o:G\to\pi_1(U,\overline x)$ is continuous with respect the profinite topology on $\pi_1(U,\overline x)$ and the topology on $G$ introduced at the beginning of the proof of Proposition \ref{connected_components}. Therefore
$\rho_{\kappa}\circ\pi^o(\mathcal H(g))$ is dense in $\rho_{\kappa}\circ\pi^o(gH)$ in the $\kappa$-adic topology, where $\rho_{\kappa}:\pi_1(U,\overline x)\to\textrm{GL}_n(E_{\kappa})$ is the representation corresponding to $\mathcal F_{\kappa}$, as in Definition \ref{7.4}. Since $H$ is an open subgroup of finite index in $\Gamma_{\kappa}$, and $G^o_{\kappa}$ is connected, we get that the Zariski-closure of $\rho_{\kappa}\circ\pi^o(\mathcal H(g))$ is $gG^o_{\kappa}$. By Lemma \ref{compatible_frobenii} we have $ch(\rho_{\mathcal F_{\lambda}}(\mathcal H(g)))=ch(\rho_{\kappa}\circ\pi^o(\mathcal H(g)))$, so the claim is now clear. 
\end{proof}
\begin{defn}\label{7.14} Let $\overline U$ denote the base change of $U$ to $\overline{\mathbb F}_q$. Assume now that $U$ is geometrically connected, and consider Grothendieck's short exact sequence of \'etale fundamental groups for $U$:
\begin{equation}\label{fundamentalsequence}
\CD1@>>>\pi_1(\overline U,\overline x)@>>>
\pi_1(U,\overline x)@>>>\textrm{Gal}(\overline{\mathbb F}_q/\mathbb F_q)@>>>1,\endCD
\end{equation}
which is an exact sequence of profinite groups in the category of topological groups. Let 
$G^{geo}_{\lambda}$ denote the Zariski closure of $\pi_1(\overline U,\overline x)$ with respect to $\rho_{\lambda}$, if $\lambda$ is a place $\in|E|_q$ not lying over $p$, and let $G^{geo}_{\lambda}$ denote the image of $\textrm{DGal}(\mathcal F_{\lambda},x)$  with respect to $\rho_{\lambda}$, otherwise. We will call $G^{geo}_{\lambda}$ the $\lambda$-adic geometric monodromy group. It is a normal subgroup of $G_{\lambda}$. When the $E$-compatible system $\{\mathcal F_{\lambda}|\lambda\in|E|_q\}$ is pure then $G^{geo}_{\lambda}$ is semi-simple, and its  identity component is the derived group of $G^o_{\lambda}$. This is well-known for $\lambda\in|E|_q$ not lying over $p$ (see 1.3.9 and 3.4.1 of \cite{De}), and follows from Proposition \ref{3.15}, otherwise.
\end{defn}
\begin{examples}\label{7.15} Let $E$ and $\{\mathcal F_{\lambda}|\lambda\in|E|_q\}$ be as above. For every integer $n\in\mathbb Z$ let $\mathcal F_{\lambda}(n)$ denote the $n$-th Tate twist of $\mathcal F_{\lambda}$ for every $\lambda\in|E|_q$. Recall that for every factor $\lambda$ of $E\otimes_{\mathbb Q}\mathbb Q_q$ this means that $\mathcal F_{\lambda}(n)^{\wedge}$ is the same as $\mathcal F_{\lambda}^{\wedge}$, and the Frobenius of  $\mathcal F_{\lambda}(n)$ is $q^{-n}$ times the Frobenius of $\mathcal F_{\lambda}$. Then $\{\mathcal F_{\lambda}(n)|\lambda\in|E|_q\}$ is an $E$-compatible system which is pure of weight $w-2n$ if $\{\mathcal F_{\lambda}|\lambda\in|E|_q\}$ is pure of weight $w$. Moreover let $\mathcal F_{\lambda}^{\vee}$ denote the dual of $\mathcal F_{\lambda}$ for every $\lambda\in|E|_q$ (in its respective Tannakian category).  Then $\{\mathcal F_{\lambda}^{\vee}|\lambda\in|E|_q\}$ is an $E$-compatible system which is pure of weight $-w$ if $\{\mathcal F_{\lambda}|\lambda\in|E|_q\}$ is pure of weight $w$.
\end{examples}
\begin{prop}\label{invariants} Let $\{\mathcal F_{\lambda}|\lambda\in|E|_q\}$ be a semi-simple pure $E$-compatible system of weight $w$ over $U$. Then the $E_{\lambda}$-dimension of the space of invariants $V_{\lambda}^{G_{\lambda}^{geo}}$ (when $\lambda\in|E|_q$ is not lying over $p$) and the $E^{(n)}_{\lambda}$-dimension of the space of invariants $V_{\lambda}^{G_{\lambda}^{geo}}$ (when $\lambda$ is a factor $E\otimes_{\mathbb Q}\mathbb Q_q$) is independent of $\lambda$.
\end{prop}
\begin{proof} The $L$-function:
$$L(U,\mathcal F_{\lambda}^{\vee}(1),t)=\prod_{y\in|U|}\det(1-t^{\deg(y)}\cdot
\textrm{Frob}_y(\mathcal F_{\lambda}^{\vee}(1)))^{-1}$$
of $\mathcal F_{\lambda}^{\vee}(1)$ is a power series with coefficients in $E$ and it is independent of the choice of $\lambda$. It is a rational function by the Grothendieck--Verdier trace formula when $\lambda$ is a place of $E$ not above $p$, and by the Etesse-Le~Stum trace formula (see Theorem 6.3 of \cite{ES} on pages 570--571) and the finiteness of rigid cohomology in coefficients (see \cite{Ke1b}), otherwise. Therefore it will be enough to show that dimension of $V_{\lambda}^{G_{\lambda}^{geo}}$ is equal to sum of the orders of poles of all Weil numbers of weight $-w$. This was proved for $\lambda$ which is a place of $E$ not above $p$ in \cite{LP2} (see the proof of Proposition 2.1 on page 566). We are going to use the analogous argument to show the same for every irreducible factor of $E\otimes_{\mathbb Q}\mathbb Q_q$. 

By the Etesse-Le~Stum trace formula:
$$L(U,\mathcal F_{\lambda}^{\vee}(1),t)=\prod_{i=0}^2\det(1-t\cdot F^f|H^i_{rig,c}(U,\mathcal F_{\lambda}^{\vee}(1)))^{(-1)^{i+1}}.$$
Since $\mathcal F_{\lambda}^{\vee}(1)$ is pure of weight $-2-w$, for $i=0,1$ the group $H^i_{rig,c}(U,\mathcal F_{\lambda}^{\vee}(1))$ is mixed of weight $\leq-2-w+i$ by Theorem 5.3.2 of \cite{Ke2} on page 1445; in particular it has weights $<-w$. Therefore the first two factors of the product above do not contribute to the order of the pole at any Weil number of weight $w$. So the sum of the orders of poles of all Weil numbers of weight $-w$ is just the dimension of $H^2_{rig,c}(U,\mathcal F_{\lambda}^{\vee}(1))$ over $E_{\lambda}$. Since there is a perfect pairing:
$$\mathcal F_{\lambda}\otimes\mathcal F^{\vee}_{\lambda}(1)\longrightarrow
\mathcal O^{\dagger}_U(1),$$
by Poincar\'e duality (see Theorem 9.5 of \cite{Cr2} on pages 753--754) we get that the latter is the $E_{\lambda}$-dimension of $H^0_{rig}(U,\mathcal F_{\lambda})$. The latter is the rank of the largest trivial sub-object of $\mathcal F_{\lambda}^{\wedge}$. This is the same as the $E^{(n)}_{\lambda}$-dimension of the space of $\textrm{DGal}(\mathcal F_{\lambda},x)$-invariants of $\omega_x(\mathcal F_{\lambda})$.
\end{proof}
\begin{example}\label{8.20} Let $\{\mathcal F_{\lambda}|\lambda\in|E|_q\}$ be an $E$-compatible system of rank $n$. Every $E$-linear representation $\rho$ of $GL_{n,E}$ gives rise to an $E_{\lambda}$-linear representation of $GL_{n,E_{\lambda}}$, by base change, which we will denote by the same symbol by abuse of notation. We may apply the construction of Remark \ref{8.3b} to each $\mathcal F_{\lambda}$ (corresponding to the representation $\rho$) to get a $E$-compatible system $\{\mathcal F^{\rho}_{\lambda}|\lambda\in|E|_q\}$ which is semi-simple if $\{\mathcal F_{\lambda}|\lambda\in|E|_q\}$ is, and if the latter is also mixed, the the former is mixed, too. If the representation $\rho$ is faithful, then both the arithmetic monodromy groups and the geometric mondromy groups of the new $E$-compatible system are isomorphic to that of the original system.
\end{example}
We say that a finite extension $D/E$ of number fields is $d$-admissible (where $d$ is a positive integer) if all irreducible factors of $D\otimes_{\mathbb Q}\mathbb Q_{q^d}$ are totally ramified extensions of $\mathbb Q_{q^d}$. In this case all irreducible factors of $E\otimes_{\mathbb Q}\mathbb Q_{q^d}$ are totally ramified extensions of $\mathbb Q_{q^d}$, too, so the map $r:|D|_{q^d}\to|E|_{q^d}$ in Example \ref{7.7} is well-defined.
\begin{prop}\label{components} Let $\{\mathcal F_{\lambda}|\lambda\in|E|_q\}$ be a semi-simple pure $E$-compatible system of rank $n$ over $U$. If $G_{\lambda}^{geo}$ is connected for some $\lambda$, then it is so for every $\lambda$. 
\end{prop}
\begin{proof} Assume that there are $\delta,\kappa\in|E|_q$ such that $G_{\delta}^{geo}$ is connected, while $G_{\kappa}^{geo}$ is not.  Then by Lemma 2.3 of \cite{LP2} there is a representation $\rho$ of $GL_{n,E_{\kappa}}$ such that the dimensions of the $G_{\kappa}^{geo}$-invariants and of the $(G_{\kappa}^{geo})^o$-invariants on the fibre of $\mathcal F_{\kappa}^{\rho}$ at $x$ are different. By switching to $\{\mathcal F^{(d)}_{\lambda}|\lambda\in|E|_q\}$ for a suitably divisible $d$, taking a finite extension $K/E$, which we may assume to be $d$-admissible, and switching to the $K$-linear extension of $\{\mathcal F^{(d)}_{\lambda}|\lambda\in|E|_q\}$, we may assume that $\rho$ is actually the base change of an $E$-linear representation of $GL_{n,E}$ to $E_{\lambda}$, since these operations do not change the geometric monodromy groups. (This fact is well-known for those elements of $|E|_q$ which do not lie over $p$, and follows from Lemma \ref{coefficients}, otherwise.) By slight abuse of notation let $\rho$ denote the latter representation, too. By switching to $\{\mathcal F^{\rho}_{\lambda}|\lambda\in|E|_q\}$, we may assume that the dimensions of the $G_{\kappa}^{geo}$-invariants and of the $(G_{\kappa}^{geo})^o$-invariants on the fibre of $\mathcal F_{\lambda}$ at $x$ are different without the loss of generality.

For every  finite \'etale cover $\mathbf{p}:V\to U$ of geometrically connected curves and for every $\lambda\in|E|_q$ let $G_{\lambda}^{geo}(\mathbf{p})$ denote the base change of the $\lambda$-adic geometric monodromy group of the pull-back of $\{\mathcal F_{\lambda}|\lambda\in|E|_q\}$ with respect to
$\mathbf{p}$ to $\overline E_{\lambda}$. Let $\overline V_{\lambda}$ denote $V_{\lambda}\otimes_{E_{\lambda}}\overline E_{\lambda}$, if $\lambda\in|E|_q$ is not lying over $p$, and  $V_{\lambda}\otimes_{E^{(n)}_{\lambda}}\overline E_{\lambda}$, otherwise. The map $\mathbf p$ induces a homomorphism $G_{\lambda}^{geo}(\mathbf{p})\to G_{\lambda}^{geo}$ which is an open immersion: this is known, when $\lambda\in|E|_q$ is not lying over $p$, and follows from Corollary \ref{4.21}, otherwise. This implies that $G_{\delta}^{geo}(\mathbf{p})\to G_{\delta}^{geo}$ induced by $\mathbf p$ is an isomorphism for every $\mathbf p$. On the other there is a $\mathbf p$ such that the image of $G_{\delta}^{geo}(\mathbf{p})\to G_{\delta}^{geo}$ is $(G_{\kappa}^{geo})^o$. The latter is known, when $\lambda\in|E|_q$ is not lying over $p$, and follows from Proposition \ref{geometric}, otherwise. By Proposition \ref{invariants} we have:
$$\dim(\overline V_{\delta}^{G_{\delta}^{geo}})=
\dim(\overline V_{\kappa}^{G_{\kappa}^{geo}})
\textrm{ and }
\dim(\overline V_{\delta}^{G_{\delta}^{geo}(\mathbf p)})=
\dim(\overline V_{\kappa}^{G_{\kappa}^{geo}(\mathbf p)}),$$
on the other hand by the above:
$$\dim(\overline V_{\delta}^{G_{\delta}^{geo}})=
\dim(\overline V_{\delta}^{G_{\delta}^{geo}(\mathbf p)})
\textrm{ and }
\dim(\overline V_{\kappa}^{G_{\kappa}^{geo}})\neq
\dim(\overline V_{\kappa}^{G_{\kappa}^{geo}(\mathbf p)})$$
which is a contradiction.
\end{proof}
\begin{thm}\label{chin_independence} Let $\{\mathcal F_{\lambda}|\lambda\in|E|_q\}$ be a semi-simple pure $E$-compatible system of weight $w$ over $U$. Then there exists a finite extension $K/E$, a connected split semi-simple algebraic group $\mathcal G$ over $K$ and a $K$-linear vector space $V$ equipped with a $K$-linear representation $\rho$ of $\mathcal G$ such that for some positive integer $d$ with $K/E$ being $d$-admissible, for every $\lambda\in|K|_{q^d}$ not lying over $p$ the triples:
$$(\mathcal G\otimes_K K_{\lambda},V\otimes_KK_{\lambda},
\rho\otimes_KK_{\lambda})$$
and
$$(G_{r(\lambda)}^o\otimes_{E_{r(\lambda)}}K_{\lambda},V_{r(\lambda)}\otimes_{E_{r(\lambda)}}K_{\lambda},\rho_{r(\lambda)}^{\textrm{\rm alg}}\otimes_{E_{r(\lambda)}}K_{\lambda})$$
are isomorphic, and for every $\lambda\in|K|_{q^d}$ which is a
factor $\lambda$ of $K\otimes_{\mathbb Q}\mathbb Q_{q^d}$:
$$(\mathcal G\otimes_KK^{(n)}_{\lambda},V\otimes_K K_{\lambda}^{(n)},
\rho\otimes_KK^{(n)}_{\lambda})$$
and
$$(G_{r(\lambda)}^o\otimes_{E^{(n)}_{r(\lambda)}}K^{(n)}_{\lambda},V_{r(\lambda)}\otimes_{E^{(n)}_{r(\lambda)}}K^{(n)}_{\lambda},\rho_{r(\lambda)}^{\textrm{\rm alg}}\otimes_{E^{(n)}_{r(\lambda)}}K^{(n)}_{\lambda})$$
are isomorphic.
\end{thm}
\begin{proof} By Chin's main result (see Theorem 1.4 of \cite{Ch} on page 724) there is a finite extension $K/E$, a connected split semi-simple algebraic group $\mathcal G$ over $K$ and a $K$-linear vector space $V$ equipped with a $K$-linear representation $\rho$ of $\mathcal G$ such that the condition in the claim holds for every $\lambda\in|K|_{q^d}$ not lying over $p$. Now let $d$ be a positive integer such that $K/E$ is $d$-admissible. We may assume that $d$ is actually $1$ by switching to the compatible system $\{\mathcal F_{\lambda}^{(d)}|\lambda\in|E|_q\}$ and using Lemma \ref{reduction_lemma}. Similarly we may assume that our chosen base point $x$ has degree one. We may also assume that $K=E$ using Remark \ref{7.9}. By taking a suitable finite \'etale covering of $U$ and using Lemma \ref{reduction_lemma} we may even assume that the arithmetic monodromy groups are connected, and using Lemma \ref{components} we can assume that the geometric monodromy groups are connected, too.

Taking an additional finite extension of $E$, if it is necessary, we can also assume that the representation $\rho$ is a direct sum of absolutely irreducible representations:
$$\rho=\bigoplus_{j=1}^r\rho_j.$$
This decomposition induces another decomposition:
$$\mathcal F_{\lambda}\cong\bigoplus_{j=1}^r\mathcal F_{j,\lambda}$$
for every $\lambda\in|E|_q$ not lying over $p$ such that for every index $j$ the collection $\{\mathcal F_{j,\lambda}|\lambda\in|E|_q, 
\lambda\!\not|p\}$ is a semi-simple pure $E$-compatible system of weight $w$ over $U$ in the usual sense. By Abe's main result in \cite{Abe} we may assume that this collection extends to a full semi-simple pure $E$-compatible system $\{\mathcal F_{j,\lambda}|\lambda\in|E|_q\}$ of weight $w$ over $U$, at the prize of extending $E$ further, switching to compatible systems over
$U^{(d)}$, and arguing as above. 

Then $\{\oplus_{j=1}^r\mathcal F_{j,\lambda}|\lambda\in|E|_q\}$ is also a semi-simple pure $E$-compatible system of weight $w$ over $U$, so we get that for every $\lambda\in|E|_{q}$ which is a
factor $\lambda$ of $E\otimes_{\mathbb Q}\mathbb Q_{q}$ the $F$-isocrystals $\oplus_{j=1}^r\mathcal F_{j,\lambda}$ and
$\mathcal F_{\lambda}$ are compatible in the sense that for every $y\in|U|$ the polynomials:
$$\det(1-t^{\deg(y)}\cdot\textrm{Frob}_y(\bigoplus_{j=1}^r\mathcal F_{j,\lambda}))\textrm{ and }
\det(1-t^{\deg(y)}\cdot\textrm{Frob}_y(\mathcal F_{\lambda}))$$
are equal. It is a consequence of  the $p$-adic Chebotar\"ev density theorem (Theorem \ref{chebotarev}) that this implies that
$$\mathcal F_{\lambda}\cong\bigoplus_{j=1}^r\mathcal F_{j,\lambda}$$
for every such $\lambda$, too (see Corollary 10.2 of \cite{HP}). Note that, by construction, all $\mathcal F_{j,\lambda}$ are absolutely irreducible when $\lambda$ is not over $p$. The same is true for all other $\lambda$, too. Indeed assume that this is not the case; then after taking a finite extension of $E$, if it is necessary, there is a factor $\lambda$ of $E\otimes_{\mathbb Q}\mathbb Q_{q}$ such that $\mathcal F_{j,\lambda}\cong\mathcal F_{j,\lambda}^<\oplus\mathcal F_{j,\lambda}^>$ for some $F$-isocrystals of positive rank. Fix a $\kappa\in|E|_q$ not over $p$ and apply Abe's theorem quoted above to get two lisse semi-simple $E_{\kappa}$-sheaves $\mathcal F_{j,\kappa}^<$ and $\mathcal F_{j,\kappa}^>$ on $U$ compatible with $\mathcal F_{j,\lambda}^<$ and $\mathcal F_{j,\lambda}^>$, respectively (again possibly after taking a finite extension of $E$, etc.). By the usual Chebotar\"ev density theorem we have $\mathcal F_{j,\kappa}\cong\mathcal F_{j,\kappa}^<\oplus\mathcal F_{j,\kappa}^>$, but this is a contradiction.

For every $j$ let $d_j$ be the rank of the $E$-compatible system
$\{\mathcal F_{j,\lambda}|\lambda\in|E|_q\}$ and set $\underline d=(d_1,d_2,\ldots,d_r)\in\mathbb N^m$. Similarly to Notation \ref{7.5} let $V_{j,\lambda}$ denote the fibre of $\mathcal F_{j,\lambda}$ with respect to $\overline x$, when $\lambda\in|K|_{q^d}$ is not lying over $p$, and with respect to $x$, otherwise. For every $j$ let $V_j$ be the $K$-linear vector space underlying the representation $\rho_j$. Choosing a basis for each $V_{j}$ and $V_{j,\lambda}$, the group $\mathcal G$ (respectively $G_{\lambda}$) become a subgroup of $GL_{\underline d,E}$ (respectively $GL_{\underline d,E_{\lambda}}$), unique up to conjugation. We only need to show that for every
$\lambda\in|E|_q$ the base change of $\mathcal G$ and $G_{\lambda}$ to $\overline E_{\lambda}$ are conjugate as subgroups of $GL_{\underline d,\overline E_{\lambda}}$. Note that we already know this for all $\lambda$ not over $p$.

For every linear algebraic group $\mathbf G$ let $\mathbf G'$ and $Z(\mathbf G)$ denote the derived group of $\mathbf G$ and the centre of $\mathbf G$, respectively. Note that $Z(\mathcal G)$ is a subgroup of $Z(GL_{\underline d,E})$ since all representations $\rho_j$ are irreducible. Similarly $Z(G_{\lambda})$ is a subgroup of $Z(GL_{\underline d,E_{\lambda}})$. Therefore these groups are invariant under conjugation, and hence the same holds for their identity components, too. Since $\mathcal G$ (respectively $G_{\lambda}$) is the product of $\mathcal G'$ and $Z(\mathcal G)^o$ (respectively of $G_{\lambda}'$ and $Z(G_{\lambda})^o$), it will be enough to show that $\mathcal G'$ and $G_{\lambda}'$ are conjugate after base change to $\overline E_{\lambda}$, and similarly $Z(\mathcal G)^o$ and $Z(G_{\lambda})^o$ are equal after base change to $\overline E_{\lambda}$.

In order to prove the first, by the remark at the end of Definition \ref{7.14} it suffices to prove that for any pair $\lambda,\kappa\in|E|_q$, after fixing an isomorphism $\overline E_{\lambda}\cong\overline E_{\kappa}$, the subgroups
$$G_{\lambda}^{geo}\times_{E_{\lambda}}\overline E_{\lambda}
\subseteq GL_{\underline d,\overline E_{\lambda}}\textrm{ and }
G_{\lambda}^{geo}\times_{E_{\kappa}}\overline E_{\kappa}\subseteq 
GL_{\underline d,\overline E_{\kappa}}$$
are all conjugate. Every representation of $GL_{\underline d,\overline E_{\lambda}}\cong GL_{\underline d,\overline E_{\kappa}}$, or after a finite extension of $E$, of $GL_{\underline d,E}$, on a finite dimensional vector space $W$ can be obtained from the standard representations by means of linear algebra (as explained in Remark \ref{8.3a}). Thus it gives rise to a semi-simple pure $E$-compatible system (see Example \ref{8.20}) to which we can apply Proposition \ref{invariants}. It follows that the dimension of invariants of $G_{\lambda}^{geo}$ in $V^{\rho}_{\lambda}$ is independent of $\lambda$. The desired assertion is now a consequence of Theorem \ref{larsen-pink}.

Taking the determinant in each factor, the connected group $G_{\lambda}$ maps onto a subtorus $T_{\lambda}$ of the product of multiplicative groups
$\mathbb G^r_{m,E_{\lambda}}$. Every character of $\mathbb G^r_{m,E_{\lambda}}$, or equivalently, of $\mathbb G^r_{m,E}$, gives rise to an $E$-compatible system on $U$, so the question whether it is trivial on $T_{\lambda}$ is independent of $\lambda$ (for example by the usual and the $p$-adic Chebotar\"ev theorems). Thus each $T_{\lambda}=T\times_EE_{\lambda}$ for some subtorus $T\subseteq\mathbb G^r_{m,E}$. The centre of
$\textrm{GL}_{\underline d,E_{\lambda}}$ maps onto $\mathbb G^r_{m,E_{\lambda}}$, and the identity component of the pre-image of $T_{\lambda}$ is $Z(G_{\lambda})^o$. In other words, these identity components come from a fixed torus in the centre of $\textrm{GL}_{\underline d,E}$.
\end{proof}
By slightly modifying the argument above, we may give a proof of Theorem \ref{strong-independence} which does not rely on Abe's extension of the Langlands correspondence for overconvergent $F$-isocrystals.
\begin{proof}[Proof of Theorem \ref{strong-independence}] The $\mathbb Q$-compatible system $\{H^1(A)_l|l\in|\mathbb Q|_q\}$ is also pure, as we noted in Example \ref{7.3}, it is the same as the one we have considered in the introduction by Lemma \ref{d-module_comparison}, and it is semi-simple by Theorem \ref{semisimple}. Switching from $U$ to $U^{(n)}$, and using Lemma \ref{reduction_lemma}, we may assume that the base point $x$ has degree one. By taking a finite \'etale cover of $U$ we may assume that all $G_{\lambda}$ and $G_{\lambda}^{geo}$ are connected, by arguing as above. Then in particular $\textrm{End}_L(A_L)=\textrm{End}_{\overline L}(A_L)$. By the Wedderburn theorem, the semi-simple algebra $\textrm{End}_L(A_L)\otimes\mathbb Q$  splits over some number field $E$. Switching from $U$ to $U^{(m)}$ for some sufficiently divisible $m$, and using Lemma \ref{reduction_lemma}, we may assume that the base point $x$ has degree one and $E/\mathbb Q$ is $1$-admissible. Choose an isomorphism with a direct sum of matrix algebras:
$$\textrm{End}_L(A_L)\cong\prod_{j=1}^rM_{n_j}(E).$$
For every $\lambda\in|E|_q$ this induces a decomposition:
$$\mathcal F_{r(\lambda)}\otimes_{\mathbb Q_{r(\lambda)}}E_{\lambda}\cong
\bigoplus_{j=1}^r\mathcal F_{j,\lambda}^{\oplus n_j}.$$
By Theorems \ref{isogeny} and \ref{semisimple} the $\mathcal F_{j,\lambda}$ are absolutely irreducible and pairwise inequivalent. For fixed $j$ and varying $\lambda$ they form a pure $E$-compatible system over $U$. Now we can conclude the argument similarly to the proof of Theorem \ref{chin_independence}. 
\end{proof}

\end{document}